\documentclass[12pt]{amsart} 
\usepackage{lmodern}
\usepackage[utf8]{inputenc}
\usepackage{latexcolors}  
\usepackage{amsmath, amsthm, amssymb, hyperref}
\usepackage[margin=1in]{geometry}
\usepackage{enumerate}
\usepackage[shortlabels]{enumitem}
\usepackage{mathrsfs}
\usepackage{mathtools}

\usepackage{pifont}
\usepackage{multicol}
\usepackage{tikz-cd}
\usetikzlibrary{backgrounds}
\usetikzlibrary{decorations.pathmorphing}
\usepackage{graphicx, float} 
\usepackage[capitalise, noabbrev]{cleveref}
\usepackage{venndiagram}
\usepackage{tikz-cd}
\usepackage[normalem]{ulem}
\usepackage{verbatim}
\usepackage{caption}
\usepackage{subcaption}

\usetikzlibrary{shapes.geometric, calc}
\usetikzlibrary{3d,arrows.meta,positioning}
\hypersetup{colorlinks, linkcolor=red, citecolor=red}

\newcommand{\btimes}{\mathbin{\rotatebox[origin=c]{90}{$\ltimes$}}}
\newcommand{\N}{\mathbb{N}}

\newcommand{\Q}{\mathbb{Q}}
\newcommand{\R}{\mathbb{R}}
\newcommand{\K}{\mathbb{K}}

\renewcommand{\P}{\mathcal{P}}
\newcommand{\st}{\colon}
\newcommand{\Z}{\mathbb{Z}}
\newcommand{\PP}{\mathbb{P}}

\newcommand{\F}{\mathcal{F}}
\newcommand{\M}{\mathcal{M}}
\newcommand{\MS}{\mathcal{MS}}
\newcommand{\I}{\mathcal{I}}

\newcommand{\grobner}{Gröbner }
\newcommand{\mobius}{Möbius }
\newcommand{\ba}{\mathbf{a}}
\newcommand{\be}{\mathbf{e}}
\newcommand{\by}{\mathbf{y}}
\newcommand{\bb}{\mathbf{b}}
\newcommand{\bc}{\mathbf{c}}
\newcommand{\bx}{\mathbf{x}}
\newcommand{\bw}{\mathbf{w}}
\newcommand{\bv}{\mathbf{v}}

\newcommand{\bp}{\mathbf{p}}
\newcommand{\bz}{\mathbf{z}}
\newcommand{\bm}{\mathbf{m}}
\newcommand{\bq}{\mathbf{q}}

\renewcommand{\tilde}{}

\newcommand{\tuple}[1]{\langle#1 \rangle}

\newcommand{\qand}{\quad \mbox{and} \quad}

\newcommand{\qfor}{\quad \mbox{for} \quad}

\newcommand{\qforall}{\quad \mbox{for all} \quad}

\newcommand{\qotherwise}{\quad \mbox{otherwise}}
\newcommand{\B}{\mathcal{B}}
\newcommand{\init}{\mathrm{in}}
\newcommand{\rev}{\mathrm{rev}}

\DeclareMathOperator{\diag}{diag}

\newtheorem{theorem}{Theorem}[section]

\newtheorem{proposition}[theorem]{Proposition}
\newtheorem{corollary}[theorem]{Corollary}
\newtheorem{lemma}[theorem]{Lemma}
\theoremstyle{definition}
\newtheorem{definition}[theorem]{Definition}
\newtheorem{example}[theorem]{Example}
\newtheorem{remark}[theorem]{Remark}
\newtheorem{question}[theorem]{Question}

\newtheorem*{Acknowledgments}{Acknowledgments}

\DeclareMathOperator{\aff}{aff}
\DeclareMathOperator{\initial}{in}
\DeclareMathOperator{\rank}{rank}
\DeclareMathOperator{\conv}{conv}

\DeclareMathOperator{\im}{im}
\DeclareMathOperator{\free}{free}
\DeclareMathOperator{\supp}{supp}
\DeclareMathOperator{\sign}{sign}
\DeclareMathOperator{\partition}{part}
\DeclareMathOperator{\HF}{HF}
\DeclareMathOperator{\HS}{HS}

\usepackage{todonotes}
\reversemarginpar

\title{Symmetric (co)homology polytopes}

\date{\today}

\begin{document}

\author[T. Donzelmann]{Torben Donzelmann}
\address[T. Donzelmann]
{Universität Osnabrück, Albrechtstra\ss e 28a, 49076 Osnabr\"uck, Germany}
\email{torben.donzelmann@uos.de}

\author[T. Holleben]{Thiago Holleben}
\address[T. Holleben]
{Department of Mathematics \& Statistics,
Dalhousie University,
6297 Castine Way,
PO BOX 15000,
Halifax, NS,
Canada B3H 4R2}
\email{hollebenthiago@dal.ca}

\author[M.~Juhnke]{Martina Juhnke}
\address[M.~Juhnke]
{Universität Osnabrück, Albrechtstra\ss e 28a, 49076 Osnabr\"uck, Germany}
\email{martina.juhnke@uni-osnabrueck.de}

\begin{abstract}
Symmetric edge polytopes are a recent and well-studied family of centrally symmetric polytopes arising from graphs.
In this paper, we introduce a generalization of this family to arbitrary simplicial complexes. 
We show how topological properties of a simplicial complex can be translated into geometric properties of such polytopes,
and vice versa.  
We study the integer decomposition property, facets and reflexivity of these polytopes. 
Using  Gröbner basis techniques, we  obtain a (not necessarily unimodular) triangulation of these polytopes.
Due to the tools we use, most of our results hold in the more general setting of arbitrary centrally symmetric polytopes.
\end{abstract}

\maketitle

\section{Introduction}
When modeling a combinatorial object based on an underlying graph $G$, it is often the case that the incidence matrix of $G$ is a valuable tool for understanding the corresponding object. A prime example of this is the sandpile group~\cite[Chapter 4]{K2019} of a graph $G$, where the size of the group can be computed via the incidence matrix using the well-known Matrix Tree theorem~\cite{B1974}.

From a topological point of view, interpreting a graph $G$ as a $1$-dimensional simplicial complex, the incidence matrix of $G$ is just its first boundary map in homology. Having this connection in mind, in 2010 Duval, Klivans and Martin~\cite{DKM2009} introduced the (co)critical groups of an arbitrary simplicial complex $\Delta$, with the goal of generalizing the Matrix Tree theorem to higher dimensional simplicial complexes.

From a different perspective, another object defined in terms of the incidence matrix of a graph $G$ that has received a lot of attention recently is the so-called \emph{symmetric edge polytope} $\P_G$~ (see e.g., \cite{CD,CRV2025,DDM2022,HJM2019,KT2023,HT-Matching}) of $G$, defined as
\[
\P_G=\conv\left(\pm\mathbf{v}~:~ \mathbf{v} \text{ is a column of }\partial_1\right),
\]
where $\partial_1$ denotes the incidence matrix of $G$. Symmetric edge polytopes are of importance in various areas of Mathematics, as e.g., in the study of finite metric spaces and computational phylogenetics, and also in Physics through their connection to the Kuramoto model for interacting oscillators (see \cite{DDM2022} for more details). 
Many combinatorial properties of symmetric edge polytopes have been shown to not depend on the underlying graph explicitly but just on the graphic matroid it represents. Motivated by this, symmetric edge polytopes have been generalized to regular matroids in \cite{DJK2024} (see also \cite{DHO2024} for more results on these). 

In this paper, we aim at a more topological generalization of symmetric edge polytopes. 
 Namely, interpreting a graph and its incidence matrix as a $1$-dimensional simplicial complex and its first boundary map in homology, respectively, (see also~\cite{DKM2009}), generalize symmetric edge polytopes to arbitrary simplicial complexes as follows: Given a $d$-dimensional simplicial complex $\Delta$ with top boundary map (in homology) $\partial_d$, we call
 \begin{equation}\label{eq:SHP}
\P_\Delta\coloneqq\conv[\partial_d|-\partial_d] \qquad \text{ and } \P^\Delta\coloneqq \conv[\partial^\top_d|-\partial^\top_d] 
\end{equation}
  the \emph{symmetric homology polytope} and the \emph{symmetric cohomology polytope}  of $\Delta$, respectively, where, for a matrix $A$,  the notation $\conv[A|-A]$ means that we take the convex hull over the columns of $A$ and $-A$.
  The goal of this paper is to initiate the study of these two new classes of polytopes and to motivate further research in them. We will see that sometimes they behave very similar to their classical counterparts, but sometimes they also show a different behavior.

  More precisely, after having computed basic properties, as the dimension and their number of vertices, our first main result is the precise description of the faces of a symmetric homology polytope. From this, we, in particular, derive a characterization of their facets via labelings of the facets of $\Delta$ and spanning forests (see \Cref{t:facets}). This result generalizes the facet description of the symmetric edge polytope of a graph in a very natural way (see \cite[Theorem 3.1]{HJM2019}). Moreover, \Cref{t:facets} allows us to compute the number of facets of $\P_\Delta$ explicitly if the underlying simplicial complex is a connected closed orientable pseudomanifold (see \Cref{c:cycle_facets}). In the $1$-dimensional situation, the obtained formula specializes to the known formula of the number of facets of the symmetric edge polytope of a cycle and hence provides another parallel to symmetric edge polytopes of graphs (cf., \cite[Proposition 4.3]{DDM2022}). 

Interestingly, another similarity, which is a dissimilarity at the same time, is given by the Gr\"obner basis of the toric ideal of these polytopes: In \Cref{t:groebner_homology}, we manage to compute a Gr\"obner basis for the toric ideals of \emph{any} centrally symmetric lattice polytope. On the one hand, compared to the Gr\"obner basis from \cite[Proposition 3.8]{HJM2019}, our Gr\"obner basis only requires one completely new type of binomials and thus is a nice generalization. On the other hand, while toric ideals of symmetric edge polytopes of graphs admit a squarefree Gr\"obner basis, this is no longer true for symmetric (co)homology polytopes (see \Cref{t:noSquarefreeGroebner}). Consequently, symmetric (co)homology polytopes do not necessarily have the integer decomposition property anymore and we give a necessary  criterion for this property to hold in~\Cref{p:notidp}. In particular, we show that torsion in the $(d-1)$\textsuperscript{st} homology group of a $d$-dimensional simplicial complex is an obstruction for both types of polytopes. This is one example where we are able to translate topological properties of $\Delta$ into geometric properties of the polytopes $\P_\Delta$ and $\P^\Delta$. 
  
The main source for the different behaviors of symmetric (co)homology polytopes compared to their more classical counterparts for graphs lies in the different natures of the boundary matrices. Indeed, while incidence matrices of graphs are always totally unimodular, this is not true for arbitrary boundary matrices. For symmetric edge polytopes, totally unimodularity was crucial in the proof of  various properties of these polytopes, e.g., reflexivity (see ~\cite[Proposition 3.2]{MHN2011})). 
Once we start considering boundary maps of higher dimensional simplicial complexes, the situation quickly becomes more complicated. This was also the case in ~\cite{DKM2009}.

We deal with this obstruction arising from non-totally unimodular matrices in two different ways: 

Using the Smith normal form of boundary maps, we investigate when general centrally symmetric polytopes are reflexive. Our main result is a complete characterization of reflexivity for centrally symmetric crosspolytopes (see \Cref{t:reflexiveCriterionCrossPolytope}). This translates into a particular nice characterization of reflexivity in the setting of symmetric homology polytopes (see \Cref{t:reflexivitySymHomPol}), since in this setting the Smith normal form can be interpreted topologically \cite[Theorem 11.4, Theorem 11.5]{M1984}. In particular, similar as for the integer decomposition property, higher torsion in $(d-1)$\textsuperscript{st} homology group is an obstruction to reflexivity. 
These results also resemble the results regarding critical simplicial groups in~\cite{DKM2009,DKM2011,DKM2013,DKM2015}.

Another aspect of our approach is to focus on classes of simplicial complexes whose boundary maps are known to be totally unimodular. 
In this direction, we are able to show that  both, the homology and cohomology polytopes, of orientable pseudomanifolds are unimodularly equivalent to symmetric edge polytopes of special graphs (see~\cref{t:orientablepseudomanifolds,t:orientablepseudomanifoldswithboundary}). 
Using the classification of totally unimodular boundary maps from  \cite{DHK2011}, we show that symmetric cohomology polytopes  can be used to detect combinatorial equivalence of shellable simplicial spheres (see~\cref{c:distinguishspheres}). 

The paper is structured as follows. In~\cref{s:prelim} we introduce the necessary notions and notation that are used throughout the paper, including the definition of the symmetric (co)homology polytope of a simplicial complex. In \Cref{s:basic}, we study basic properties and provide a face description. In~\cref{s:grobner} we a provide a \grobner basis for the toric ideal of symmetric (co)homology polytopes and study the existence of regular unimodular triangulations for these polytopes. In~\cref{s:IDP} we explore the integer decomposition property (IDP, for short) for these polytopes, and give a topological characterization of when these polytopes are combinatorially equivalent to a crosspolytope. In~\cref{s:reflexivity} we provide criteria for understanding when are symmetric homology polytopes reflexive in terms of torsion generators in homology. In~\cref{s:cohomology} we focus on studying symmetric cohomology polytopes for  simplicial complexes of small dimension. In~\cref{s:examples} we explore symmetric (co)homology polytopes of triangulations of manifolds. Finally, in~\cref{s:questions} we gather some examples that show some of the interesting behavior of symmetric homology and cohomology polytopes,  and state questions that follow from our work.

\section{Preliminaries}\label{s:prelim}
\subsection{Simplicial complexes}

A \emph{simplicial complex} $\Delta$ on vertex set $V$ is a collection of subsets of $V$, called \emph{faces}, such that $\sigma \in \Delta$ and $\tau \subset \sigma$ implies $\tau \in \Delta$. Maximal faces of $\Delta$ are called \emph{facets}, and we write $\Delta = \tuple{\sigma_1, \dots, \sigma_s}$ if $\sigma_1, \dots, \sigma_s$ are the facets of $\Delta$. 
The \emph{dimension} of a face $\sigma \in \Delta$ is $\dim \sigma = |\sigma| - 1$, and the \emph{dimension} of $\Delta$ is $\dim \Delta = \max (\dim \sigma \st \sigma \in \Delta)$.  A $1$-dimensional simplicial complex is called a \emph{graph}. Given a $d$-dimensional simplicial complex $\Delta$, the vector $f(\Delta) = (f_0(\Delta), \dots, f_d(\Delta))$, where $f_i$ is the number of $i$-dimensional faces of $\Delta$ is called the \emph{$f$-vector of $\Delta$}.

If every facet of $\Delta$ has the same dimension, $\Delta$ is called \emph{pure}. The $0$-dimensional faces of $\Delta$ are called \emph{vertices}, the $1$-dimensional faces are called \emph{edges}, $j$-dimensional faces will also be called  \emph{$j$-faces} and for a $d$-dimensional pure simplicial  complex, the $(d-1)$-dimensional faces are called the \emph{ridges} of $\Delta$. A ridge that is contained in exactly one facet of $\Delta$ is called \emph{free}. 
The \emph{facet-ridge graph} $G(\Delta)$ of $\Delta$ is the graph whose vertices are the facets of $\Delta$ and whose edges correspond to pairs of facets of $\Delta$ that intersect in a ridge.
Ordering the faces of a simplicial complex by inclusion gives rise to a poset, the so-called \emph{face poset} in a natural way. 
Simplicial complexes $\Delta$ and $\Gamma$ are \emph{combinatorially equivalent} if there exists a bijection between their vertex sets that induces rise to a bijection between their face posets.

A simplicial complex $\Delta$ is called a \emph{pseudomanifold} if it satisfies the following properties:

\begin{enumerate}
    \item $\Delta$ is pure;
    \item every ridge of $\Delta$ is contained in at most two facets;
    \item $\Delta$ is \emph{strongly connected}, i.e., for every pair of facets $\sigma, \sigma'$ of $\Delta$, there exists a sequence of facets $\sigma_1, \dots, \sigma_s$ such that $\sigma_1 = \sigma$, $\sigma_s = \sigma'$ and $\sigma_i \cap \sigma_{i + 1}$ is a ridge of $\Delta$ for every $i$.
\end{enumerate}
Note that, by definition, a pseudomanifold $\Delta$ is connected. 
The \emph{boundary} of a pseudomanifold $\Delta$ is the simplicial complex $\partial \Delta = \tuple{\sigma_1, \dots, \sigma_t}$, where $\{\sigma_1, \dots, \sigma_t\}$ is the set of ridges of $\Delta$ contained in exactly one facet of $\Delta$. If $\partial \Delta = \emptyset$, $\Delta$ is called a \emph{closed pseudomanifold} or a \emph{pseudomanifold without boundary}. Note that for a closed pseudomanifold $\Delta$, the edges of the facet-ridge graph are in $1$-$1$--correspondence to the ridges of $\Delta$.

\begin{example}[\emph{A triangulation of the real projective plane}]\label{ex:projectiveplane}
    The $2$-dimensional real \emph{projective space/plane} 
    $\R\PP^2$ is an important space in algebraic topology, as it serves as the simplest example satisfying several interesting topological properties. Throughout this paper, we will often use the following triangulation $\Delta_{\R\PP^2}$ of $\R\PP^2$:  
    $$
        \Delta_{\R\PP^2} = \tuple{125, 126,134,136,145,234,235,246,356,456}.
    $$
    where we use the short notation $125$ for $\{1,2,5\}$, and similarly for other sets. 
    One can easily check that $\Delta_{\R\PP^2}$ is a pseudomanifold without boundary. \Cref{fig:projectivePlane} shows the triangulation $\Delta_{\R\PP^2}$ together with its facet-ridge graph $G(\Delta_{\R\PP^2})$. 
\begin{figure}
    \centering
    \begin{center}
\tikzset{every picture/.style={line width=0.75pt}} 

\begin{tikzpicture}[x=0.5pt,y=0.5pt,yscale=-1,xscale=1]

\draw  [fill={rgb, 255:red, 155; green, 155; blue, 155 }  ,fill opacity=1 ] (191.33,318.48) -- (70.12,249.62) -- (68.04,110.89) -- (187.17,41.02) -- (308.38,109.88) -- (310.46,248.61) -- cycle ;
\draw    (191,225.5) -- (191.33,318.48) ;
\draw    (191,225.5) -- (70.12,249.62) ;
\draw    (191,225.5) -- (310.46,248.61) ;
\draw    (237,144.5) -- (191,225.5) ;
\draw    (237,144.5) -- (308.38,109.88) ;
\draw    (187.17,41.02) -- (237,144.5) ;
\draw    (145,148.5) -- (187.17,41.02) ;
\draw    (68.04,110.89) -- (145,148.5) ;
\draw    (145,148.5) -- (237,144.5) ;
\draw    (70.12,249.62) -- (145,148.5) ;
\draw    (191,225.5) -- (145,148.5) ;
\draw    (237,144.5) -- (310.46,248.61) ;
\draw    (551,289.5) -- (473,289.5) ;
\draw [shift={(473,289.5)}, rotate = 180] [color={rgb, 255:red, 0; green, 0; blue, 0 }  ][fill={rgb, 255:red, 0; green, 0; blue, 0 }  ][line width=0.75]      (0, 0) circle [x radius= 3.35, y radius= 3.35]   ;
\draw [shift={(551,289.5)}, rotate = 180] [color={rgb, 255:red, 0; green, 0; blue, 0 }  ][fill={rgb, 255:red, 0; green, 0; blue, 0 }  ][line width=0.75]      (0, 0) circle [x radius= 3.35, y radius= 3.35]   ;
\draw    (587,231.5) -- (551,289.5) ;
\draw [shift={(551,289.5)}, rotate = 121.83] [color={rgb, 255:red, 0; green, 0; blue, 0 }  ][fill={rgb, 255:red, 0; green, 0; blue, 0 }  ][line width=0.75]      (0, 0) circle [x radius= 3.35, y radius= 3.35]   ;
\draw [shift={(587,231.5)}, rotate = 121.83] [color={rgb, 255:red, 0; green, 0; blue, 0 }  ][fill={rgb, 255:red, 0; green, 0; blue, 0 }  ][line width=0.75]      (0, 0) circle [x radius= 3.35, y radius= 3.35]   ;
\draw    (520.25,204.75) -- (587,231.5) ;
\draw [shift={(587,231.5)}, rotate = 21.84] [color={rgb, 255:red, 0; green, 0; blue, 0 }  ][fill={rgb, 255:red, 0; green, 0; blue, 0 }  ][line width=0.75]      (0, 0) circle [x radius= 3.35, y radius= 3.35]   ;
\draw [shift={(520.25,204.75)}, rotate = 21.84] [color={rgb, 255:red, 0; green, 0; blue, 0 }  ][fill={rgb, 255:red, 0; green, 0; blue, 0 }  ][line width=0.75]      (0, 0) circle [x radius= 3.35, y radius= 3.35]   ;
\draw    (520.25,204.75) -- (460,242.5) ;
\draw [shift={(460,242.5)}, rotate = 147.93] [color={rgb, 255:red, 0; green, 0; blue, 0 }  ][fill={rgb, 255:red, 0; green, 0; blue, 0 }  ][line width=0.75]      (0, 0) circle [x radius= 3.35, y radius= 3.35]   ;
\draw [shift={(520.25,204.75)}, rotate = 147.93] [color={rgb, 255:red, 0; green, 0; blue, 0 }  ][fill={rgb, 255:red, 0; green, 0; blue, 0 }  ][line width=0.75]      (0, 0) circle [x radius= 3.35, y radius= 3.35]   ;
\draw    (473,289.5) -- (460,242.5) ;
\draw [shift={(460,242.5)}, rotate = 254.54] [color={rgb, 255:red, 0; green, 0; blue, 0 }  ][fill={rgb, 255:red, 0; green, 0; blue, 0 }  ][line width=0.75]      (0, 0) circle [x radius= 3.35, y radius= 3.35]   ;
\draw [shift={(473,289.5)}, rotate = 254.54] [color={rgb, 255:red, 0; green, 0; blue, 0 }  ][fill={rgb, 255:red, 0; green, 0; blue, 0 }  ][line width=0.75]      (0, 0) circle [x radius= 3.35, y radius= 3.35]   ;
\draw    (616,198.5) -- (587,231.5) ;
\draw [shift={(587,231.5)}, rotate = 131.31] [color={rgb, 255:red, 0; green, 0; blue, 0 }  ][fill={rgb, 255:red, 0; green, 0; blue, 0 }  ][line width=0.75]      (0, 0) circle [x radius= 3.35, y radius= 3.35]   ;
\draw [shift={(616,198.5)}, rotate = 131.31] [color={rgb, 255:red, 0; green, 0; blue, 0 }  ][fill={rgb, 255:red, 0; green, 0; blue, 0 }  ][line width=0.75]      (0, 0) circle [x radius= 3.35, y radius= 3.35]   ;
\draw    (526,138.5) -- (520.25,204.75) ;
\draw [shift={(520.25,204.75)}, rotate = 94.96] [color={rgb, 255:red, 0; green, 0; blue, 0 }  ][fill={rgb, 255:red, 0; green, 0; blue, 0 }  ][line width=0.75]      (0, 0) circle [x radius= 3.35, y radius= 3.35]   ;
\draw [shift={(526,138.5)}, rotate = 94.96] [color={rgb, 255:red, 0; green, 0; blue, 0 }  ][fill={rgb, 255:red, 0; green, 0; blue, 0 }  ][line width=0.75]      (0, 0) circle [x radius= 3.35, y radius= 3.35]   ;
\draw    (616,198.5) -- (582,128.5) ;
\draw [shift={(582,128.5)}, rotate = 244.09] [color={rgb, 255:red, 0; green, 0; blue, 0 }  ][fill={rgb, 255:red, 0; green, 0; blue, 0 }  ][line width=0.75]      (0, 0) circle [x radius= 3.35, y radius= 3.35]   ;
\draw [shift={(616,198.5)}, rotate = 244.09] [color={rgb, 255:red, 0; green, 0; blue, 0 }  ][fill={rgb, 255:red, 0; green, 0; blue, 0 }  ][line width=0.75]      (0, 0) circle [x radius= 3.35, y radius= 3.35]   ;
\draw    (582,128.5) -- (526,138.5) ;
\draw [shift={(526,138.5)}, rotate = 169.88] [color={rgb, 255:red, 0; green, 0; blue, 0 }  ][fill={rgb, 255:red, 0; green, 0; blue, 0 }  ][line width=0.75]      (0, 0) circle [x radius= 3.35, y radius= 3.35]   ;
\draw [shift={(582,128.5)}, rotate = 169.88] [color={rgb, 255:red, 0; green, 0; blue, 0 }  ][fill={rgb, 255:red, 0; green, 0; blue, 0 }  ][line width=0.75]      (0, 0) circle [x radius= 3.35, y radius= 3.35]   ;
\draw    (526,138.5) -- (460,134.5) ;
\draw [shift={(460,134.5)}, rotate = 183.47] [color={rgb, 255:red, 0; green, 0; blue, 0 }  ][fill={rgb, 255:red, 0; green, 0; blue, 0 }  ][line width=0.75]      (0, 0) circle [x radius= 3.35, y radius= 3.35]   ;
\draw [shift={(526,138.5)}, rotate = 183.47] [color={rgb, 255:red, 0; green, 0; blue, 0 }  ][fill={rgb, 255:red, 0; green, 0; blue, 0 }  ][line width=0.75]      (0, 0) circle [x radius= 3.35, y radius= 3.35]   ;
\draw    (460,134.5) -- (424,192.5) ;
\draw [shift={(424,192.5)}, rotate = 121.83] [color={rgb, 255:red, 0; green, 0; blue, 0 }  ][fill={rgb, 255:red, 0; green, 0; blue, 0 }  ][line width=0.75]      (0, 0) circle [x radius= 3.35, y radius= 3.35]   ;
\draw [shift={(460,134.5)}, rotate = 121.83] [color={rgb, 255:red, 0; green, 0; blue, 0 }  ][fill={rgb, 255:red, 0; green, 0; blue, 0 }  ][line width=0.75]      (0, 0) circle [x radius= 3.35, y radius= 3.35]   ;
\draw    (460,242.5) -- (424,192.5) ;
\draw [shift={(424,192.5)}, rotate = 234.25] [color={rgb, 255:red, 0; green, 0; blue, 0 }  ][fill={rgb, 255:red, 0; green, 0; blue, 0 }  ][line width=0.75]      (0, 0) circle [x radius= 3.35, y radius= 3.35]   ;
\draw [shift={(460,242.5)}, rotate = 234.25] [color={rgb, 255:red, 0; green, 0; blue, 0 }  ][fill={rgb, 255:red, 0; green, 0; blue, 0 }  ][line width=0.75]      (0, 0) circle [x radius= 3.35, y radius= 3.35]   ;
\draw    (460,134.5) .. controls (278,291.5) and (507,369.5) .. (551,289.5) ;
\draw    (424,192.5) .. controls (414,21.5) and (678,51.5) .. (616,198.5) ;
\draw    (582,128.5) .. controls (755,171.5) and (586,394.5) .. (473,289.5) ;

\draw (182,20.4) node [anchor=north west][inner sep=0.75pt]    {$1$};
\draw (317,94.4) node [anchor=north west][inner sep=0.75pt]    {$6$};
\draw (321,242.4) node [anchor=north west][inner sep=0.75pt]    {$4$};
\draw (237,117.4) node [anchor=north west][inner sep=0.75pt]    {$2$};
\draw (137,118.4) node [anchor=north west][inner sep=0.75pt]    {$5$};
\draw (199,233.4) node [anchor=north west][inner sep=0.75pt]    {$3$};
\draw (41,92.4) node [anchor=north west][inner sep=0.75pt]    {$4$};
\draw (44,238.4) node [anchor=north west][inner sep=0.75pt]    {$6$};
\draw (183,323.4) node [anchor=north west][inner sep=0.75pt]    {$1$};
\draw (425,284.4) node [anchor=north west][inner sep=0.75pt]    {$136$};
\draw (562,275.4) node [anchor=north west][inner sep=0.75pt]    {$134$};
\draw (590,232.4) node [anchor=north west][inner sep=0.75pt]    {$234$};
\draw (415,237.4) node [anchor=north west][inner sep=0.75pt]    {$356$};
\draw (500,217.4) node [anchor=north west][inner sep=0.75pt]    {$235$};
\draw (434,181.4) node [anchor=north west][inner sep=0.75pt]    {$456$};
\draw (453,106.4) node [anchor=north west][inner sep=0.75pt]    {$145$};
\draw (508,110.4) node [anchor=north west][inner sep=0.75pt]    {$125$};
\draw (558,100.4) node [anchor=north west][inner sep=0.75pt]    {$126$};
\draw (610,205.4) node [anchor=north west][inner sep=0.75pt]    {$246$};

\end{tikzpicture}

\end{center}
    \caption{The triangulation $\Delta_{\R\PP^2}$ of the real projective plane (\emph{left}) and its facet-ridge graph $G(\Delta_{\R\PP^2})$ (\emph{right}).}
    \label{fig:projectivePlane}
\end{figure}
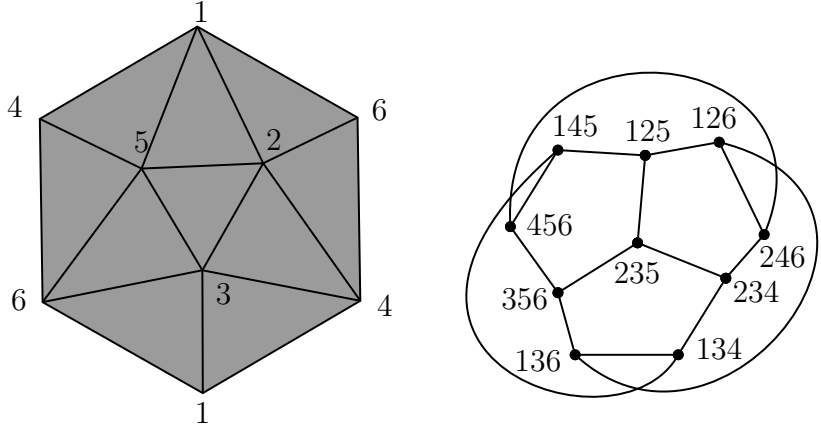
\end{example}

\subsection{Simplicial (co)homology}
We now recall basic facts on simplicial (co)homology, as it will play a very important role in the next sections. We refer the reader to~\cite{M1984} for more details.

Let $\Delta$ be a simplicial complex on vertex set $[n]\coloneqq \{1,\ldots,n\}$, $\K$ be a field (or $\Z$) and $C_j(\Delta;\K)$ be an $f_j(\Delta)$-dimensional $\K$-vector space (or a free abelian group) with basis $\{\sigma \st \sigma \in \Delta \mbox{ and } \dim \sigma = j\}$. 
The $\K$-linear map $\partial_j: C_j(\Delta;\K) \to C_{j-1}(\Delta;\K)$ defined on a $j$-face $\sigma = \{i_1, \dots, i_{j+1}\}$ with $  i_1 < \dots < i_{j+1}$ by  
\[
    \partial_j(\sigma) = \sum_{k = 1}^{j+1} (-1)^{k+1} (\sigma \setminus \{i_k\})
\]
is called the \emph{$j$\textsuperscript{th} boundary map} (in homology) of $\Delta$. 
Note that the matrix that represents the boundary map $\partial_j$ is a $\{0,1,-1\}$-matrix whose columns have $j+1$ nonzero entries, that alternate in sign. Abusing notation slightly, we will use $\partial_j$ for both, the linear map and the corresponding matrix (with respect to the standard basis).

\begin{remark}\label{r:incidenceboundary}
    Given a graph $G$ on vertex set $[n]$, the boundary map $\partial_1$ is often called the \emph{incidence matrix} of $G$.
\end{remark}    

The maps $\partial_j$ satisfy the equality $\partial_j \circ\partial_{j+1} = 0$ for every $j$. In particular, we have the inclusion $\im \partial_{j + 1} \subset \ker \partial_{j}$. The \emph{$j$\textsuperscript{th} homology group of $\Delta$} with coefficients in $\K$ is defined as the quotient
\[
    \tilde H_j(\Delta; \K) \coloneqq \frac{\ker_\K \partial_j}{\im_\K \partial_{j + 1}},
\]
where the subscript $\K$ indicates that we take the elements of the kernel/image whose entries belong to $\K$. We will use this type of notation also later on.
The fundamental theorem of finitely generated abelian groups~\cite[Theorem 3, Section 5.2]{DummitFoote} says that for every $j$, we have an isomorphism $\tilde H_j(\Delta;\Z) \cong F \oplus T$, where $F$ is a free abelian group, and $T$ is a finite abelian group.
If $T$ has a subgroup isomorphic to $\Z_p\coloneqq \Z/p\Z$ (where $p$ is a positive integer), we say that $\tilde H_j(\Delta;\Z)$ has \emph{$p$-torsion}. 
If $\Delta$ is such that $H_j(\Delta;\Z)$ does not have $p$-torsion for any $p$, then $\Delta$ and $H_j(\Delta;\Z)$, respectively, are called \emph{torsion-free}.
Note that, if $\dim \Delta=d$, then $\im \partial_{d + 1} = 0$ and hence $\tilde H_d(\Delta;\Z) = \ker \partial_d$ is a subgroup of $C_d(\Delta;\Z)$ and as such it has to be free.
Similarly, the \emph{$j$\textsuperscript{th} cohomology group of $\Delta$} with coefficients in $\K$ is defined as the quotient
\[
    \tilde H^j (\Delta; \K) \coloneqq \frac{\ker_\K \partial_j^\top}{\im_\K \partial_{j - 1}^\top}.
\]
Elements in $C_j(\Delta;\K)$, $\ker_\K \partial_j$ and $\im_\K \partial_j$ are called \emph{$j$-chains}, \emph{$j$-cycles} and \emph{$j$-boundaries}, respectively, so that nonzero elements in $\tilde H_j(\Delta;\K)$ come from cycles that are not boundaries. 
When we replace $\partial_j$ by $\partial_j^\top$ we just add the prefix \emph{co} to the names just mentioned. In particular, in the setting of cohomology, we call elements in $C_j(\Delta;\K)$ \emph{$j$-cochains}.

\begin{example}[\emph{From torsion to homology cycles: Björner's example}]\label{ex:bjorner}
    Consider the  triangulation $\Delta_{\R\PP^2}$ of $\R\PP^2$  from~\cref{ex:projectiveplane}. Using Sage~\cite{sagemath} we compute 
    $$
        H_0(\Delta_{\R\PP^2}; \Z) = \Z, \quad H_1(\Delta_{\R\PP^2}; \Z) = \Z_2 \qand H_2(\Delta_{\R\PP^2}; \Z) = 0,
    $$
    where the torsion generator of $H_1(\Delta; \Z)$ is given by 
\begin{equation}\label{eq:torsionGen}
        - [125] - [126] + [134] + [136] + [145] - [234] - [235] - [246] + [356] - [456].
\end{equation}
Here, the notation $[ijk]$ (where $1\leq i<j<k\leq 6$) is used to refer to the basis element in $C_2(\Delta_{\R\PP^2};\K)$ corresponding to the face $ijk$. 

    Let $\B$ be  the simplicial complex $\B$ that is the union of  $\Delta_{\R\PP^2}$ with the additional facet $123$.  This simplicial complex is also often referred to as \emph{Björner's example}, appearing in Hachimori's library of simplicial complexes~\cite{Hlib}. Computation shows that 
    $$
        H_0(\B; \Z) = \Z, \quad H_1(\B; \Z) = 0 \qand H_2(\B; \Z) = \Z.
    $$
    We want to emphasize that adding the facet $123$ has destroyed the torsion on one side, and created non-trivial $2$-dimensional homology, on the other side. 
    The generator for $H_2(\B; \Z)$ is given by 
    $$
        2[123] - [125] - [126] + [134] + [136] + [145] - [234] - [235] - [246] + [356] - [456].
    $$
    To the best of our knowledge,  $\B$ is the smallest simplicial complex  (in terms of the number of vertices) having a homology cycle over $\Z$ with coprime coefficients such that there exists a coefficient different from $-1$, $0$ and $1$. This property will be important in later sections.
\end{example}

Given a simplicial complex $\Delta$ and a subcomplex $\Gamma$, the \emph{$j$\textsuperscript{th} relative homology group} of the pair $(\Delta, \Gamma)$ with coefficients in $\K$ is defined as the quotient 
\[
    \tilde H_j(\Delta, \Gamma; \K) = \frac{\ker_\K \partial_j^\ast}{\im_\K \partial_{j+1}^\ast},
\]
where $\partial_j^\ast: \frac{C_j(\Delta)}{C_j(\Gamma)} \to \frac{C_{j-1}(\Delta)}{C_{j-1}(\Gamma)}$ is the map induced by taking the quotient. (Note that the boundary map of $\Gamma$ is just the restriction of the boundary map of $\Delta$.) 

A $d$-dimensional closed pseudomanifold $\Delta$ is called \emph{orientable} if $H_d(\Delta; \Z) =\Z$. Otherwise,  $\Delta$ is called \emph{non-orientable}. Similarly, a $d$-dimensional pseudomanifold $\Delta$ with boundary $\partial \Delta$ is called \emph{orientable} if $H_d(\Delta,\partial\Delta;\Z)=\Z$.

Finally, a matrix $M$ is said to be \emph{totally unimodular} if every minor of $M$ is either $-1, 0$ or $1$. 
Totally unimodular matrices turn out to play an important role in the study of linear programs and polytopes.
The question of when a boundary matrix of a simplicial complex $\Delta$ is totally unimodular was answered in~\cite{DHK2011}. 
We now state a special case of this result, that will be useful for our purposes.

\begin{theorem}{\cite[Theorem 4.1]{DHK2011}}\label{t:tuboundary}
    Let  $\Delta$ be a $d$-dimensional  orientable pseudomanifold. Then $\partial_d$ is totally unimodular.
\end{theorem}

We note that~\cite[Theorem 4.1]{DHK2011} is originally stated for compact orientable manifolds  but it is easy to verify that the exact same proof works for orientable pseudomanifolds.

\begin{example}[\emph{A non-totally unimodular boundary map}]
    Let $\MS$  be the triangulation of the Möbius strip, shown in \Cref{fig:enter-label}. $\MS$ is a $2$-dimensional non-orientable pseudomanifold with boundary whose second boundary map $\partial_2$ is given by the following matrix:
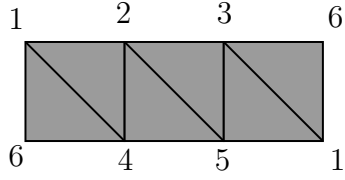
\begin{figure}
    
    \begin{center}

\tikzset{every picture/.style={line width=0.75pt}} 

\begin{tikzpicture}[x=0.75pt,y=0.75pt,yscale=-1,xscale=1]

\draw  [fill={rgb, 255:red, 155; green, 155; blue, 155 }  ,fill opacity=1 ] (100,103) -- (150,103) -- (150,153) -- (100,153) -- cycle ;
\draw  [fill={rgb, 255:red, 155; green, 155; blue, 155 }  ,fill opacity=1 ] (150,103) -- (200,103) -- (200,153) -- (150,153) -- cycle ;
\draw  [fill={rgb, 255:red, 155; green, 155; blue, 155 }  ,fill opacity=1 ] (200,103) -- (250,103) -- (250,153) -- (200,153) -- cycle ;
\draw    (100,103) -- (150,153) ;
\draw    (150,103) -- (200,153) ;
\draw    (200,103) -- (250,153) ;

\draw (90,85) node [anchor=north west][inner sep=0.75pt]   [align=left] {$\displaystyle 1$};
\draw (144,82) node [anchor=north west][inner sep=0.75pt]   [align=left] {$\displaystyle 2$};
\draw (195,82) node [anchor=north west][inner sep=0.75pt]   [align=left] {$\displaystyle 3$};
\draw (145,156) node [anchor=north west][inner sep=0.75pt]   [align=left] {$\displaystyle 4$};
\draw (194,156) node [anchor=north west][inner sep=0.75pt]   [align=left] {$\displaystyle 5$};
\draw (252,156) node [anchor=north west][inner sep=0.75pt]   [align=left] {$\displaystyle 1$};
\draw (90,154) node [anchor=north west][inner sep=0.75pt]   [align=left] {$\displaystyle 6$};
\draw (251,84) node [anchor=north west][inner sep=0.75pt]   [align=left] {$\displaystyle 6$};
\end{tikzpicture}
    \end{center}
    \caption{The triangulation $\MS$ of the M\"obius strip.}
    \label{fig:enter-label}
\end{figure}
$$
    \partial_2 = \left(\begin{array}{rrrrrr}
1 & 0 & 0 & 0 & 0 & 0 \\
0 & 1 & 1 & 0 & 0 & 0 \\
-1 & 0 & 0 & 1 & 0 & 0 \\
0 & -1 & 0 & 0 & 0 & 0 \\
0 & 0 & -1 & -1 & 0 & 0 \\
0 & 0 & 0 & 0 & 1 & 0 \\
1 & 0 & 0 & 0 & 0 & 1 \\
0 & 0 & 0 & 0 & -1 & -1 \\
0 & 1 & 0 & 0 & 1 & 0 \\
0 & 0 & 1 & 0 & 0 & 0 \\
0 & 0 & 0 & 0 & 0 & 1 \\
0 & 0 & 0 & 1 & 0 & 0
\end{array}\right),
$$
The rows and columns of $\partial_2$ are ordered according to the lexicographic order of the edges and $2$-faces of $\MS$, respectively. 
Deleting the rows corresponding to the boundary of $\MS$ (i.e., rows with a unique nonzero entry), we get a $(6\times 6)$-matrix whose determinant has absolute value equal to $2$. Consequently, $\partial_2$ is not totally unimodular. 
\end{example}

\subsection{Smith normal form}

The computation of homology groups with coefficients in a field $\K$ can be translated to a standard linear algebra problem. When working over $\Z$ however, standard linear algebra tools are no longer available, and so one often has to rely on results on free modules over principal ideal domains such as $\Z$.

Given a matrix $A \in \Z^{n \times m}$ of rank $r$, one version of the fundamental theorem of abelian groups says that there exist weakly unimodular matrices (i.e., matrices whose determinant has absolute value $1$) $S \in \Z^{n \times n}$ and $T \in \Z^{m \times m}$ such that 
$$
    S^{-1}AT^{-1} = \begin{pmatrix}
        \alpha_1           & \cdots & 0 &   0 & \cdots & 0 \\
     0                 & \ddots  & 0 &   \vdots &  & \vdots  \\
        0                 &\cdots   & \alpha_r & 0  & \cdots & 0  \\
        0 &   \cdots &  0 & 0 & \cdots & 0  \\
        \vdots &  & \vdots & \vdots & &\vdots\\
        0  & \cdots & 0& 0 & \cdots &0                          
    \end{pmatrix}  
    = \colon D,
$$

where $\alpha_i$ divides $\alpha_{i + 1}$ for $1 \leq i \leq r-1$ (see for example~\cite[Section 2.2]{S2016}). The sequence $\alpha_1, \dots, \alpha_r$ is unique up to sign, and the matrix $D$ is called the \emph{Smith normal form} of $A$ (\emph{SNF} for short), and we also write $\diag(\alpha_1, \dots, \alpha_r)$ for $D$ (without referencing the size of $D$). The elements $\alpha_i$ are called the \emph{elementary divisors} of $A$ and can be computed as
$$
    \alpha_i = \frac{d_i(A)}{d_{i-1}(A)},
$$
where $d_0(A)\coloneqq1$ and $d_i(A)$ is the greatest common divisor of the $(i\times i)$-minors of $A$ for $1\leq i\leq r$. 

The usefulness of the SNF of a matrix $A$ comes from the fact that $\im_\Z A \cong \im_\Z D$, where $\im_\Z A=\{A\mathbf{v}~:~\mathbf{v}\in \Z^m\}$ and similarly, for $D$. 
In other words, the columns of $A$ and $D$ span the same integral lattice. 
This implies for example that for a $d$-dimensional simplicial complex $\Delta$, the elementary divisors of $\partial_d$ encode the rank of $\tilde H_d(\Delta)$, as well as the algebraic data of the torsion of $\tilde H_d(\Delta)$. 
Below we gather the results  we will need about the SNF of boundary maps.
We want to emphasize that similar results can be stated without the topological context for arbitrary integral matrices.

\begin{theorem}{\cite[Theorems 11.4 and 11.5]{M1984}}\label{t:munkressnf}
    Let $\Delta$ be a simplicial complex and let $0\leq i\leq \dim \Delta$.  
    If $\partial_i$ is an $(n\times m)$-matrix of rank $r$ with elementary divisors $\alpha_1, \dots, \alpha_r$, then 
    $$
        \tilde H_{i - 1}(\Delta; \Z) \cong \Z^{\min(n,m) - r} \bigoplus_{j = 1}^{r} \Z_{\alpha_j}.
    $$
\end{theorem}


\begin{example}[\emph{Moore spaces and torsion}]\label{ex:moore}
    Given an abelian group $G$ and an integer $i$, a simplicial complex (or more generally, a CW complex) $\M_G^i$ is called a \emph{Moore space of $G$} if $H_0(\M_G^i; \Z) \cong \Z$,  $H_i(\M_G^i; \Z)\cong G$ and $H_j(\M_G^i; \Z)=0$ for $j\neq i$ (see e.g., \cite[Chapter 2]{Hatcher} for more information on Moore spaces). \Cref{fig:Moore} shows an example of a Moore space $\M_{\Z_3}^1$ for $\Z_3$. Its unique nontrivial homology group for $i > 0$ is $H_1(\M_{\Z_3}^1; \Z) \cong \Z_3$.
\begin{figure}
   \tikzset{every picture/.style={line width=0.75pt}} 
\begin{center}
\begin{tikzpicture}[x=0.75pt,y=0.75pt,yscale=-1,xscale=1]

\draw  [fill={rgb, 255:red, 155; green, 155; blue, 155 }  ,fill opacity=1 ] (193,114.2) -- (175.64,161.89) -- (131.68,187.27) -- (81.7,178.46) -- (49.07,139.58) -- (49.07,88.82) -- (81.7,49.94) -- (131.68,41.13) -- (175.64,66.51) -- cycle ;
\draw    (131.68,41.13) -- (125.47,93.4) ;
\draw    (81.7,49.94) -- (125.47,93.4) ;
\draw    (175.64,66.51) -- (125.47,93.4) ;
\draw    (175.64,66.51) -- (164.47,92.4) ;
\draw    (125.47,93.4) -- (147.47,110.4) ;
\draw    (81.7,49.94) -- (85.47,89.4) ;
\draw    (92.47,116.4) -- (125.47,93.4) ;
\draw    (92.47,116.4) -- (49.07,88.82) ;
\draw    (49.07,88.82) -- (85.47,89.4) ;
\draw    (85.47,89.4) -- (92.47,116.4) ;
\draw    (125.47,93.4) -- (85.47,89.4) ;
\draw    (92.47,116.4) -- (49.07,139.58) ;
\draw    (147.47,110.4) -- (92.47,116.4) ;
\draw    (147.47,110.4) -- (193,114.2) ;
\draw    (164.47,92.4) -- (193,114.2) ;
\draw    (164.47,92.4) -- (147.47,110.4) ;
\draw    (164.47,92.4) -- (125.47,93.4) ;
\draw    (92.47,116.4) -- (81.7,178.46) ;
\draw    (81.7,178.46) -- (135.47,142.4) ;
\draw    (147.47,110.4) -- (175.64,161.89) ;
\draw    (135.47,142.4) -- (131.68,187.27) ;
\draw    (135.47,142.4) -- (175.64,161.89) ;
\draw    (135.47,142.4) -- (92.47,116.4) ;
\draw    (147.47,110.4) -- (135.47,142.4) ;

\draw (177,47) node [anchor=north west][inner sep=0.75pt]   [align=left] {$\displaystyle 1$};
\draw (133.68,190.27) node [anchor=north west][inner sep=0.75pt]   [align=left] {$\displaystyle 1$};
\draw (40,70) node [anchor=north west][inner sep=0.75pt]   [align=left] {$\displaystyle 1$};
\draw (196,96) node [anchor=north west][inner sep=0.75pt]   [align=left] {$\displaystyle 2$};
\draw (69,181) node [anchor=north west][inner sep=0.75pt]   [align=left] {$\displaystyle 2$};
\draw (71,31) node [anchor=north west][inner sep=0.75pt]   [align=left] {$\displaystyle 2$};
\draw (126,21) node [anchor=north west][inner sep=0.75pt]   [align=left] {$\displaystyle 3$};
\draw (177.64,164.89) node [anchor=north west][inner sep=0.75pt]   [align=left] {$\displaystyle 3$};
\draw (36,136) node [anchor=north west][inner sep=0.75pt]   [align=left] {$\displaystyle 3$};
\draw (130.58,70.26) node [anchor=north west][inner sep=0.75pt]   [align=left] {$\displaystyle 4$};
\draw (143,115) node [anchor=north west][inner sep=0.75pt]   [align=left] {$\displaystyle 5$};
\draw (167,80) node [anchor=north west][inner sep=0.75pt]   [align=left] {$\displaystyle 6$};
\draw (85.58,72.67) node [anchor=north west][inner sep=0.75pt]   [align=left] {$\displaystyle 7$};
\draw (137.47,145.4) node [anchor=north west][inner sep=0.75pt]   [align=left] {$\displaystyle 8$};
\draw (81,119) node [anchor=north west][inner sep=0.75pt]   [align=left] {$\displaystyle 9$};

\end{tikzpicture}
\end{center}
 \caption{A Moore space for $\mathbb{Z}_3$.}
    \label{fig:Moore}
\end{figure}
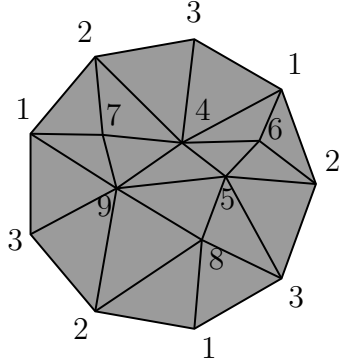

The $f$-vector of $\M_{\Z_3}^1$ is $(1,9,27,19)$ and using Sage~\cite{sagemath} we see that the elementary divisors of the second boundary map $\partial_2$ of $\M_{\Z_3}^1$ are 
$$
    \underbrace{1,\dots, 1}_{\text{$18$ times}}, 3,
$$
as predicted by \Cref{t:munkressnf}.
\end{example}

We will strongly employ the SNF of an integral matrix $A$ when studying properties of $\P_A$, such as reflexivity and the integer decomposition property.
We note that the idea of using special forms from linear algebra to study these properties for centrally symmetric polytopes has appeared before in the literature. 
For example, in~\cite{OH2014}, the authors use the Hermite normal form of a matrix to prove similar results to ours. However, the use of the SNF allows us to show that several results in~\cite{OH2014} hold under weaker assumptions. More precisely, we can replace the condition from \cite{OH2014}  of all maximal minors having  the same absolute value by a weaker condition on the greatest common divisors of the maximal minors. 
Another advantage of using the SNF comes from its natural interpretations in algebraic topology (see \cref{t:munkressnf}).

\subsection{Matroids}


One very useful perspective that we will often take throughout the paper is to look at the linear dependencies of the boundary maps of a simplicial complex.
In order to do so, we will use the language of matroids.
For more details on matroids we refer to \cite{Oxley}.

A \emph{matroid} $M$ is a pair $(E, \mathcal{I})$, where $E$ is a set, called the \emph{ground set} of $M$, and $\mathcal{I}$ is a collection of subsets of $E$ called \emph{independent sets} such that
\begin{enumerate}
    \item if $\sigma \in \mathcal{I}$ and $\tau \subset \sigma$ then $\tau \in \mathcal{I}$,
    \item $\emptyset \in \mathcal{I}$, and
    \item if $\sigma, \tau \in \mathcal{I}$ and $|\sigma| > |\tau|$, then there exists $x \in \sigma \setminus \tau$ such that $\tau \cup\{ x\} \in \mathcal{I}$.
\end{enumerate}
 The maximal independent sets of $M$ are called the \emph{bases} of $M$. A subset $\sigma \not \in \I$ is called a \emph{dependent set}. Minimal dependent sets are called \emph{circuits}.

Given an $(n\times m)$-matrix $A$, we define the matroid $M[A] = ([m], \I)$, where $\sigma \in \I$ if and only if the corresponding collection of columns of $A$ is linearly independent. A matroid $M$ is called \emph{realizable over a field $\K$} if there exists a matrix $A$ with coefficients in $\K$ such that $M$ is isomorphic to $M[A]$, i.e., there is a bijection  between the ground set of $M$ and the ground set of $M[A]$ that preserves independent sets and such that the preimage of an independent set is also an independent set. 
In this case, $A$ is called a \emph{representation} of $M$. When $M$ can be represented by the incidence matrix of a graph,  $M$ is called a \emph{graphic matroid}.
If a matroid $M$ is realizable over every field $\K$, $M$ is called \emph{regular}. 

\begin{example}\label{ex:cycleMatroid}
    Let $E = [n] = \{1,\dots, n\}$ and $\I = 2^{[n]} \setminus\{ [n]\}$. Then $M_n = (E, \I)$ is the graphic matroid of the $n$-cycle $C_n$. This matroid is also called \emph{uniform matroid} on $n$ elements of rank $n-1$.
\end{example}

The following well-known proposition will play a key role in future sections.

\begin{proposition}\cite{Oxley}\label{p:regulargraphical}
    A matroid $M$ is regular if and only if there exists a totally unimodular matrix $A$ that represents $M$. In particular, graphic matroids are regular.
\end{proposition}

Throughout this paper, we will often deal with matroids of the form $M[\partial_d]$ (and $M[\partial_d^\top]$), where $\partial_d$ is the $d$\textsuperscript{th} boundary map of a $d$-dimensional simplicial complex $\Delta$. 
In this case, it is useful to note that independent sets of $M[\partial_d]$ correspond to $d$-dimensional pure subcomplexes of $\Delta$ with trivial top homology, while circuits of $M[\partial_d]$ correspond to minimal $d$-dimensional subcomplexes of $\Delta$ with non-trivial  top homology.

\begin{example}[\emph{Totally unimodular and non-totally unimodular representations of regular matroids}]

    Consider the matroid $M_n$ from~\cref{ex:cycleMatroid}. 
    Since $M_n$ is graphic for every $n$, it is a regular matroid by~\cref{p:regulargraphical}. More generally, let $\Delta$ be a $d$-dimensional simplicial complex with $n$ facets, $H_d(\Delta; \Z) \cong \Z$ and such that $H_d(\Gamma;\Z)=0$ for any $d$-dimensional subcomplex $\Gamma$ of $\Delta$. This last condition ensures that the unique homology cycle of $\Delta$ uses all facets of $\Delta$.
    Then the matroid represented by the $d$\textsuperscript{th} boundary map $\partial_d$ of $\Delta$ is isomorphic to $M_n$, since the only dependent set of $M_n$ is $[n]$ itself.
    This does not imply that $\partial_d$ is totally unimodular, as can be seen for example from Bj\"orner's example (see \cref{ex:bjorner}), where any $(10\times 10)$-minor of $\partial_d$ not containing the column corresponding to the facet $123$ is divisible by $2$. 
    In particular, this idea allows us to generate non-totally unimodular representations of regular matroids.
    As we will see in later sections (for example~\cref{ex:nonunimodulartriangulation}), this turns out to be useful for proving that some geometric properties of polytopes do not only depend on the underlying matroid but its representation. 
\end{example}

We will need one more concept for matroids. Given a matroid $M$, its \emph{dual matroid} $M^\ast$ is the matroid on the same ground set as $M$, whose bases are the complements of the bases of $M$. If $M$ is a realizable (resp. regular) matroid, then $M^\ast$ is also realizable (resp. regular). When $M^\ast$ is a graphic matroid, we say that $M$ is a \emph{cographic matroid}. It is a well-known fact that a graphic matroid $M$ is cographic if and only if $M$ can be represented by the incidence matrix of a planar graph~\cite[Theorem 5.2.2]{Oxley}.

The following statement, which gives an instance of when duality between matroids occurs quite naturally, will be used in \Cref{s:cohomology}.
\begin{proposition}{\cite[Proposition 6.1]{DKM2011}}\label{l:matroid_duality}
    Let $\Delta$ be a $d$-dimensional simplicial complex such that $H_d(\Delta; \Z) = 0$.
     Then the matroid $M[\partial_d^\top]$ is dual to the matroid $M[\partial_{d-1}]$.
\end{proposition}

\subsection{Polytopes}
In this section, we will provide necessary background on polytopes. For more information on this topic we refer to \cite{BeckRobbins} and \cite{Ziegler}.

A \emph{polytope} $P$ is the convex hull of finitely many points $\bv_1,\dots, \bv_n\in \R^d$. 
If $\bv_1,\dots,\bv_n\in \Z^d$, then $P$ is called a \emph{lattice polytope}.
A \emph{supporting hyperplane} of $P$ is a hyperplane $H=\{\bx\in \R^d~:~\ba^\top \bx=b\}$ with $\ba\in\R^d$ and $b\in \R$ such that $P\subseteq \{ \bx \in \R^d\colon \ba^\top \bx\leq b\}$. A \emph{face} $F$ of $P$ is the intersection $P\cap H$ of $P$ with a supporting hyperplane $H$. By convention, we will also consider $\emptyset$ and $P$ as faces. 
The \emph{dimension} of a face is the dimension of its affine hull. $0$-dimensional and $1$-dimenisonal faces are called \emph{vertices} and \emph{edges}, respectively. Faces of dimension $\dim P-1$ are called \emph{facets}. 
By a famous theorem of Weyl-Minkowski (see e.g., \cite{Ziegler}), every polytope $P\subseteq \R^d$ can also be written as a bounded intersection of finitely many closed half-spaces, i.e., $P=\{\bx\in \R^d~:~A^\top \bx\leq b\}$ for some matrix $A\in \R^{m\times d}$ and some $b\in \R^m$. 
In particular, if this description is irredundant, then each inequality (i.e., row of $A$) corresponds to a facet of $P$.
We call polytopes $P$ and $Q$ \emph{combinatorially equivalent} if their face lattices are isomorphic, i.e., there is a bijection between their vertex sets that extends to a bijection between the sets of faces of $P$ and $Q$. 
A lattice polytope $P\subseteq \R^d$ is called \emph{reflexive} if  $P=\{\bx\in \R^d~:~A^\top \bx\leq \textbf{1}\}$ for some matrix $A\in \Z^{m\times d}$. 
Equivalently, its \emph{polar dual} $\P^\lor\coloneqq \{ \by\in \R^d~:~ \bx^\top \by\leq 1 \qforall \bx\in P\}$ is a lattice polytope. 
An important criterion for reflexivity from the theory of totally unimodular matrices is the following.
\begin{theorem} {\cite[Theorem 19.1]{S1986}}\label{t:TUreflexive}
    Let $P\subseteq \R^d$ be a lattice polytope with vertices $\bv_1,\ldots,\bv_n$. If the matrix with columns $\bv_1,\ldots,\bv_n$ is totally unimodular, then $P$ is reflexive.
\end{theorem}

Given an $m$-dimensional lattice polytope $P \subset \R^d$, the numerical function $E(k, P) = |kP \cap \Z^m|$, which can be shown to be a polynomial in $k$, is called the~\emph{Ehrhart polynomial} of $P$, and the generating function
$$
    E_P(t) = \sum_{k = 0}^\infty E(k, P)t^k = \frac{h_P^\ast(t)}{(1 - t)^{m + 1}}= \frac{h_0^\ast+h_1^\ast t+\cdots+h_m^\ast t^m}{(1 - t)^{m + 1}}
$$
is called the~\emph{Ehrhart series} of $P$. The polynomial $h_P^\ast(t)=h_0^\ast+h_1^\ast t+\cdots+h_m^\ast t^m$  is called the~\emph{$h^\ast$-polynomial} of $P$.

We say that a lattice polytope $P\subseteq \R^d$ has the \emph{integer decomposition property} (IDP for short) if for every $k\in \N$, every $\textbf{x}\in  kP\cap \Z^d$ can be written as $\textbf{x}=p_{1} + \dots + p_{k}$, where $p_i\in P\cap \Z^d$. This is the case if and only if the Hilbert series of $\K[P]$ (see \Cref{subsect:TrAndGr} for the definition) and the Ehrhart series 
of $P$ agree.
A lattice polytope $P$ is called \emph{spanning}, if any lattice point in $\operatorname{aff}(P)\cap \Z^d$ can be written as an integer combination of the lattice points in $P$.  It is easily seen that every polytope which is IDP, is also spanning. 

\subsubsection{Centrally symmetric polytopes}\label{subs:cs}
The main object of interest of this article are \emph{centrally symmetric} polytopes, i.e., polytopes $P$ with $P=-P$. Every such polytope $P$ can be written as
\begin{equation}\label{eq:cs}
    P =  \conv\{\bv_1,\dots, \bv_n, -\bv_1,\dots, -\bv_n\}\eqqcolon \conv [A |-A],
\end{equation}
where $\{\pm \bv_1,\ldots,\pm \bv_n\}$ are the vertices of $P$ and $A$ is the matrix with columns $\bv_1,\ldots, \bv_n$. We will also write $\P_A$ for the  polytope in \eqref{eq:cs}. While the representation in \eqref{eq:cs} is the smallest possible one (in terms of the number of columns), adding other lattice points of $P$ as columns to $A$ will not change the polytope. This will be useful in \Cref{s:grobner}. We will frequently have the situation that the entries of $A$ are just $-1$, $0$ and $1$ (see \Cref{fig:cs polytopes} for examples). 
Maybe the most prominent and also easiest example of a centrally symmetric polytope $P$ arises when we take the convex hull of the unit vectors and their negatives, i.e., $P=\conv(\pm \be_i~:~1\leq i\leq n)\subseteq \R^n$. Any polytope that is combinatorially equivalent to this polytope will be referred to as \emph{crosspolytope}. 
Another important example is provided by the \emph{symmetric edge polytope}  (\emph{SEP} for short) $\P_G$ of  a graph $G$, which  is defined as
\[
\P_G=\conv [\partial_1 | -\partial_1]
\]
(see  \cite{MHN2011} for the original definition). 
More recently,  in \cite{DJK2024}, this class of polytopes has been extended to regular matroids by associating to a regular matroid $M[A]$~--~represented by a matrix $A$~--~the polytope $\P_A$ as defined in \eqref{eq:cs}. These polytopes are referred to as \emph{generalized symmetric edge polytopes} and share several nice properties with usual symmetric edge polytopes.

\begin{figure}[h!]
    \centering
    \begin{subfigure}{0.3\textwidth}
        \includegraphics[width= \textwidth]{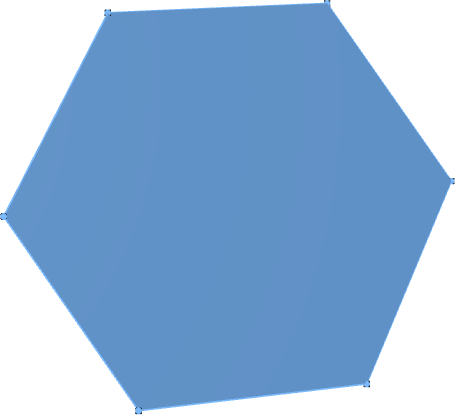}
    \end{subfigure} 
    \begin{subfigure}{0.3\textwidth}
        \includegraphics[width= \textwidth]{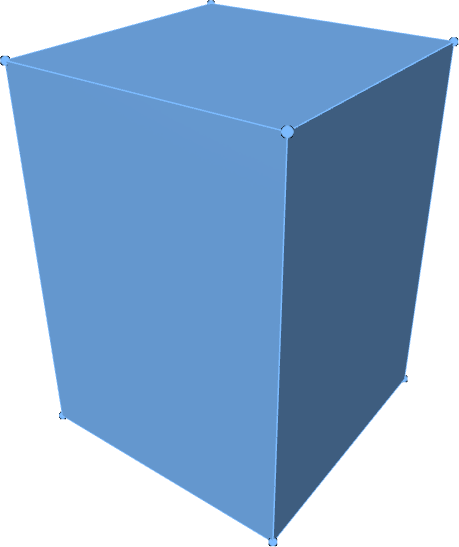}
    \end{subfigure}
    \begin{subfigure}{0.3\textwidth}
        \includegraphics[width= \textwidth]{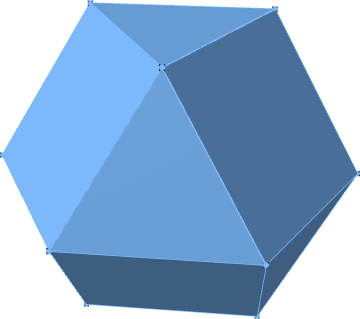}
    \end{subfigure}
    \caption{Centrally symmetric lattice polytopes defined as the convex hull of the matrices $[A | -A]$, where $M$ is the incidence matrix of the cycle graphs $C_3, C_4$ and the complete graph $K_4$ respectively. In other words, these polytopes are the symmetric edge polytopes of these graphs.}
\label{fig:cs polytopes}
\end{figure}

In this article, we want to investigate a class of centrally symmetric polytopes $\P_A$, that generalize symmetric edge polytopes of graphs from a topological perspective. Since any graph can be interpreted as a $1$-dimensional simplicial complex, it is natural to extend the definition of the symmetric edge polytope of a graph to arbitrary simplicial complexes. More precisely, we look at the following class of polytopes:

\begin{definition}[\emph{Symmetric (co)homology polytopes}] 
    Let $\Delta$ be a $d$-dimensional simplicial complex. The polytope 
    \[
    \P_\Delta \coloneqq \P_\partial\coloneqq\conv[\partial_d|-\partial_d]
    \]
    is called the \emph{symmetric homology polytope of $\Delta$}. The polytope 
    \[
    \P^\Delta \coloneqq \P_{\partial^\top}\coloneqq \conv[\partial_d^\top|-\partial_d^\top]
    \]
    is called the \emph{symmetric cohomology polytope of $\Delta$}.
\end{definition}
On the one hand, (similar as for graphs) the topological interpretation of the underlying matrices will allow us to translate geometric properties of the polytopes $\P_\Delta$ and $\P^\Delta$ into topological properties of $\Delta$ and vice versa. 
On the other hand, when generalizing the definition of a symmetric edge polytope from graphs to arbitrary simplicial complexes, we lose some crucial properties. In particular, the underlying matrix does not need to be totally unimodular (see  \Cref{p:regulargraphical}) anymore~--~a property which is true for graphs as well as for regular matroids and which allows to show that these polytopes are reflexive in an easy way. We will need to account for this issue when investigating properties, such as reflexivity or when deriving a facet description.

Polytopes $P$ and $Q$ are called \emph{unimodularly equivalent} if, after possibly embedding them in the same ambient space by adding coordinates, there exists a weakly unimodular  matrix $A$ such that $P= AQ$. We write $P\cong_u Q$ in this case.

We will need the following criterion for unimodular equivalence from \cite{DJK2024}.

\begin{theorem}{\cite[Theorem 3.2 and Remark 3.4]{DJK2024}}\label{t:matroid_equivalence}
    If $A$ and $B$ are totally unimodular matrices representing isomorphic matroids, then  $\P_A$ and $\P_B$ are unimodularly equivalent.
\end{theorem}


\subsection{Triangulations and Gr\"obner basis}\label{subsect:TrAndGr}
A \emph{triangulation} $\Delta$ of a $d$-dimensional lattice polytope $P\subseteq \R^m$ is a subdivision of $P$  into lattice simplices of dimension at most $d$.  
In the following, we use $\mathcal{V}(P)$ to denote the set of vertices of $P$. 

A triangulation $\mathcal{T}$ is called \emph{regular} if there exists a height function $\bw\in \R^{\mathcal{V}(P)}$ such that for every simplex $S\in \mathcal{T}$, the simplex $\conv\left((p,\bw(p))~:~p\in \mathcal{V}(S)\right)$ is a lower face of $\conv\left((p,\bw(p))~:~p\in \mathcal{V}(S)\right)$, i.e., a face having an outer normal vector with negative last coordinate. 
Such a triangulation is \emph{unimodular} if all its maximal simplices are unimodular, i.e., they have Euclidean volume $\frac{1}{d!}$ and hence  normalized volume equal to $1$.
In other words, a unimodular simplex, is an integral simplex of minimal Euclidean volume. Unimodular triangulations do not always exist, the simplest example being Reeve's tetrahedron in dimension $3$.
This is, in particular the case, if the lattice $M_P\coloneqq \aff(P)\cap \Z^m$ spanned by $P$ is not unimodularly equivalent to $\Z^m$. 
In this situation, the minimal volume of an $d$-dimensional  simplex with vertices in $M_P$  equals $\frac{|\det(M_P)|}{d!}$, where $\det(M_P)$ is the determinant of a matrix whose columns are a lattice basis for $M_P$. 

Motivated by this, we will call a $d$-simplex with vertices in $M_P$  \emph{unimodular with respect to $M_P$} if it has Euclidean volume $\frac{|\det(M_P)|}{d!}$. 
Similarly, a triangulation $\Delta$ of $P$ is called \emph{unimodular with respect to $M_P$} if each maximal simplex in $\Delta$ has this property.
This situation will occur several times in this paper.

To calculate a regular, possibly unimodular (with respect to $M_P$), triangulation of $P$, one often computes a Gr\"obner basis. 
Let $\mathbb{K}$ be a field and $\mathbb{K}[t_1^{\pm},\ldots,t_d^{\pm},s]$ be the Laurent polynomial ring. 
For a lattice polytope $P\subseteq \R^m$ and any $\alpha=(\alpha_1,\ldots,\alpha_m)\in P\cap \Z^m$, we set $t^\alpha=t_1^{\alpha_1}\cdots t_m^{\alpha_m} \in \mathbb{K}[t_1^{\pm},\ldots,t_m^{\pm},s]$. 
The ring  $\mathbb{K}[P]\coloneqq \mathbb{K}[t^\alpha s~:~\alpha\in P\cap \Z^m]$ is the \emph{toric ring} of $P$, i.e., $\mathbb{K}[P]$ is the subring of $\mathbb{K}[t_1^{\pm},\ldots,t_m^{\pm},s]$ generated by the monomials $t^\alpha s$ for $\alpha\in P\cap \Z^m$. Let $S=\mathbb{K}[x_\alpha~|~\alpha\in P\cap \Z^m]$, where $\deg(x_\alpha)=1$. The \emph{toric ideal} $I_{P}$ of $P$ is the kernel of the surjective ring homomorphism
\[
\pi_P \colon S\to \mathbb{K}[P] \text{, with } \pi_P(x_\alpha) \coloneqq t^\alpha s.
\]

A \emph{term order} on a polynomial ring, here $R=\mathbb{K}[x_1,\ldots,x_m]$, is a  total order $\prec$ on the set of monomials, such that for all monomials $a,b,c$ we have $1\prec a$, and $a\prec b$ implies $ac\prec bc$. We will mostly consider the \emph{degree reverse lexicographic ordering} $\prec_{\rev}$ (degrevlex), induced by $x_1\prec\cdots \prec x_m$; we have $x_1^{\alpha_1}\cdots x_m^{\alpha_m}\prec_{\rev}x_1^{\beta_1}\cdots x_m^{\beta_m}$  if and only if $\sum_{i=1}^m \alpha_i<\sum_{i=1}^m \beta_i$, or  $\sum \alpha_i=\sum \beta_i$ and $\alpha_\ell>\beta_\ell$ for $\ell=\min\{j~:~\alpha_j\neq \beta_j\}$. For a polyomial $f\in R$, its \emph{initial term} $\init_\prec(f)$ is the largest monomial appearing in $f$. The \emph{initial ideal} $\init_\prec(I)$ of an ideal $I\subseteq R$ is the ideal generated by all initial terms in $I$. We call a set of  generators $g_1,\ldots,g_r$ of $I$ a \emph{Gr\"obner basis} (w.r.t. $\prec$) if $\init_\prec(g_1),\ldots,\init_\prec(g_r)$ generate $\init_\prec(I)$. 

A Gr\"obner basis of the toric ideal $I_{P}$ of a lattice polytope $P$ gives rise to a regular triangulation $\Delta_{P}$ of $P$ on the vertex set $P\cap \Z^d$  as follows: A  set $M\subseteq P\cap \Z^d$ is the vertex set of a simplex in $\Delta_{P}$ if and only if 
$$
\prod_{\alpha\in M}x_\alpha\notin \sqrt{\init_\prec(I_{P})}= \bigl(f: f^k\in \init_\prec(I_{P}) \text{ for some } k\bigr).
$$ Indeed, every regular triangulation of $P$ can be constructed this way \cite{S1996}. Moreover, the following statement holds:

\begin{theorem}{\cite[Corollary 8.9]{S1996}}\label{t:sturmfels}
The triangulation $\Delta_{P}$ is regular and \emph{unimodular w.r.t. the lattice} $\aff(P)\cap \Z^d$  if and only if $\init_\prec(I_{P})=\sqrt{\init_\prec(I_{P})}$. This is the case if and only if $I_{P}$ has a squarefree Gr\"obner basis, i.e., a Gr\"obner basis of $I_P$ whose initial terms are squarefree.
\end{theorem}

Finally, the~\emph{Hilbert function} of a toric ring $\K[P]$ is the numerical function $\HF_{\K[P]}(i) = \dim_{\K} (\K[P])_i$, where $(\K[P])_i$ denotes the homogeneuos component of $K[P]$ of degree $i$.  The~\emph{Hilbert series} of $\K[P]$ is the generating function 
$$
    \HS_{\K[P]}(t) = \sum_{i = 0}^\infty \HF_{\K[P]}(i) t^i = \frac{h_P(t)}{(1 - t)^{d + 1}},
$$
where $d$ is the dimension of the polytope $P$. We will use in several examples that a lattice polytope $P$ is IDP if and only if $h_P(t)=h_P^\ast(t)$.

\section{Symmetric (co)homology polytopes and their basic properties}\label{s:basic}
In this section, we  study basic properties of symmetric (co)homology polytope. In particular, we provide a description of the faces and facets of symmetric homology polytopes (\Cref{t:facets} and \Cref{c:facettrees}).

\begin{remark}[\emph{Pure vs nonpure simplicial complexes}]\label{r:pure}
     We note that, by definition of $\partial_d$, the polytopes $\P_\Delta$ and $\P^{\Delta}$ only depend on the facet-ridge-incidences of $\Delta$ for $d$-dimensional facets (where $\dim\Delta=d$) but not on possibly lower-dimensional facets. We can hence assume $\Delta$ to have only facets of dimension $d$ or $d-1$. Since the last ones correspond to rows of zeros in the top boundary map $\partial_d$,  the existence of such a facet would imply that $\P_\Delta$ lies in a coordinate hyperplane. Similarly, for $\P^\Delta$, this would mean that the origin appears in the convex hull, which does not affect the polytope either. In the following, we will therefore often assume that $\Delta$ is pure.
\end{remark}

\begin{remark}[\emph{Arbitrary matrices and boundary maps of cell complexes}]
    As pointed out in~\cite[Remark 4.2]{DKM2015}, every integer valued matrix can be realized as the top boundary map of some cell complex. Therefore, theoretically, we could try to translate those of our results that hold for arbitrary matrices to this more topological setting. Unfortunately, most of these results require  some extra conditions, such as the polytope containing the origin as unique interior lattice point, which do not necessarily translate into natural conditions on cell complexes. For this reason, we state topological results only for  the setting of simplicial complexes.
\end{remark}

We now prove several basic properties of symmetric (co)homology polytopes. 
We begin by characterizing their vertices and lattice points.

\begin{lemma}\label{l:latticepoints}
    Let $\Delta$ be a pure $d$-dimensional simplicial complex and let $\P$ be either $\P_\Delta$ or $\P^\Delta$. Then the vertices of $\P$ are the columns of $\partial_d$ (resp. $\partial_d^\top$). 
    In particular, 
    \[
    f_0(\P_\Delta)=2f_d(\Delta) \qquad  \text{and} \qquad f_0(\P^\Delta)=2(f_{d-1}(\Delta) - \sum_{\sigma\in \Delta, \dim \sigma=d}\max(\free(\sigma)-1,0)),
    \]
    where $\free(\sigma)$ denotes the number of free ridges of a facet $\sigma$. 
    Moreover, the only other lattice point of $\P$ is the origin.
\end{lemma}
Before proving this statement, we note that whereas, in $\partial_d$ all columns are different, in $\partial_d^\top$ one and the same column can occur several times. 
\begin{proof}

We first note that since $\partial_d$ is a $\{-1,0,1\}$-matrix, we have $|\textbf{x}_i|\leq 1$ for all $i$ and all $\textbf{x}\in\P$. In particular, $\textbf{x}\in\P$ is a lattice point if and only if $\textbf{x}_i\in\{-1,0,1\}$ for all $i$. 

In the following, let $A_\Delta=[\partial_d|-\partial_d]$ and let $A^\Delta$ be obtained from $[\partial_d^\top|-\partial_d^\top]$ by removing 
repeated columns. Note that $[\partial_d^\top|-\partial_d^\top]$ does not have $0$-columns since $\Delta$ is pure.  
Let $\bw \in \P_\Delta$ (resp. $\bw \in \P^\Delta$) be a lattice point that is different from the origin and from $\bv$ for any column $\bv$ of $A_\Delta$ (resp.  $A^\Delta$). 
Then, there exist columns $\bv_{1},\ldots,\bv_{s}$ of $A_\Delta$ (resp. $A^\Delta$) and $\lambda_1,\ldots,\lambda_s>0$ with $\sum_{\ell=1}^s\lambda_\ell=1$ such that $\bw=\sum_{\ell=1}^s\lambda_\ell \bv_{\ell}$. By assumption, we have $s\geq 2$ and there exists $i$ such that $\bw_i\neq 0$.
Since $\P_\Delta$ and $\P^\Delta$ are centrally symmetric and since $|\bw_i|\leq 1$  by the argument at the beginning,  we can assume that  $\bw_i=1$. 
Moreover, since $|(\bv_j)_i|\leq 1$ for all $1\leq j\leq s$ and $\sum_{\ell=1}^s\lambda_\ell=1$, it follows that $(\bv_{j})_i=1$ for all $1\leq j\leq s$. This, in particular, implies that $\{-\bv,\bv\}\not\subseteq \{\bv_1,\ldots,\bv_s\}$ for any column $\bv$ of $\partial_d$ (resp. $\partial_d^\top$). We show the following claim.

{\sf Claim:} There exist $1\leq j<j'\leq s$ and  $i\neq k$ such that $(\bv_j)_k\neq 0$ and $(\bv_{j'})_k \neq 0$.

 First assume $\bw=\be_i$. Since $\bw\neq \bv_\ell$ for $1\leq \ell\leq s$, there exist $1\leq j\leq s$ and  $k\neq i$ such that $(\bv_j)_k\neq 0$. As $\bw_k=0$, this forces the existence of another $1\leq j'<s$ with $j'\neq j$ such that $(\bv_{j'})_k\neq 0$, which shows the claim. 
 If $\bw\neq \be_i$, then there exists $k\neq i$ such that $\bw_k \neq 0$ and,  by the same argument as in the second paragraph of the proof, we have $(\bv_j)_k\neq 0$ for all $1\leq j\leq s$. This also shows the claim.

The claim together with the fact that $(\bv_{j})_i=(\bv_{j'})_i=1$ (by the discussion above, implies that the facets 
$\sigma_j$, $\sigma_{j'}$ corresponding to the vectors $\bv_{j}$,  $\bv_{j'}$ share two ridges which is impossible, as facets can intersect in at most one ridge. (Note that by the argument above $F_j\neq F_{j'}$.)
For $\P = \P^\Delta$, it follows that the two ridges $r_j$, $r_{j'}$ corresponding to the columns  $\bv_{j}$, $\bv_{j'}$ are contained in two distinct joint facets.
This is impossible since two ridges can be contained in at most one joint facet. 

The same argument also shows that every column of $A_\Delta$ (resp. $A^\Delta$) is a vertex of $\P_\Delta$ (resp. $\P^\Delta$).
Since all columns of $\partial_d$ are distinct, this implies that $f_0(\P_\Delta)=2f_d(\Delta)$. 

Columns of $\partial_d^\top$ are the same if and only if the corresponding ridges are contained in the same facets.
Since two ridges already determine a facet completely, this can only happen if the corresponding ridges are contained in exactly one facet.
Hence, the ridges are free ridges of the same facet. 
Consequently, for every $d$-dimensional facet $\sigma$ with $m$ free ridges, $\partial_d^\top$ contains  $m$ identical columns all corresponding to the same vertex of $\P^\Delta$ . The claim follows. 
\end{proof}

It is not hard to see that for a column $\bv$ of $\partial_d$ a supporting hyperplane is provided by 
 the linear functional $\ell_\bv$ mapping $\bx$ to $\bv^\top \bx$. Indeed, $\ell_\bv(\bv)=d+1$ and, since every column has exactly $d+1$ entries equal to $1$ or $-1$ (and all other entries equal to $0$), $\ell_\bv(\bw)< d+1$ for every other column $\bw$ of $[\partial_d|-\partial_d]$.
 Moreover, the corresponding linear functional defines a supporting hyperplane for columns of $\partial_d^\top$ that do not correspond to  free ridges.
 
\begin{proposition}\label{p:basic}
    Let $d>0$ and let $\Delta$ be a $d$-dimensional simplicial complex. Then:
    \begin{enumerate}
        \item $\dim \P_\Delta = \dim \P^\Delta = f_{d}(\Delta) - \dim_\Q H_d(\Delta; \Q)$.
        \item The affine hull of $\P_\Delta$ is the orthogonal complement of the linear space spanned by the $d$-cocycles of $\Delta$ in $\R^{f_{d-1}(\Delta)}$. 
        \item The affine hull of $\P^\Delta$ is the orthogonal complement of the linear space spanned by the $d$-cycles of $\Delta$ in $\R^{f_{d}(\Delta)}$.
    \end{enumerate}
\end{proposition}

\begin{proof}
We first prove (1). 
 Since both, $\P_\Delta$ and $\P^\Delta$, contain the origin, we have
    $$
    \dim \P_\Delta = \rank [\partial_d | - \partial_d] = \rank \partial_d = \rank \partial_d^\top = \rank [\partial_d^\top | - \partial_d^\top] = \dim \P^\Delta.
    $$
    Since $\dim\Delta=d$, we know that $H_d(\Delta; \Q) = \ker_\Q \partial_d$ (as a $\Q$-vector space), which implies $\rank \partial_d = f_d(\Delta) - \dim_\Q H_d(\Delta; \Q)$ and thus the claim.

To prove (2), we interpret $\partial_d$ as a map from $\R^{f_d(\Delta)}$ to $\R^{f_{d-1}(\Delta)}$. With this, we have  $\P_\Delta\subseteq \im_\R\partial_d$ and $\dim \P_\Delta = f_d(\Delta)-\dim \ker_\R\partial_d= \dim \im_\R\partial_d$ by (1). It follows that the affine hull of $\P_\Delta$ equals $\im_\R\partial_d$. Moreover, by standard facts from linear algebra, the latter is orthogonal to $\ker_\R \partial_d^\top$, which is the exactly the space of $d$-cocycles of $\Delta$. The claim follows. 

Finally, (3) can be shown by the same arguments as (2).
\end{proof}

\begin{remark}[\emph{Symmetric homology polytopes of graphs and SEPs}]
    For a connected graph $G$ on $n$ vertices and $s$ edges \cref{p:basic} says that $\dim \P_G = s- \dim_\Q H_1(G;\Q)$. On the other hand, since $\P_G$ is the symmetric edge polytope of $G$ in this case, it is well-known and easy to see that  $\dim \P_G=n-1$ and hence, $s- \dim_\Q H_1(G;\Q)=n-1$, an identity which can also easily be seen graph theoretically.
    Moreover, by construction, $\P_G$ is contained in the hyperplane normal to $(1,\dots,1)$ which generates the space of $1$-cocycles of $G$. 
    \end{remark}


\subsection{Facets}
The goal of this section is to describe faces  of the symmetric homology polytope $\P_\Delta$. In particular, we will  provide a natural generalization of the facet description of the symmetric edge polytope of a graph \cite[Theorem 3.1]{HJM2019}. We start with a description of general faces.

\begin{theorem}\label{t:facets}
    Let $\Delta=\langle \sigma_1,\ldots,\sigma_s\rangle$ be a pure $d$-dimensional simplicial complex with $r$ ridges. Let $F_i$ denote the $i$\textsuperscript{th} column of $\partial_d$. 
    Faces of $\P_\Delta$ correspond to labelings $\mathbf{\ell}\in [-1,1]^s$ of the  facets of $\Delta$ such that $\sum_{j = 1}^s \ell_j \sigma_j \in \im_\R \partial_d^\top$ and $\max(|\ell_j|~:~ 1\leq j\leq  s) = 1$. 
    In particular, the  face corresponding to the labeling $\mathbf{\ell}$ is given by $\conv\{ \ell_i F_i~:~ \vert \ell_i\vert = 1\}$. Moreover, $\{\bx\in \R^r~:~\bw^\top \bx= 1\}$, where $\bw\in \R^r$ is such that $\partial_d^\top \bw=\sum_{j = 1}^s \ell_j \sigma_j$, is a  supporting hyperplane for this face, containing $\P_\Delta$ in its negative halfspace.
    \end{theorem}

We note that the condition $\sum_{j = 1}^s \ell_j \sigma_j \in \im_\R \partial_d^\top$ on a labeling $\mathbf{\ell}$ just says that $\mathbf{\ell}$ is a coboundary. In other words, $\mathbf{\ell}$ vanishes on homology cycles.

\begin{proof}
    Let $\mathbf{\ell} \in [-1, 1]^s$ be a labeling such that $\sum_{j = 1}^s \ell_j \sigma_j \in \im_\R \partial_d^\top$ and choose $\bw \in \R^r$ such that $\partial_d^\top \bw = \sum_{j = 1}^s \ell_j \sigma_j$. For a column $F_j = \partial_d \sigma_j$ (for $-F_j = -\partial_d \sigma_j$ the same argument works) we have:
    \begin{equation}\label{eq:faces}
        \bw^\top (F_j) = \bw^\top (\partial_d \sigma_j) = (\partial_d^\top \bw)^\top \sigma_j = \left(\sum_{j = 1}^s \ell_j \sigma_j\right)^\top 
        \sigma_j = \ell_j \in [-1, 1].
    \end{equation}
     By convexity, we conclude $\P_\Delta \subseteq \{\bx \in \R^r\st \bw^\top \bx \leq 1\}$. This shows that $\P_\Delta\cap \{\bx\in \R^r~:~\bw^\top \bx=1\}$ is a face of $\P_\Delta$. Moreover, the above computation shows that the set of vertices of this face equals $\{\ell_iF_i~:~|\ell_i|=1\}$. 

    To see why each face is of the described form, let  $\Omega$ be a face of $\P_\Delta$. Since the origin $\mathbf{0}$ lies in the interior of $\P_\Delta$, it follows that $\mathbf{0}\notin \Omega$  and hence there exists
     $\bw\in \R^r$  such that $\Omega = \P_\Delta \cap \{\bx \in \R^r \st \bw^\top \bx = 1\}$ and $\P_\Delta \subseteq \{\bx \in \R^r \st \bw^\top \bx \leq 1\}$. 
     Let  $\partial_d^\top \bw = \sum_{j = 1}^s \ell_j \sigma_j$. By the choice of $\bw$ and~\eqref{eq:faces} we know that $|\ell_j| = |\bw^\top (F_j)| \leq 1$ and thus $\ell_j\in [-1,1]$.
    Setting $\mathbf{\ell} = (\ell_1, \dots, \ell_s)$ yields the desired result. 
    \end{proof}

    We want to remark that \cref{t:facets} can be seen as a version of centrally symmetric Gale duality \cite[Equations (9) and (10)]{MS1968}.
    In particular, it can also be stated for general matrices. 

Facets of symmetric edge polytopes of connected graphs are known to correspond to connected subgraphs that contain a spanning tree.(see~\cite[Theorem 3.1]{HJM2019}). 

For higher dimensional simplicial complexes, more than one notion of a \emph{simplicial spanning tree} exists. It turns out that the right one for our purposes is the following generalization of the one used in~\cite{DKM2009} to generalize the Matrix Tree theorem to higher dimensional simplicial complexes.

\begin{definition}
    Let $d \geq 1$ and let $\Delta$ be a $d$-dimensional simplicial complex. A  subcomplex $\Gamma$ of $\Delta$  is called \emph{simplicial spanning forest} for $\Delta$ if it contains all $(d - 1)$-faces of $\Delta$ and if it satisfies the following conditions: 
    \begin{itemize}
    \item[(1)] $H_d(\Gamma;\Z)=0$, and
    \item[(2)] $f_d(\Gamma)=f_d(\Delta)-\dim_\Q H_d(\Delta;\Q)$.
\end{itemize}
\end{definition}
 We want to remark, that a simplicial spanning forest is a simplicial spanning tree in the sense of \cite{DKM2009} if, additionally, $\dim_\Q H_{d - 1}(\Gamma; \Q) = \dim_\Q H_{d - 1}(\Delta; \Q) = 0$.
  
While a $d$-dimensional simplicial complex admits a simplicial spanning tree if and only if $H_k(\Delta; \Q) = 0$ for $0 < k < d$ and $H_0(\Delta; \Q) = \Q$ (see  \cite[Proposition 3.7]{DKM2009}), it is easy to see that every simplicial complex $\Delta$ admits a simplicial spannning forest: Indeed, any subcomplex of $\Delta$ whose facets correspond to a maximal set of linearly independent columns of the top boundary map $\partial_d$ is a simplicial spanning forest. We can then prove the following.


\begin{corollary}[Facets of $\P_\Delta$ and spanning forests of $\Delta$]\label{c:facettrees}
Let $\Delta = \tuple{\sigma_1, \dots, \sigma_s}$ be a pure $d$-dimensional simplicial complex. Then the facets of $\P_\Delta$ correspond to labelings $\mathbf{\ell} \in [-1,1]^s$ of the facets of $\Delta$ such that $\sum_{j=1}^s \ell_j\sigma_j\in \im_\R \partial_d^\top$ and $\Gamma_{\mathbf{\ell}} = \tuple{\sigma_i \st |\ell_i| = 1}$ contains a simplicial spanning forest of $\Delta$.
\end{corollary}

\begin{proof}

Let $\mathbf{\ell} \in [-1,1]^s$ be a labeling of the facets of $\Delta$ such that $\sum_{j=1}^s \ell_j\sigma_j\in \im_\R \partial_d^\top$. From \Cref{t:facets}, it follows that $\mathbf{\ell}$ gives rise to a face  $\Omega = \conv\{\ell_i F_i \st |\ell_i| = 1\}$ of $\P_\Delta$, where $F_i$ is the $i$\textsuperscript{th} column of the top boundary map $\partial$ of $\Delta$. Moreover, every face is obtained this way. It hence suffices to show that $\Omega$ is a facet of $\P_\Delta$ if any only if  $\Gamma_\ell  = \tuple{\sigma_i \st |\ell_i| = 1}$ contains a simplicial spanning forest.

First, assume that $\Omega$ is a facet of $\P_\Delta$. Let $\partial_\ell$ be the top boundary map of $\Gamma_\ell$, and let $B$ be a maximal independent set of columns of $\partial_\ell$. Moreover, let $\Sigma_B = \tuple{\sigma_k \st F_k \in B}$ be the corresponding simplicial complex, and let $\partial_B$ be the top boundary map of $\Sigma_B$.  Since $\dim\Omega=\dim \P_\Delta-1$, it follows that $\partial_d$, $\partial_\ell$ and $\partial_B$ have the same rank. This implies that both, $\Sigma_B$ and $\Gamma_\ell$, contain all  ridges of $\Delta$. Using that the columns of $\partial_B$ are linearly independent, we conclude that $H_d(\Sigma_B; \Z) = 0$ and that
$$
    f_d(\Sigma_B) = \rank \partial_B = \rank \partial_d = f_d(\Delta) - \dim_\Q H_d(\Delta; \Q).
$$
This shows that $\Sigma_B$ is a  simplicial spanning forest of $\Delta$. Since $\Sigma_B$ is contained in $\Gamma_\ell$, the claim follows. 

Conversely, assume that $\Gamma_\ell = \langle \sigma_i \st |\ell_i| = 1\langle$  contains a simplicial spanning forest $\Sigma$ of $\Delta$. Let $\partial_\Sigma$ be the top boundary map of $\Sigma$ and $\partial_\ell$ be the top boundary map of $\Gamma_\ell$. By the definition of a simplicial spanning forest, we obtain
$$
    \rank \partial_d = \rank \partial_\Sigma \leq \rank \partial_\ell \leq \rank \partial_d
$$
which implies $\dim\Omega=\rank\partial_\ell-1=\rank\partial_d-1=\dim\P_\Delta-1$, i.e., $    \Omega$ is a facet of $\P_\Delta$.
%
%
%
%
\end{proof}



We end this section with an application of \Cref{c:facettrees} to a specific class of simplicial complexes that can be seen as a generalizations of cycles in the case of graphs. 

 Before providing this application, we need to introduce some further notation. Given a multiset $S$  with elements in $\R$, a pair of multisets $(M_1,M_2)$ is called an \emph{equal weight partition} of $S$, if  $M_1, M_2 \subseteq S$, $M_1\cup M_2=S$ (as  multisets), and $\sum_{x\in M_1} x = \sum_{y\in M_2} y$. We denote by $\partition(S)$ the number of equal weight partitions of $S$. 
To simplify notation, we will refer to vectors in $\ker_\Z A$ as \emph{linear dependencies} (for $A$).

 
    
    \begin{corollary}\label{c:cycle_facets}
        Let $\Delta$  be a pure $d$-dimensional simplicial complex with $f_d(\Delta)=s$, $H_d(\Delta;\Z)=\Z$ and  $H_d(\Delta\setminus \{\sigma\};\Z)=0$ for every facet $\sigma$ of $\Delta$. 
        Let $\mathbf{a}\in\Z^s$ be a linear dependency that generates $H_d(\Delta;\Z)$ and let $S=\{a_1,a_2,\dots,a_s\}$. If $\mathbf{\ell}$ is the labeling corresponding to a facet $\Omega$ of $\P_\Delta$, then one of the following conditions holds:
        \begin{itemize}
            \item[(1)] $(\{a_i~:~\ell_i=1\},\{a_i~:~\ell_i=-1\})$ defines an equal weight partition of $S$;
            \item[(2)] There exists $1\leq k\leq s$ such that $\ell_k=0$ and
            $(\{a_i~:~\ell_i=1\},\{a_i~:~\ell_i=-1\})$ defines an equal weight partition of $S\setminus\{a_k\}$.
            \item[(3)] There exists $1 \leq k \leq s$ such that $\ell_k \not \in \{-1, 0, 1\}$ and either $(\{a_i~:~\ell_i=1\}\cup\{\ell_ka_k\},\{a_i~:~\ell_i=-1\})$ or  $(\{a_i~:~\ell_i=1\},\{a_i~:~\ell_i=-1\}\cup\{\ell_ka_k\})$ defines an equal weight partition of $S\setminus\{a_k\}\cup\{\ell_ka_k\}$.
            
        \end{itemize}
        Moreover, every equal weight partition as in (1), (2) or (3) defines a facet of $\P_\Delta$. In particular, the number of facets of $\P_\Delta$ is equal to 
    $$
         \partition(S) + \sum_{k=1}^s \partition(S\setminus \{a_k\}) +\sum_{k=1}^s\sum_{i=1}^{|a_k|-1} \partition(S\cup \{i\}\setminus \{a_k\}). 
    $$ 
    \end{corollary}
We note that the conditions on the homology of $\Delta$ imply that  the kernel of $\partial_d$ is spanned by a  unique linear dependency $\mathbf{a}$ involving \emph{all} facets of $\Delta$.
 Examples of such simplicial complexes are triangulations of spheres but also the simplicial complex $\B$ from~\cref{ex:bjorner}.  

    \begin{proof}
    Let $\sigma_1,\ldots,\sigma_s$ be the facets of $\Delta$, and let $F_1,\ldots,F_s$ denote the corresponding column of $\partial_d$. 
    Let $\Omega$ be a facet of $\P_\Delta$, and let $\mathbf{\ell}\in[-1,1]^s$ be the corresponding labeling. \Cref{p:basic} (1) implies that $\dim\Omega=s-2$. In particular, by~\cref{c:facettrees}, $\Omega$ has to contain exactly one of $F_i$ and $-F_i$ for all $1\leq i\leq s$, or all $1\leq i\leq s$ except one particular   $1\leq k\leq s$.  
    Also note that the set of vertices of $\Omega$ is equal to $\{\ell_iF_i \st |\ell_i| = 1\}$.
    
    Since  $\sum_{j=1}^s \ell_j\sigma_j\in \im_\R \partial_d^\top$ and $\mathbf{a}\in \ker_\R \partial_d$, we conclude 
\begin{equation} \label{eq:proofcyclefacets}
        0=\sum_{i=1}^s a_i \ell_i = \begin{cases}
         \sum_{\{i~:~\ell_i=+1\}} a_i - \sum_{\{i~:~\ell_i= -1\}} a_i,\quad &\text{ if } |\ell_i|=1\text{ for all } i\\
            \sum_{\{i~:~\ell_i=+1\}} a_i - \sum_{\{i~:~\ell_i= -1\}} a_i + \ell_k a_k,\quad  &\text{ if } |\ell_k|\neq 1 \text{ and } |\ell_i|=1 \text{ for } i\neq k.
            
        \end{cases}
\end{equation}
    In the first case, the labeling $\mathbf{\ell}$  corresponds to an equal weight partition of $S$ by definition, which is condition (1). In the second case, we first observe that it follows that $\ell_ka_k\in [-a_k,a_k]\cap \Z$. If $\ell_k=0$, then $\mathbf{\ell}$ gives rise to an equal weight partition of $S\setminus \{a_k\}$, which is condition (2). 
    Similarly, if $\ell_k\notin \{-1,0,1\}$, we have an equal weight partition of $S\setminus \{a_k\} \cup \{ \ell_k a_k \}$, which is condition (3).
    For the \emph{Moreover}-statement one easily verifies that every equal weight partition as in (1), (2) or (3) gives rise to a labeling $\mathbf{\ell}\in[-1,1]^s$ with $\sum_{j = 1}^s \ell_j \sigma_j \in \im_\R \partial^\top$ and such that  $\Gamma_\ell \coloneqq \tuple{\sigma_i \st |\ell_i| = 1}$ contains a simplicial spanning forest of $\Delta$. 

    The counting formula is immediate since the first, second and third summand above equals the number of equal weight partitions of type
    (1), (2) and (3), respectively.
    \end{proof}

As a special case of~\cref{c:cycle_facets}, we obtain a simple formula for the number of facets of $\P_\Delta$, if $\Delta$ is a connected closed orientable $d$-pseudomanifolds.
In this situation, there exists a linear dependency $\mathbf{a}\in \{-1,1\}^{f_d(\Delta)}$ that generates $H_d(\Delta;\Z)$.

\begin{corollary}\label{c:specialcycles}
    Let $\Delta$ be a connected closed orientable $d$-pseudomanifold with $s$ facets. Then the number of facets of $\P_\Delta$ equals ${2m \choose m}$, if $s=2m$, and it equals $(2m+1)\cdot{2m \choose m}$ if $s=2m+1$ (where $m\in \mathbb{N}$).
       
\end{corollary}
\begin{proof}
    The statement follows directly from \Cref{c:cycle_facets} by observing that the multiset $S$ contains $1$ with multiplicity $s$ and nothing else.
    We then have two cases:
    If $s=2m$, equal weight partitions correspond to $m$-element submultisets of $S$ and $S\setminus \{1\}$ has no equal weight partition.
    If $s=2m+1$, then $S$ has no equal weight partition and equal weight partitions of $S\setminus \{1\}$ correspond to $m$-element submultisets of $S\setminus \{1\}$ and the choice of which of the $1$s to delete from $S$.
\end{proof}

Finally, we note that~\cref{c:specialcycles} specializes to the number of facets of the symmetric edge polytope of a cycle when $d=1$ (see e.g., \cite[Proposition 4.3]{DDM2022}).

\section{Gröbner bases and regular (unimodular) triangulations}\label{s:grobner}
 


The focus of this section lies in the construction of regular triangulations for general centrally symmetric lattice polytopes.  
As pointed out in \Cref{subs:cs},  every such polytope is of the form $\P_A$, where $A$ is a matrix that contains at least one vertex from each antipodal pair of vertices of $\P_A$ as a column, but that could possibly contain all lattice points of $P$. We call $A$ \emph{saturated}, if every lattice point of $\P_A$, except the origin, occurs in $[A|-A]$ as a column exactly once. It turns out that this is the right type of presentation for the purpose of this section. 

 In order to compute regular triangulations, we will compute a Gröbner basis for the corresponding toric ideal. We note that, by~\cref{l:latticepoints}, boundary maps of simplicial complexes and incidence matrices of graphs are saturated.
  In the latter case, it is well-known that the toric ideals have a \emph{squarefree} Gr\"obner basis which implies that these polytopes always admit a regular unimodular triangulation~\cite{DHO2024,HJM2019}. 
  Though it is not surprising, that this is no longer true for arbitrary centrally symmetric lattice polytopes, it turns out that this can even fail  for symmetric (co)homology polytopes (see ~\cref{t:noSquarefreeGroebner}). 
  Still the Gr\"obner basis, we provide, can be seen as a natural analogue of the Gr\"obner basis from \cite{DHO2024,HJM2019} for the considered more general setting (see \Cref{t:groebner_homology}).

In order to compute a Gröbner basis of $I_{\P_A}$, we need to introduce a monomial order and set up some notation. In the following, let $A\in \Z^{m\times s}$ be saturated.

Inside the toric ring, there is one variable for each lattice point in $\P_A$. We will use $z$ for the variable corresponding to the origin.  We further use $x_{F^+}$  and $x_{F^-}$ for the variable corresponding to a column $F$ of $A$ and its negative $-F$, respectively. To simplify notation, we will also just write $x_F$ if $F$ is a column of $[A|-A]$. Moreover, for a column $F$ of $A$, we set $\overline{F^+}\coloneqq F^-$ and $\overline{F^-}=F^+$.  


For a fixed ordering $F_1,\ldots, F_s$ of the columns of $A$, we consider the degree reverse lexicographic ordering induced by 
    \begin{equation}\label{eq:order}
        z\prec x_{F_{1}^+}\prec x_{F_{1}^-} \prec \dots \prec x_{F_{s}^+}\prec x_{F_s^-}.        
    \end{equation}
    A linear dependency $\mathbf{0}\neq \mathbf{a}\in \ker_\Z A$ is will be called \emph{minimal}  if there is no $\mathbf{a'}\in \ker_\Z A$ with $\mathbf{a}'\neq \mathbf{a}$ such that $\sign(\mathbf{a}_i')=\sign(\mathbf{a}_i')$ and $|\mathbf{a}'_i|\leq |\mathbf{a}_i|$ for all $1\leq i\leq s$.  Here, $\sign(\mathbf{a}_i)\in \{-,0,+\}$ denotes the sign of $\mathbf{a}_i$ and similarly for $\mathbf{a}'_i$.

    While for (generalized) symmetric edge polytopes the non-zero entries of  every minimal linear dependency are either $1$ or $-1$,  this does not need to be the case in our more general setting. This forces several changes compared to the approach in \cite[Proposition 3.8]{HJM2019}. 
    We associate the following multiset to a vector  $\mathbf{a}\in \Z^s$:
    \[
        M_\mathbf{a} = \{\underbrace{\sign(\ba_{\ell})F_{\ell},\cdots, \sign(\ba_{\ell}) F_{\ell}}_{|\ba_{\ell}| \mbox{ times}}~:~\ell\in \supp(\mathbf{a}) \},
    \]
    where $\supp(\mathbf{a})\coloneqq\{1\leq i\leq s~:~\ba_i\neq 0\}$ denotes the \emph{support} of $\mathbf{a}$. 
    Note that $|M_\mathbf{a}|=\sum_{\ell=1}^s |\ba_\ell|$. Moreover, by construction, multisets on the column set of $[A|-A]$ are in natural bijection with vectors in $\Z^s$ and, under this bijection, minimal linear dependencies $\mathbf{a}$ correspond to multisets  $M_{\mathbf{a}}$ that are minimal with respect to inclusion.

Given a multiset $M_\ba$ on the column set of  $[A|-A]$, we let $-M\coloneqq \{-F~:~F\in M_\ba\}$ be the multiset obtained from $M_\ba$ by flipping all signs and we write $\mathbf{m}_{M_\ba}$, or $\mathbf{m}_\ba$ for short, for the monomial associated to $M_\ba$, i.e., $\mathbf{m}_{M_\ba}=\mathbf{m}_\ba =\prod_{F\in M_\ba}x_F$.
For $\mathbf{a}\in \Z^s$, we further set $\min(\mathbf{a})=\min\supp(\mathbf{a})$. Our main result in this section is the following:
    \begin{theorem}\label{t:groebner_homology}\
    Let $A \in \Z^{m \times s}$ be saturated. Then the following binomials form a \grobner basis of the toric ideal $I_{\P_A}$ with respect to the degrevlex monomial ordering induced by~\eqref{eq:order}:
    \begin{enumerate}[label=(\arabic*)]

    \item\label{i:gb1} For every minimal linear dependency $\mathbf{a}\in \Z^s$
    with $\vert M_\mathbf{a}\vert =2k$ and every $k$-element submultiset $M\subseteq M_\mathbf{a}$ such that
    \begin{enumerate}[label=(\alph*),ref=(1)(\alph*)]
        \item\label{i:gb1a} $+F_{\min(\mathbf{a})}\notin M$, if $\mathbf{a}_{\min(\mathbf{a})}>0$, and, 
        \item\label{i:gb1b} $\underbrace{\{-F_{\min(\mathbf{a})},\ldots,-F_{\min(\mathbf{a})}\}}_{|\ba_{\min(\mathbf{a})}| \mbox{ times}}\not\subseteq M$, if $\mathbf{a}_{\min(\mathbf{a})}<0$:
    \end{enumerate}
    $$
            \prod_{F \in M} x_{F} -  \prod_{F \in M_{\ba}\setminus M} x_{\overline{F}}.
        $$
    \item\label{i:gb2} For every minimal linear dependency $\ba\in \Z^s$ with $\vert M_{\ba}\vert=2k+1$ and every $(k+1)$-element submultiset $M\subseteq M_{\ba}$:
    $$
          \prod_{F \in M} x_{F} -  z\prod_{F \in M_{\ba}\setminus M} x_{\overline{F}}.        
    $$
    \item\label{i:gb3} For every minimal linear dependency $\ba\in \Z^s$ such that $\vert M_{\ba}\vert=2k$ and every $(k+1)$-element submultiset $M$ that contains the smallest element $+F_{\min(\mathbf{a})}$ at least twice:
     $$
          \prod_{F \in M} x_{F} -  z^2\prod_{F \in M_{\ba} \setminus M} x_{\overline{F}}.
    $$
    \item\label{i:gb4} For any $F \in \{F_1, \dots, F_s\}$:
    $$x_{F^+} x_{F^-} - z^2.$$
        
    \end{enumerate}
    In each case, the binomial is written so that the initial term has positive sign.
\end{theorem}

In the following, a binomial that falls into one of the cases in the theorem will be referred to as a binomial of type \ref{i:gb1a}, \ref{i:gb1b}, \ref{i:gb2}, \ref{i:gb3} or \ref{i:gb4}.
\begin{proof}
    We only need to show that for any binomial $\mathbf{m}_1-\mathbf{m}_2\in I_{\P_A}$ one of $\bm_1$ or $\bm_2$ is divisible by a leading monomial of  one of the binomials above (see e.g.,~\cite{S1996}). Moreover, we may assume that $\mathbf{m}_1$ and $\mathbf{m}_2$ are coprime. 
 Let $\mathbf{p}, \mathbf{q}\in \R^{2s+1}$ be the exponent vector of $\bm_1$ and $\bm_2$, respectively (with the coordinates numbered from $0$ to $2s$ ordered as in \eqref{eq:order}).   
  
     If there exists $F\in\{F_1,\ldots,F_s\}$ such that $x_{F^+} x_{F^-} \mid \bm_i$ (for $i\in\{1,2\}$), then $\bm_i$ is divisible by the leading term of $x_{F^+} x_{F^-} - z^2$, which is a binomial of type~\ref{i:gb4}.
     Hence assume that $x_{F^+} x_{F^-}$ does neither divide $\bm_1$ nor $\bm_2$ for any column $F$ of $A$. 
  Let $\bb,\bc$ be the projections of the exponent vectors that are obtained by setting 
            
    $$
        \bb_\ell = \bp_{2\ell-1} - \bp_{2\ell} \qand \bc_\ell = \bq_{2\ell-1} - \bq_{2\ell} \qfor 1\leq \ell\leq s.
    $$
Note that, since neither $\bm_1$ nor $\bm_2$ are divisible by $x_{F^+}x_{F^-}$ for any column $F$ of $A$, we have $\bb_\ell\in\{\bp_{2\ell-1},-\bp_{2\ell}\}$, and similarly for $\bc_\ell$.
Since $\bm_1 - \bm_2 \in I_{\P_A}$, it follows that $\mathbf{b}-\mathbf{c}$ is a linear dependency for $A$.
Let $\ba$ be a minimal linear dependency for $A$ having the same sign pattern as $\bb-\bc$ and being componentwise smaller than or equal to $\bb-\bc$ in absolute value. We might have $\ba=\bb-\bc$. If $\mathbf{b}$ and $\mathbf{c}$ have a nonzero entry at the same position, those must have opposite signs, since $\bm_1$ and $\bm_2$ are coprime by assumption. This implies 
\[
|\ba_\ell|\leq |(\bb-\bc)_\ell|=|\bb_\ell|+|\bc_\ell| \qfor 1\leq \ell \leq s.
\]
 Moreover, restricted to the nonzero coordinates of $\bb$, the sign patterns of $\ba$ and $\bb$ coincide. Similarly, restricted to the nonzero coordinates of $\bc$,  the vectors $\ba$ and $\bc$ have opposite sign patterns. 
Choose vectors $\bb',\bc'\in \mathbb{N}^s$ such that $\vert \ba_\ell \vert =  \bb'_\ell + \bc'_\ell$, $ \bb'_\ell\leq \vert \bb_\ell\vert $ and $\bc'_\ell \leq \vert \bc_\ell\vert$ for $1\leq\ell\leq s$. Moreover, let $\overline{\bb'}$ be the vector with entries $\overline{\bb'}_\ell=\sign(\ba_\ell)\cdot\bb'_\ell$ for $1\leq \ell\leq s$ and $\overline{\bc'}= \overline{\bb'}-\ba$. 
We consider the multiset $M_\ba$ and its submultiset $M_{\overline{\bb'}}$. 
By construction, the monomial $\mathbf{m}_{\overline{\bb'}}$ and $\mathbf{m}_{\overline{\bc'}}$ divides $\bm_1$ and $\bm_2$, respectively. 
 Without loss of generality we can assume that $\bm_{\overline{\bc'}}\prec_{\rev} \bm_{\overline{\bb'}}$, otherwise consider the linear dependency $-\ba$, which swaps the roles of $\bm_1$ and $\bm_2$ and respectively the roles of $\bb$ and $\bc$. We consider several cases.

{\sf Case 1:} $|M_\ba|=2k+1$. \\
If $|M_{\overline{\bb'}}|\geq k+1$, then $\mathbf{m}_{M'}$ is the leading term of $\mathbf{m}_{M'}-z\mathbf{m}_{M_{\ba}\setminus M'}$ for any $M'\subseteq  M_{\overline{\bb'}}$ with $|M'|=k+1$. As  $\mathbf{m}_{M'}$  divides  $\mathbf{m}_{M_{\overline{\bb'}}}$, it also divides $\bm_1$. Hence, $\bm_1$ is divisible by a leading term of a binomial of type \ref{i:gb2}.

 {\sf Case 2:} $|M_\ba|=2k$. We need to distinguish further subcases.\\
{\sf Case 2(a):} $|M_{\overline{\bb'}}|=k$ and $\ba_{\min(\ba)}>0$. 
We consider the binomial $\mathbf{m}_{\overline{\bb'}}-\mathbf{m}_{\overline{\bc'}}$.
If $\overline{\bb'}_{\min(\ba)}=0$, then its leading term is $\mathbf{m}_{{\overline{\bb'}}}$, and since $\mathbf{m}_{\overline{\bb'}}$ divides $\bm_1$,  the claim follows. Note that this binomial is of type~\ref{i:gb1a}. 

{\sf Case 2(b):} $|M_{\overline{\bb'}}|=k$ and $\ba_{\min(\ba)}<0$. 
    In this case, since the leading term of the binomial $\bm_{\overline{\bb'}} - \bm_{\overline{\bc'}}$ is $\bm_{\overline{\bb'}}$ by assumption, we know that 
    $$
        \underbrace{\{-F_{\min(-\mathbf{a})},\ldots,-F_{\min(-\mathbf{a})}\}}_{|\ba_{\min(-\mathbf{a})}| \mbox{ times}} \not \subseteq M_{\overline{\bb'}}.
    $$
    Hence, $\bm_{\overline{\bb'}} - \bm_{\overline{\bc'}}$ falls into case~\ref{i:gb1b}, and it divides $\bm_1$.
    
{\sf Case 3:} $|M_{\overline{\bb'}}|= k+j$, for some $j>0$.
Consider the binomial $\bm_{\overline{\bb'}} - z^{2j} \mathbf{m}_{\overline{\bb'}-\ba}$.
Then the leading term is $\bm_{\overline{\bb'}}$, as it is not divisible by $z$. 
We will once again consider several subcases:

{\sf Case 3(a):} $\overline{\bb'_{\min(\ba)}}\geq 2$. 
In this case, consider the $(k+1)$-element submultiset $M' \subseteq M_{\overline{\bb'}}$ obtained by deleting the largest $j-1$ elements with respect to the chosen ordering on the columns of $A$. As $M'$ contains $+F_{\min(\ba)}$ at least twice, $\bm_{M'}$ is the leading term of the binomial $\bm_{M'}-z^2 \bm_{M_{\ba} \setminus M'}$ of type~\ref{i:gb3}, and it divides $\bm_1$.

{\sf Case 3(b):} $\overline{\bb'_{\min(\ba)}}= 1$.
In this case, consider the $k$-element submultiset $M' \subset M_{\overline{\bb'}}$ obtained by deleting the smallest $j$ elements with respect to the chosen ordering on the columns of $A$. As $M'$ does not contain $\pm F_{\min(\ba)}$, the monomial $\bm_{M'}$ is the leading term of the binomial $\bm_{M'}-\bm_{M_{\ba}\setminus M'}$ of type~\ref{i:gb1a}, and it divides $\bm_1$.

{\sf Case 3(c):} $\overline{\bb'_{\min(\ba)}}\leq 0$.
In this case, consider the $k$-element submultiset $M' \subset M_{\overline{\bb'}}$ obtained by deleting the smallest $j$ elements with respect to the chosen ordering on the columns of $A$. 
As  $-F_{\min(\ba)}$ was  deleted at least once, 
the monomial $\bm_{M'}$ is the leading term of the binomial $\bm_{M'}-\bm_{M_{\ba}\setminus M'}$ of type~\ref{i:gb1b}, and it divides $\bm_1$.
\end{proof}

\begin{remark}[\emph{Triangulations of symmetric (co)homology polytopes and topology}]\label{re:grobnerhomology}
 In the case of (co)homology polytopes, minimal linear dependencies correspond to (co)cycles in (co)homology. This yields a topological description of  triangulations of $\P_\Delta$ and $\P^\Delta$.
\end{remark}

\begin{remark}[\emph{The difference between totally unimodular and non-totally unimodular matrices}]
    We want to emphasize that dropping the totally unimodular condition on the underlying matrix $A$, only forces one completely new type of binomials in the \grobner basis of $I_{\P_A}$ to appear, namely case \ref{i:gb3} in~\cref{t:groebner_homology}. 
    This case does not occur in the totally unimodular case, that has been described in \cite{DJK2024}, since the kernel of any totally unimodular matrix can be generated by a set of $\{-1,0,1\}$-vectors (see e.g., \cite[Lemma 3.18]{Onn}).
    
    In particular,~\cref{t:groebner_homology} shows that 
    if every minimal linear dependency only uses $\{-1,0,1\}$-entries, $I_{\P_A}$ has a squarefree \grobner basis and hence, $\P_A$ admits a regular unimodular triangulation.

\end{remark}

\begin{figure}
\begin{subfigure}[b]{0.4 \textwidth}
\centering
          $$\partial_2^{\top} = \begin{pmatrix}
        1 & -1 & 0 & 1 & 0 & 0\\
        1 & 0 & -1 & 0 & 1 & 0 \\
        0 & 1 & -1 & 0 & 0 & 1 \\
        0 & 0 & 0 & 1 & -1 & 1
    \end{pmatrix}  $$     
\end{subfigure}
\begin{subfigure}{0.4\textwidth}
    \centering
    \includegraphics[width = 0.5\textwidth]{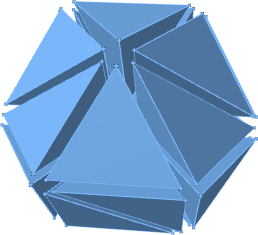}
    
\end{subfigure}
    \caption{A triangulation of a projection of $\P^{\Delta}$, where $\Delta$ is the boundary of a tetrahedron, and the corresponding coboundary map.}
\end{figure}

We want to point out that if the matrix $A$ has a minimal linear dependency $\ba$ with entries bigger than $1$ in absolute value, then the associated multiset $M_\ba$ is not a set. Consequently,  the Gröbner basis from~\cref{t:groebner_homology} might have leading terms that are not squarefree, coming from binomials of type~\ref{i:gb1b} and~\ref{i:gb3}. 
In particular, the associated triangulation of $\P_A$ is not necessarily unimodular in this situation. 
The following theorem shows that in a special situation, $\P_A$ does not have a unimodular triangulation at all.

\begin{theorem}\label{t:noSquarefreeGroebner}
    Let $A \in \Z^{m \times s}$ be a saturated matrix of rank $s-1$. Assume that (up to sign) $A$ has a unique minimal linear dependency $\ba\in \ker_\Z(A)$ with $\sum_{j=1}^n \ba_j = 2k$ and $\vert \ba_j\vert>1$ for some $1\leq j\leq s$. 
    Then $\P_A$ does not have any regular unimodular triangulation.
\end{theorem}  

\begin{proof}
In view of~\cref{t:sturmfels}, it is enough to prove that the toric ideal of $\P_A$ does not admit a squarefree Gröbner basis.
Since  $\P_{A'} = \P_A$ for any matrix $A'$ that~--~up to sign~-- has the same set of columns as $A$, without loss of generality, we can assume that $\ba_i\geq0$ for all $1\leq i\leq n$.

Assume by contradiction that there exists a term order $\prec$ with squarefree initial ideal $\initial_{\prec}(I_{\P_A})$. 

Let $\mathcal{G}$ be the \grobner basis from \Cref{t:groebner_homology} and let $J_\mathcal{G}$ be the set of \emph{all} monomials appearing in the binomials of $\mathcal{G}$. As $\mathcal{G}$ generates $I_{\P_A}$, it follows that $\initial_{\prec}(I_{\P_A})\subseteq \langle J_{\mathcal{G}}\rangle$. Moreover, since $\initial_{\prec}(I_{\P_A})$ is squarefree, we even have $\initial_{\prec}(I_{\P_A})\subseteq \langle J_{\mathcal{G}}\setminus \{z^2\}\rangle$. 
We will show the claim by constructing a binomial in $I_{\P_A}$ whose monomials are not squarefree and that has to belong to the Gr\"obner basis of 
$I_{\P_A}$ with respect to $\prec$.

Note that since $\ba_j\geq 2$, there exists a submultiset $M$ of $M_{\ba}$ with $|M|=k$ that contains $F_j$ at least twice. 
The \emph{critical} binomial is the following binomial in $I_{\P_A}$
\begin{equation}\label{eq:binom2}
\prod_{F\in M} x_{F^+} - z^2\prod_{F\in M_\ba\setminus M} x_{F^-}= x_{F_j^+}^2\bm_1 - z^2\bm_2
\end{equation}
(which is of type~\ref{i:gb3}). Since neither of its two monomials is squarefree, its leading term is not squarefree. Moreover, as $\bm_1$ and $\bm_2$ are only divisible by variables of the form $x_{F^+}$ and $x_{F^-}$, respectively, neither of them is divisible by $x_{F^+}x_{F^-}$ for any $F$. 
As (up to sign) $\ba$ is the unique linear dependency, it hence remains to check divisibility for monomials being part of binomials of degree $k$ arising from $\ba$. 
We distinguish two cases.

{\sf Case 1: $\init_{\prec}(x_{F_j^+}^2\bm_1 - z^2\bm_2)=z^2\bm_2$.}
Any monomial in $\langle J_{\mathcal{G}}\setminus \{z^2\}\rangle$ of degree at most $k$ that divides $z^2\bm_2$ has to be of degree $k$. However, since no  monomial in $\langle J_{\mathcal{G}}\setminus \{z^2\}\rangle$ of degree $k$ is divisible by $z$, none of these can divide 
 $z^2\bm_2$.
    
{\sf Case 2: $\init_{\prec}(x_{F_j^+}^2\bm_1 - z^2\bm_2)=x_{F_j^+}^2 \bm_1$.} As in Case 1, monomials in $\langle J_{\mathcal{G}}\setminus \{z^2\}\rangle$ that divide $x_{F_j^+}\bm_1$ have to be of degree $k$. The only such squarefree monomial of degree $k$ is $x_{F_j^+} \bm_1$ which appears in the unique binomial $x_{F_j^+}\bm_1 - x_{F_j^-}\bm_2$ in $I_{\P_A}$.
If $\init_\prec(x_{F_j^+}\bm_1 - x_{F_j^-}\bm_2 )=x_{F_j^-} \bm_2$, it follows that  $x_{F_j^+}^2\bm_1 - z^2\bm_2$ has to be part of the Gr\"obner basis of $\initial_{\prec}(I_{\P_A})$. 
If $\init_\prec(x_{F_j^+}\bm_1 - x_{F_j^-}\bm_2 )=x_{F_j^+}\bm_1$, then by the same argument applied to $x_{F_j^+}\bm_1 - x_{F_j^-}\bm_2$ and $x_{F_j^-}^2 \bm_2 - z^2 \bm_1$, it follows that the latter binomial  has to be part of the Gr\"obner basis of $\initial_{\prec}(I_{\P_A})$.
\end{proof}

    \begin{example}[\emph{Symmetric homology polytopes without regular unimodular triangulations}]\label{ex:nonunimodulartriangulation}
        Consider the $2$-dimensional simplicial complex $\B$ from \cref{ex:bjorner}.
       We know that $H_2(\B; \Z) \cong  \Z$ with generator
        $$
            \ba = 2[123] + [125] + [126] + [134] + [136] + [145] + [234] + [235] + [246] + [356] + [456].
        $$
       Note that $\ba$ has a coefficient $2$ in the coordinate corresponding to $[123]$, and the sum of its entries is $12$. In particular, the top boundary map $\partial_2$ of $\B$ satisfies all of the assumptions in~\cref{t:noSquarefreeGroebner} and hence we conclude $\P_\B$ does not have a regular unimodular triangulation.  
    \end{example}

\section{Crosspolytopes and the integer decomposition property}
\label{s:IDP}

    The goal of this section is twofold. In the first part, we classify all matrices that yield a crosspolytope (\Cref{l:cross_polytope}. In the second part, we ask when a general centrally symmetric polytope, given as $\P_A$, is spanning or IDP. We provide a complete answer for the first property in terms of the Smith normal form  of $A$ and provide a necessary criterion for the second property (\Cref{p:notidp}).  For both questions, we also consider the special situation of symmetric (co)homology polytopes. 

\subsection{Crosspolytopes}
We start by classifying matrices that yield crosspolytopes.
    
    \begin{lemma}\label{l:cross_polytope}
        Let $A\in \Z^{m\times s}$ such that $\P_A$ has $2s$ vertices. 
        Then $\P_A$ is an $s$-dimensional crosspolytope if and only if $A$ has rank $s$.
    \end{lemma}
    \begin{proof}
    First assume that $A$ has rank $s$. 
        We proceed by induction on $s$. If $s = 1$, then $\P_A$ is a line segment, and as such a $1$-dimensional crosspolytope.

        Now assume the result is true for $s < t$. Let $A \in \Z^{m \times t}$ be a matrix of rank $t$. Consider a fixed column $F$ of $A$ and let $B$ be the matrix obtained from $A$ by deleting $F$. Since $B$ has rank $t - 1$, it follows that $\P_B$ is a $(t-1)$-dimensional crosspolytope, by induction. Since adding $F$ to $B$ increases the rank by $1$, the vertices of $\P_A$ given by $F$ and $-F$ do not lie in the affine span of $\P_B$. As $\P_B$ contains the origin, they have to lie on opposite sides of the affine hull of $\P_B$. Consequently, $\P_{A}$ is a bipyramid over $\P_B$ with apices $F$ and $-F$. As such $\P_{A}$ is an $s$-dimensional crosspolytope. 

    The other direction is immediate since an $s$-dimensional  crosspolytope has $2s$ vertices.
        \end{proof}
In the context of symmetric homology polytopes~\cref{l:cross_polytope} can be restated as follows.
    \begin{corollary}\label{c:crosspolytope}
    Let $\Delta$ be a $d$-dimensional simplicial complex with $f_d(\Delta)=s$.  Then $\P_\Delta$ is  an $s$-dimensional crosspolytope if and only if $H_d(\Delta; \Z) = 0$.   
    \end{corollary}

This corollary, in particular, means that all symmetric homology polytopes coming from a dimensional simplicial complex with trivial top homology and a fixed number of facets have the same combinatorial type. Namely,  the one of a crosspolytope, which is as simple as it could get. 

    \begin{proof}
        The claim follows from \Cref{l:cross_polytope} by observing that the $\rank \partial_d=s$ if and only if $H_d(\Delta;\Z)=0$.
    \end{proof}

    One might ask under which conditions symmetric cohomology polytopes are combinatorially equivalent to crosspolytopes. Since the top boundary map of a simplicial complex is never injective and hence always has nontrivial linear dependencies, this is never the case.combinatorially equivalent to a crosspolytope.

\subsection{The integer decomposition property}
We start with  a complete characterization of when centrally symmetric polytopes, given as $\P_A$, are spanning. As such, this also yields a necessary criterion for $\P_A$ to be IDP.
    \begin{proposition}\label{p:notidp}
        Let $A\in \Z^{m\times s}$ be a saturated matrix of rank $r$ with elementary divisors $\alpha_1\leq  \dots\leq \alpha_r$.  Then $\P_A$ is spanning if and only if $\alpha_r=1$.
        In particular, if $\alpha_r > 1$, then $\P_A$ is not IDP.
    \end{proposition}
 
    \begin{proof}
    Let $D = S^{-1} A T^{-1}$ be the SNF of $A$, where $S\in \Z^{m\times m}$ and $T\in \Z^{s\times s}$ are weakly unimodular matrices.
    By construction, the columns of $A$ and $D$ span isomorphic lattices, and this isomorphism extends to an isomorphism between $\aff([A|-A])$ and $\aff([D|-D])$. Consequently, $\P_A$ is spanning if and only if every lattice point in $\aff([D|-D])$ can be written as an integer combination of columns of $D$. Since the $r$\textsuperscript{th} unit vector belongs to  $\aff([D|-D])$  but not to the lattice $\im_\Z D$ spanned by $D$ (as $\alpha_r>1$), it follows that $\P_A$ is not spanning and, hence, not IDP.
    %
    We also note that the same argument proves that $\P_A$ is spanning if $\alpha_r=1$.
    \end{proof}

    In the context of symmetric (co)homology polytopes, \Cref{p:notidp} can be interpreted topologically as follows.

    \begin{corollary} \label{c:notidphomology}
         Let $\Delta$ be a $d$-dimensional simplicial complex such that $H_{d-1}(\Delta; \Z)$ has torsion. Then $\P_\Delta$ and $\P^\Delta$ are not IDP.
     \end{corollary}

     \begin{proof}
        By~\cref{t:munkressnf},  $\partial$ (and hence, $\partial^\top$) has an elementary divisor $\alpha_r > 1$ if and only if $H_{d - 1}(\Delta; \Z)$ has torsion. The result then follows directly by~\cref{p:notidp}.
     \end{proof}

Positively stated, the previous corollary says that for a $d$-dimensional complex $\Delta$, if $\P_\Delta$ or $\P^\Delta$ is IDP, then, necessarily, the $(d-1)$\textsuperscript{st} homology group of $\Delta$ has to be torsion-free. 
    The next example shows that the converse statement is not true in general.

    \begin{example}[\emph{A non-IDP symmetric homology polytope}]
    \label{ex:bjornernotidp}
    Let $\B$ be the $2$-dimensional simplicial complex from \cref{ex:bjorner}. Since $H_1(\B; \Z) = 0$, we cannot apply~\cref{c:notidphomology}. Using Sage~\cite{sagemath} we compute that the $h^\ast$-polynomial of $\P_{\B}$ is given by 
    $$
        t^{10} + 12t^9 + 67t^8 + 232t^7 + {\color{red} 562t^6 +  1276t^5} + 562t^4 + 232t^3 + 67t^2 + 12t + 1,
    $$
    while the numerator polynomial of the Hilbert series of $K[\P_{\B}]$ equals
    $$
        t^{10} + 12t^9 + 67t^8 + 232t^7 + {\color{red}814t^6 + 1024t^5} + 562t^4 + 232t^3 + 67t^2 + 12t + 1.
    $$
    Since these two polynomials are different, $\P_\B$ is not IDP (see the last paragraph before \Cref{s:basic}). An example of a lattice point in the $5$\textsuperscript{th} dilate of $\P_\B$ that breaks the IDP property is provided by:
    $$
        (-1, -1, 0, 1, 1, -1, 0, 0, 0, -1, -1, 0, 0, -1, 0) \in \R^{15},
    $$
    which corresponds to 
    $$
        -[12] - [13] + [15] + [16] -[23] -[34] -[35] - [46]\in C_1(\B)
    $$
    Moreover, a computation shows that there are exactly $252 = 1276 - 1024$ such points.
    \end{example}

We now give an example of a non-IDP reflexive symmetric homology polytope that is combinatorially equivalent to a crosspolytope, and where the point causing failure of the IDP property can be understood topologically.
    
    \begin{example}[\emph{Non IDP symmetric homology polytopes vs torsion of simplicial complexes}]\label{ex:projectiveplanenotidp}
        Let $\Delta_{\R\PP^2}$ be the triangulation of the real projective plane from~\cref{ex:projectiveplane}. Then since $H_1(\Delta_{\R\PP^2}; \Z) = \Z_2$ and $H_2(\Delta_{\R\PP^2}; \Z) = 0$, \cref{c:notidphomology,c:crosspolytope} directly imply that $\P_{\Delta_{\R\PP^2}}$ is a non-IDP  crosspolytope. This can also be seen by noticing that 
    $$
        \frac{1}{10} F_1 + \dots + \frac{1}{10} F_{10} = \frac{1}{5}\bf{1}
    $$
    where $F_i$ ranges over the columns of $\partial_2$ and $\bf{1}$ is the all ones vector in $\R^{15}$. In particular, we have ${\bf{1}} \in 5\P_{\Delta_{\R\PP^2}}$, but it is impossible to decompose the vector $\bf{1}$ as a sum of five lattice points of $\P_{\Delta_{\R\PP^2}}$.
    
    It should be noted that the lattice point that breaks the IDP property of $\P_{\Delta_{\R\PP^2}}$ corresponds to the torsion generator of $H_1(\Delta_{\R\PP^2}; \Z)$ provided in \eqref{eq:torsionGen}, where every facet occurs with coefficient $1$.
    \end{example}

    \begin{remark}[\emph{Regular unimodular triangulations vs squarefree initial ideals}]
        \cref{ex:projectiveplanenotidp} shows that $\P_{\Delta_{\R\PP^2}}$ is not IDP and therefore cannot have a regular unimodular triangulation with respect to the lattice $\Z^9$ (see e.g., \cite[Theorem 2.1]{FH2024}). 
        On the other hand, since $H_2(\Delta_{\R\PP^2}; \Z)=0$ it has a squarefree \grobner basis by~\cref{t:groebner_homology,re:grobnerhomology}. By \Cref{t:sturmfels}, $\P_{\Delta_{\R\PP^2}}$ hence has a regular unimodular triangulation with respect to the ambient lattice $\P_{\Delta_{\R\PP^2}}\cap\Z^9$. This examples showcases subtleties very good the dependence on the considered lattice when one speaks about unimodular triangulations.
    \end{remark}

    \section{Reflexivity}\label{s:reflexivity}
    The focus in this section lies in the question of when a centrally symmetric polytope, given as $\P_A$, is reflexive. In particular, we are aiming towards an answer just in terms of the matrix $A$. We do not want to provide an answer in full generality (see \cite[Lemma 2.11]{OH2014} for results in this direction), but will mostly restrict to the case that the polytope is a crosspolytope. 
    While in \Cref{s:grobner} saturated matrices were the right representation to be considered, to verify reflexivity knowing the vertices of the polytope is enough. Similarly to \cref{l:cross_polytope}, we will therefore require that the number of vertices of $\P_A$ equals the number of columns of $[A|-A]$  (cf., \Cref{l:latticepoints}) and we will call such a matrix \emph{reduced}. If $A$ is reduced, then, in particular, no pair of columns of $A$ is linearly dependent. 

In the following, we will first consider general centrally symmetric lattice polytopes and then apply our results to symmetric homology polytopes.

    \subsection{Reflexivity of general centrally symmetric lattice polytopes}\label{subs:ReflexivityGeneral}

    A general method for showing that a polytope is reflexive, is to prove that the matrix formed by the vertices of the polytope is totally unimodular (see~\cref{t:TUreflexive}). This approach has successfully been applied for  symmetric edge polytopes~\cite{H2015} as well as their generalization to regular matroids~\cite{DJK2024}.
    Since we do not require our matrices to be totally unimodular, we will not be able to follow this approach. 
    
    Let $\{F_1,\ldots,F_n\}$ denote the set of columns of $A$.
    By the well-known duality between the face lattices of a polytope and its polar (see e.g., \cite{Ziegler}), vertices of the polar polytope $\P_A^\lor$ correspond to facets of $\P_A$ as follows: 
  Given a  facet $\F$ of $\P_A$, let $A_\F$ denote the submatrix of $[A|-A]$ consisting of the columns of $[A|-A]$ corresponding to the vertices in $\F$. Moreover, let $\overline{A}_\F$ denote the complimentary submatrix to $[A_\F| -A_\F]$, i.e., the matrix consisting of the columns of $[A|-A]$ that do not belong to $[A_\F| -A_\F]$. 
    The vertex of $\P_A^\lor$ corresponding to $\F$ is the unique solution $\bv_\F\in\aff(\P_A)$ of the linear system of inequalities
        \begin{equation*}
        A_\F^\top\bx =\bf{1} \quad \text{such that } \overline{A}_\F^\top\bx< 1.
        \end{equation*}
        
    By definition, the polytope $\P_A$ is reflexive if and only if $\bv_\F$ is integral for every facet $\F$ of $\P_A$. 
    
    In this case, we must have $-A_\F^\top \bv_\F = -1$ and
    $ \overline{A}_\F^\top\bv_\F=0$, where the latter follows from $A$ and $\bv_\F$ being integral and $[A|-A]$ containing the negative of each column.
    Summarizing this discussion, it follows that $\P_A$ is reflexive if and only if, for every facet $\F$ of $\P_A$, 
  \begin{equation}\label{eq:reflexive}
        A^\top \bx = \bb^\F,\qquad \text{where} \qquad  \bb_i^{\F}  = \begin{cases}
            1,  & \quad \text{if $F_i^+\in \F$}\\
            -1, & \quad \text{if $F_i^-\in \F$}\\
            0, & \qotherwise,
        \end{cases}
    \end{equation}
    has an integral solution contained in the affine hull of $\P_A$.


    Using this criterion however requires knowing (the vertices of) all facets of $\P_A$, for which our description in \Cref{t:facets} is not precise enough. 
    Therefore a complete characterization of  reflexivity for $\P_A$ seems to be extremely challenging and out of reach.
    In the following we will focus on the case where we have a good understanding of the facets. 
    Namely when $A\in \Z^{m\times s}$ has full rank $s$, which corresponds to $\P_A$ being a crosspolytope.

    For verifying whether \eqref{eq:reflexive} has an integral  solution we will use the following characterizations of solutions to integer equations. For more details on such criteria see~\cite{S1986}.
    
    \begin{theorem}{\cite[Theorem 4.1, p. 51]{S1986}}\label{t:schrijver}
    Let $A\in \Z^{m\times s}$ be an integral matrix and $\bb\in \Z^s$.
    \begin{enumerate}[label = (\arabic*)]
        \item \label{i:schrijver1}\cite[Corollary 4.1a]{S1986}  The system $A^\top \bx=\bb$ has an integral solution $\bx$ if and only if $\by^\top\bb$ is an integer for all $\by \in \Q^s$ such that $A \by$ is integral.   
        \item \label{i:schrijver2} \cite[p. 51]{S1986} 
        If $A$ has 
        rank $s$, then $A^\top \bx=\bb$ has an integral solution $\bx$ if and only if the greatest common divisor of the $s$-minors of $A$ divides all $s$-minors of $[A^\top | \bb]$.
        \end{enumerate}        
    \end{theorem}

    Note that~\cref{t:schrijver} \ref{i:schrijver2} can be restated in terms of the elementary divisors of $A$. This perspective is a useful variation of~\cite[Lemma 2.11]{OH2014}, which basically says that if all  nonzero  $m$-minors of an $(m\times s)$-matrix  $A$ of rank $m$ are equal to $\pm1$, then the polytope $\P_A$ is reflexive. As we will soon see, the language of elementary divisors allows us to assume that the greatest common divisor  of minors is equal to $1$.

    We will use \Cref{t:schrijver} to investigate reflexivity of $\P_A$ if $A$ has full column rank. We start with the case that the elementary divisors are all equal to $1$. 

      \begin{lemma}\label{t:reflexiveCrossPolytope}
Let $A \in \Z^{m \times s}$ be a matrix  
of rank $s$ with elementary divisors $\alpha_1=\alpha_2=\cdots=\alpha_s=1$. Then $\P_A$ is reflexive.
    \end{lemma}
    \begin{proof}
    We restrict to the case that $\P_A$ is full-dimensional. 
    Otherwise, as unimodular transformations preserve reflexivity and elementary divisors, we can apply a unimodular transformation $\psi$ to $\P_A$, such that $\psi(\P_A)=\P_B$, where $\P_B$ is a full-dimensional polytope. 
        By~\cref{l:cross_polytope}  $\P_A$ is an $s$-dimensional crosspolytope. Since the elementary divisors are all equal to $1$, it follows that the greatest common divisor of the $s$-minors of $A$ equals $1$. (Note that this is true for the greatest common divisor of the $i$-minors for any $1\leq i\leq s$.) 
        As, trivially, $1$ divides any number,  $1$ divides the $s$-minors of 
         $[A^\top | \bb]$.  By~\cref{t:schrijver} \ref{i:schrijver2}, we conclude that $A^\top \bx=\bb$  has an integral solution for every $\bb\in \Z^{s}$.
        In particular, choosing for $\bb$ the vector $\bb^\F$ (for a facet $\F$ of $\P_A$) defined in \eqref{eq:reflexive}, it follows that for every facet $\F$ of $\P_A$ \eqref{eq:reflexive} has an integral solution. 
    \end{proof}

    We now turn to the general situation that the matrix  $A\in \Z^{m\times s}$ has elementary divisors $1=\alpha_1=\cdots=\alpha_t<\alpha_{t+1}<\cdots <\alpha_r$, where $\alpha_i$ divides $\alpha_{i+1}$.
Moreover, let $S\in \Z^{m\times m}$ and $T\in \Z^{s\times s}$ weakly unimodular such that the SNF of $A$ is given by $D=S^{-1}AT^{-1}$. 
By the assumption on the elementary divisors we have the canonical isomorphism
\begin{equation}\label{eq:isoCylclic}
(\im_\R D\cap \Z^m)/\im_\Z D\cong \underbrace{\{\mathbf{0}\}}_{\in \Z^{t}}\times \Z_{\alpha_{t+1}}\times \cdots \times \Z_{\alpha_r}\times \underbrace{\{\mathbf{0}\}}_{\in \Z^{m-r}}\cong \Z_{\alpha_{t+1}}\times \cdots \times \Z_{\alpha_r}.
\end{equation}

Moreover, as $D$ is the SNF of $A$, there is an isomorphism between $\im_\R D\cap \Z^m$ and  $\im_\R A\cap \Z^m $ which is given by 
\begin{equation}\label{eq:isoLattice}
\textbf{x}=D\cdot \textbf{v}\mapsto S\cdot D\cdot \textbf{v}= A(T^{-1}\textbf{v})=A(T^{-1}\widetilde{\textbf{x}}).
\end{equation}
where $\textbf{v}\in \Z^s$, $\textbf{x}=(x_1,\ldots,x_r,0,\ldots,0)$ and $\widetilde{\textbf{x}}=(x_1/\alpha_1,\ldots,x_{r-1}/\alpha_{r-1},x_{r}/\alpha_r,0\ldots,0$.
Since $T$ is weakly unimodular, we have $T^{-1}\widetilde{\textbf{x}}\in \Z^s$ if and only if $\widetilde{\textbf{x}}\in \Z^s$ if and only if $\textbf{x}\in \im_\Z D$. Consequently, the above isomorphism restricts to an isomorphism from $\im_\Z D$ to $\im_\Z A$.
Combining \eqref{eq:isoCylclic} and \eqref{eq:isoLattice} we have an isomorphism 
\begin{equation}\label{eq:isoCQuotient}
\Phi_A:\;\Z_{\alpha_{t+1}}\times\cdots \times\Z_{\alpha_r}\to (\im_\R A\cap\Z^m)/\im_\Z A,
\end{equation}
where $\Phi_A(\ell_{t+1},\ldots,\ell_{t_r})= S\cdot \left(\ell_{t+1}\cdot \mathbf{e}_{t+1}+\cdots+\ell_r\cdot \mathbf{e}_{r}\right)+\im_\Z A$ for $0\leq \ell_j\leq \alpha_j-1$ for $t+1\leq j\leq r$.  
 Let $\textbf{w}_{A,j}\coloneqq S\mathbf{e}_{t+j}\in \Z^m$ for $1\leq j\leq r-t$. Observe that, for   $1\leq j\leq r-t$
\begin{equation}\label{eq:wAndv}
\Phi_A(\mathbf{e}_j)=\textbf{w}_{A,j}+\im_\Z A=S\mathbf{e}_{t+j}+\im_\Z A=SD\frac{1}{\alpha_{t+j}}\mathbf{e}_{t+j}+\im_\Z A=A(T^{-1}\frac{1}{\alpha_{t+j}}\mathbf{e}_{t+j})+\im_\Z A.
\end{equation}
We note that in \eqref{eq:wAndv}, we use $\mathbf{e}_j$ and $\mathbf{e}_{t+j}$ for standard unit vectors in different ambient spaces. 
We set $\textbf{v}_{A,j}=T^{-1}\frac{1}{\alpha_{t+j}}\mathbf{e}_{t+j}$ for $1\leq j\leq r-t$.

The next statement says that  if $A$ has full column rank, then the reflexivity of $\P_A$ only depends on the signed sums of the entries of the vectors $\textbf{v}_{A,j}$.


    \begin{theorem}\label{t:reflexiveCriterionCrossPolytope}
        Let $A\in \Z^{m\times s}$ be a matrix with elementary divisors $1=\alpha_1=\cdots=\alpha_t<\alpha_{t+1}<\cdots< \alpha_s$. 
        Then $\P_A$ is reflexive if and only if $\bv_{A,j}^\top \bb$ is an integer for all $\bb\in \{-1,1\}^s$ and all $1\leq j\leq s-t$.
    \end{theorem}

\begin{proof}
        By~\cref{l:cross_polytope} we know that $\P_A$ is combinatorially equivalent to a crosspolytope. Combining this fact with \eqref{eq:reflexive} it follows that $\P_A$ is reflexive if and only if $A^\top \bx=\bb$ has an integral  solution for all $\bb\in \{-1,1\}^s$.
        By~\cref{t:schrijver}~\ref{i:schrijver1} this is equivalent to  $\by^\top \bb$ being an integer for every $\by\in \Q^s$ such that $A\by$ is an integral vector. 
        
        Now assume that $\P_A$ is reflexive. Since $\textbf{w}_{A,j}\in \Z^m$ and $A\textbf{v}_{A,j}-\textbf{w}_{A,j}\in \im_\Z A\subseteq \Z^m$ by \eqref{eq:wAndv}, we conclude that $A\bv_{A,j}\in \Z^m$. As $\textbf{v}_{A,j}\in \Q^s$ by definition, the previous paragraph implies that $\textbf{v}_{A,j}^\top\bb\in \Z$ for all $\bb\in \{-1,1\}^s$.

        Conversely, assume that $\textbf{v}_{A,j}^\top\bb\in \Z$ for all $\bb\in \{-1,1,\}^s$ and $1\leq j\leq s-t$. Let $\by\in \Q^s$ such that $A\by\in \Z^m$. If $\by\in \Z^s$, then, trivially,  $\by^\top \bb$ for $\bb\in \{-1,1\}^s$. We hence assume that $\by\in \Q^s\setminus \Z^s$. Since $A$ has rank $s$, multiplication by $A$ is injective and it follows that $A\by \in (\im_\R A\cap\Z^m) \setminus \im_\Z A$. Using \eqref{eq:isoCQuotient}, there exists $(\ell_{t+1},\ldots,\ell_s)\in \Z_{\alpha_{t+1}}\times\cdots \times\Z_{\alpha_s}$ such that 
        \[
       S\cdot \left(\ell_{t+1}\cdot \mathbf{e}_{t+1}+\cdots+\ell_s\cdot \mathbf{e}_{s}\right)-A\by\in \im_\Z A.
       \]
       Hence, using \eqref{eq:wAndv}, there exists $\bz \in \Z^s$ such that
        \[
        A\bz=S\cdot\left(\ell_{t+1}\cdot \mathbf{e}_{t+1}+\cdots+\ell_s\cdot \mathbf{e}_{s}\right)-A\by=\sum_{j=1}^{s-t}\ell_{t+j} \cdot A\cdot \bv_{A,j}-A\by=A\left(\sum_{j=1}^{s-t}\ell_{t+j} \bv_{A,j}-\by\right).
        \]
        Again, using that the mulitplication by $A$ is injective, we conclude that $\bz=\sum_{j=1}^{s-t}\ell_{t+j} \bv_{A,j}-\by$. Since $\bz\in\Z^s$ and $\bv_{A,j}^\top \bb\in \Z$ for all $\bb\in \{-1,1\}^s$ and all $1\leq j\leq s-t$ by assumption, it follows that
    \[
    \by^\top\bb=\sum_{j=1}^{s-t}\ell_{t+j} \bv_{A,j}^\top\bb-\bz^\top\bb\in \Z,
    \]
        which finishes the proof.
%
 %
    \end{proof}

    \begin{remark}\label{rem:OtherVAs}
    We want to emphasize that  the proof of \Cref{t:reflexiveCriterionCrossPolytope} still works if we exchange the $\bw_{A,j}$  ($1\leq j\leq s-t$) by any $\bw_j\in (\im_\R A\cap \Z^m)$ ($1\leq j\leq s-t$) such that $\bw_1+\im_\Z A,\ldots, \bw_{s-t}+\im_\Z A$ generate $(\im_\R A\cap \Z^m)/\im_\Z A$.  In particular, in the statement of \Cref{t:reflexiveCriterionCrossPolytope} we can take $\ell_j \bv_{A,j}$ (instead of $\bv_{A,j}$) if  $1\leq\ell_j\leq \alpha_{t+j}-1$ does not divide $\alpha_{t+j}$ or even any representative within $\im_\Z A+\frac{\ell}{\alpha_{t+j}}T^{-1}\be_S$. This perspective will be helpful when translating \Cref{t:reflexiveCriterionCrossPolytope} into a more topological language.
    \end{remark}

\cref{t:reflexiveCriterionCrossPolytope} gives a criterion for when $\P_A$ is not reflexive. A closer look at its proof reveals that the values of the elementary divisors  can be interpreted as a measure of how \emph{far} $\P_A$ is from being reflexive. More precisely, we have the following statement:
    \begin{corollary}\label{c:multipleReflexive}
         Let $A\in \Z^{m\times s}$ be a matrix with elementary divisors $1=\alpha_1=\cdots=\alpha_t<\alpha_{t+1}<\cdots< \alpha_s$. 
        Then $\alpha_s\P_A^\lor$  is a lattice polytope.%
    \end{corollary}
    \begin{proof}
        For $\alpha_s=1$, the statement directly follows from \cref{t:reflexiveCrossPolytope}.
        
        Suppose $\alpha_s>1$. By the same reasoning as in the proof of \Cref{t:reflexiveCriterionCrossPolytope}, we need to show that $A^\top\bx= \alpha_s\cdot \bb$ has an integral solution for every $\bb\in \{-1,1\}^s$. 
        Since $A=SDT$,  this is equivalent to saying that $D^\top S^\top\bx= (T^{-1})^\top \cdot\alpha_s \cdot \bb$ has an integral solution for all $\bb\in \{-1,1\}^s$. 
        Note that since $S$ is weakly unimodular, finding a solution for $D^\top S^\top\bx= (T^{-1})^\top\cdot\alpha_s \cdot \bb$ is equivalent to finding a solution for 
        $D^\top \by= (T^{-1})^\top\cdot  \alpha_s \cdot \bb$, as we can set $\by=(S^{-1})^\top \bx$.
        The solution to the latter system is  $\widetilde{\by} = (\alpha_s\cdot D^{-1}) (T^{-1}) \bb$, where $(\alpha_s\cdot D^{-1})$ is an integer matrix, because $\alpha_i$ divides $\alpha_s$ for all $i$.
        Since $T$ is weakly unimodular, its inverse  $T^{-1}$ is an integral matrix, 
        which implies that  $(\alpha_sD^{-1})(T^{-1})^\top \cdot \bb= \widetilde{\by}$ is integral.
%
    \end{proof}
We want to remark that the key point underlying the proof of~\cref{c:multipleReflexive} is that since $\alpha_s$ is the least common multiple of the elementary divisors of $A$, the $\alpha_s$\textsuperscript{th} multiple of any integral vector in $\im_\R A$ already belongs to $\im_\Z A$.

We finish this subsection by saying a few words about what happens if the polytope $\P_A$ is \textbf{not} a crosspolytope. 
To simplify notation, given any subset $B$ of the columns of $[A|-A]$, we denote by $\widetilde{B}$ the submatrix of $A$ consisting of the columns $F$ such that $\{F,-F\}\cap B\neq \emptyset$. One can think of $\widetilde{B}$ as the  \emph{unsigned} version of $B$. 
We say that a crosspolytope has a \emph{reflexive} facet $F$, if $F$ has a supporting hyperplane of the form $\{\bx\in\R^m~:~\omega^\top \bx= 1\}$ for some $\omega\in \Z^m$. 
    
    The next statement more or less follows by definition.
    \begin{proposition}\label{p:smallerproblem}
    Let $A\in \Z^{m\times s}$ be of rank $r$. The following are equivalent:
    \begin{enumerate}[label = (\arabic*)]
        \item\label{i:smallerproblem1} $\P_A$ is reflexive.
        \item\label{i:smallerproblem2} For every facet $\F$ of $\P_A$, there exists a linearly independent subset $B$ of the vertices of $\F$ with $|B|=r$ such that the vertices of $B$ form a reflexive facet of the $r$-dimensional crosspolytope $\P_{\widetilde{B}}$.
        \item \label{i:smallerproblem3} For every facet $\F$ of $\P_A$, and all linearly independent  subsets $B$ of the vertices of $\F$ with $|B|=r$ the vertices of $B$ form a reflexive facet of the $r$-dimensional crosspolytope $\P_{\widetilde{B}}$.
    \end{enumerate}
    In particular, if for every facet $\F$ of $\P_A$, there exists a linearly independent subset $B$ of the vertices of $\F$ with $|B|=r$ such that the $r$-dimensional crosspolytope $\P_{\widetilde{B}}$ is reflexive, then $\P_A$ is reflexive.
    \end{proposition}
Note that the crucial condition in \ref{i:smallerproblem2} and \ref{i:smallerproblem3} is that the polytope $\P_B$ has a specific reflexive facet. Since the matrix $B$ is assumed to have rank $r$, the polytope $\P_B$ has to be a crosspolytope by \Cref{l:cross_polytope}. Moreover, in \ref{i:smallerproblem2}, one does not need to require $B$ to be linearly independent since this is already enforced by $\P_{\widetilde{B}}$ being $r$-dimensional. However, we included it in the statement to make the analogy to \ref{i:smallerproblem3} more apparent.
    
    \begin{proof}
    We first show that \ref{i:smallerproblem1} implies \ref{i:smallerproblem3}. 
    Therefore assume that $\P_A$ is reflexive. 
    
    If $r=s$, then, by \Cref{l:cross_polytope}, $\P_A$ is a crosspolytope itself, and for every facet $\F$ of $\P_A$, the only linearly independent subset $B$ of cardinality $r$ is the set of vertices of $\F$ itself.  The claim is now immediate, since $\P_A=\P_{\widetilde{B}}$ in this case.
    
    Next, suppose $r<s$. Let $\F$ be a facet of $\P_A$. Since, by assumption,  $\P_A$ is reflexive, there exists $\omega\in\Z^m$ such that 
    \[
    \F=\P_A\cap \{\bx\in\R^m~:~\omega^\top \bx= 1\}\qquad \mbox{and} \qquad \P_A\subseteq \{\bx\in\R^m~:~\omega^\top \bx\leq 1\}.
    \]
    Let $B$ be a linearly independent subset of the vertices of $\F$ with $|B|=r$. By \Cref{l:cross_polytope}, $\P_{\widetilde{B}}$ is an $r$-dimensional crosspolytope that is contained in $\P_A$. Moreover, by construction, $\conv(B)$ is a facet of $\P_{\widetilde{B}}$ with supporting hyperplane $\{\bx\in\R^m~:~\omega^\top \bx= 1\}$. 
   
    That \ref{i:smallerproblem3} implies \ref{i:smallerproblem2} is immediate so that is remains to prove that \ref{i:smallerproblem2} implies \ref{i:smallerproblem1}. 
Let $\F$ be a facet of $\P_A$. As \ref{i:smallerproblem2} holds, there exists a subset $B$ of the vertices of $\F$ with $|B|=r$ such that $\P_{\widetilde{B}}$ is an $r$-dimensional crosspolytope with reflexive facet $\conv(B)$. It follows that there exists  $\omega\in\Z^m$ such that 
    \[
    \conv(B)=\P_B\cap \{\bx\in\R^m~:~\omega^\top \bx= 1\}\qquad \mbox{and} \qquad \P_B\subseteq \{\bx\in\R^m~:~\omega^\top \bx\leq 1\}.
    \]
    Since $\dim \F=\dim\conv(B)=r$ and $\conv(B)\subseteq \F$, the affine hulls of $\conv(B)$ and $\F$ coincide. In particular, $\{\bx\in\R^m~:~\omega^\top \bx= 1\}$ is a supporting hyperplane for $\F$. This shows that $\P_A$ is reflexive.
%
%
  %
    %
    \end{proof}

\subsection{Reflexivity of symmetric homology polytopes}
We will now apply the results from the previous subsection to symmetric homology polytopes.

Our main result is the following:

\begin{theorem}\label{t:reflexivitySymHomPol}
   Let $\Delta$ be a $d$-dimensional simplicial complex with  $f_{d}(\Delta)=s$ and $H_d(\Delta; \Z) = 0$. Then $\P_\Delta$ is reflexive if and only if one of the following cases occurs:
   \begin{itemize}
       \item[(1)] $H_{d-1}(\Delta; \Z) = 0$;
       \item[(2)] $d$ is even  and 
       $H_{d-1}(\Delta; \Z) \cong \Z^\ell \oplus \Z_{2}\oplus \cdots \oplus \Z_{2}$ for some $\ell\in \N$;
       \item[(3)] $d$ is odd, 
       $H_{d-1}(\Delta; \Z) \cong \Z^\ell \oplus \Z_{2}\oplus \cdots \oplus \Z_{2}$ for some $\ell\in \N$ and there exist $\bv_{\Delta,1},\ldots,\bv_{\Delta,r} \in \{\pm\frac{1}{2},0\}^s \setminus \{\mathbf{0}\}$ having an even number of nonzero entries, such that $\partial_d \bv_{\Delta,1} +\im_\Z\partial_d, \ldots,$ $\partial_d \bv_{\Delta,r} +\im_\Z\partial_d$ generate the torsion in $H_{d-1}(\Delta; \Z)$.
   \end{itemize}
\end{theorem}
We want to emphasize, that this theorem, in particular, says that $\P_\Delta$ is never reflexive, if the $(d-1)$\textsuperscript{st} homology group of $\Delta$ has $q$-torsion for some $q\geq 3$.

Its proof will directly follow from Lemmas \ref{t:torsionreflexivitycriterion}, \ref{p:torsionNOTreflexive} and \ref{p:TorsionEvenReflexive} and \Cref{rem:odddimension} below. 

    We first observe that in terms of symmetric homology polytopes, the construction of the $\bv_{A,j}$ in~\cref{t:reflexiveCriterionCrossPolytope} corresponds to taking torsion generators in homology. This interpretation enables us to interpret \Cref{t:reflexiveCrossPolytope} and \Cref{t:reflexiveCriterionCrossPolytope} as follows.

\begin{lemma}\label{t:torsionreflexivitycriterion}
        Let $\Delta$ be a $d$-dimensional simplicial complex with   $f_{d}(\Delta)=s$ and $H_d(\Delta; \Z) = 0$. 
        \begin{enumerate}[label = (\arabic*)]
            \item\label{i:torsionreflexive1} If $H_{d-1}(\Delta,\Z) = 0$, then $\P_\Delta$ is reflexive.
            \item\label{i:torsionreflexive2}  Let $H_{d-1}(\Delta; \Z) \cong \Z^\ell \oplus \Z_{q_1}\oplus \cdots \oplus \Z_{q_r}$ for some $\ell\in \N$, 
            $1< q_1$ and such that $q_i$ divides $q_{i+1}$ for $1\leq i\leq r-1$. 
            Let  $\bv_{\Delta,1},\ldots,\bv_{\Delta,r} \in \Q^s \setminus \Z^s$ be such that $\partial_d \bv_{\Delta,1} +\im_\Z\partial_d, \ldots,\partial_d \bv_{\Delta,r} +\im_\Z\partial_d$  generate the torsion of $H_{d-1}(\Delta; \Z)$. 
 Then $\P_\Delta$ is reflexive if and only if $\bv_{\Delta,j}^\top \bb \in \Z$ for every $\bb \in \{-1, 1\}^s$ and all $1\leq j\leq r$. 
        \end{enumerate}
    \end{lemma}
    \begin{proof}
        Let $\alpha_1\leq \ldots \leq \alpha_s$ be the elementary divisors of $\partial_d$. Since $H_d(\Delta; \Z) = 0$, the map $\partial_d$ has (full) rank $s$ and therefore we have $\alpha_s \geq 1 $. 

 If $H_{d-1}(\Delta; \Z) = 0$, then~\cref{t:munkressnf} implies that $\alpha_s = 1$ and thus $\P_\Delta$ is reflexive by \cref{t:reflexiveCrossPolytope}. This shows (1). 
            
            We now prove  (2). \cref{t:munkressnf} implies that $\alpha_{s-r+i} = q_i$ for $1\leq i\leq r$.
            Since $\partial_d \bv_{\Delta,1} +\im_\Z\partial_d, \ldots,\partial_d \bv_{\Delta,r} +\im_\Z\partial_d$  generate the torsion in $H_{d-1}(\Delta; \Z)$, the claim follows from \cref{t:reflexiveCriterionCrossPolytope} combined with \Cref{rem:OtherVAs}.
    \end{proof}

We now give explicit examples of how the results in this section can be applied to find families of reflexive (and non-reflexive) polytopes arising from combinatorial topology. More specific statements regarding special classes of simplicial complexes can be found in~\cref{s:examples}.
    
\begin{example}[\emph{Reflexivity and $3$-torsion}]\label{ex:torsion3}
        We consider the triangulation $\M_{\Z_3}^1$ of the Moore space of $\Z_3$ from \cref{ex:moore}.
        We have already seen that it has $H_2(\M_{\Z_3}^1)=0$ and $H_1(\M_{\Z_3}^1)=\Z_3$. 
        Computation shows that, when ordering the coordinates lexicographically, a torsion generator corresponds to
        $$
            \bv_{\partial_2} = \frac{1}{3} (1, 1, 1, 2, 2, 2, 2, 1, 1, 1, 1, 1, 1, 2, 2, 2, 1, 2, 1)^\top.
        $$
        By \cref{t:torsionreflexivitycriterion}, $\P_{\M_{\Z_3}^1}$ is reflexive if and only if $\bv_{\partial_2}^\top \bb\in \Z$ for all $\bb\in\{-1,1\}^{19}$. 
         
        While for example
        $$
            \bv_{\partial_2}^\top\mathbf{1}^\top = 9 \in \Z, 
        $$ 
        changing the last coordinate yields
        $$
            \bv_{\partial_2}^\top (1,\dots, 1,-1)^\top = 9-\frac{2}{3} = \frac{25}{3}\not \in \Z
        $$
         and \cref{t:torsionreflexivitycriterion} hence implies that $\P_{M_3^1}$ is not reflexive.
    \end{example}
    
It turns out that ~\cref{ex:torsion3} is part of a more general phenomenon which also states that in   \Cref{t:torsionreflexivitycriterion} \ref{i:torsionreflexive2} reflexivity can never occur if there is higher torsion. More precisely, the following statement holds:

\begin{lemma}
    \label{p:torsionNOTreflexive}
   Let $\Delta$ be a $d$-dimensional simplicial complex such that $H_d(\Delta; \Z) = 0$. If $\tilde H_{d-1}(\Delta; \Z)$ has $q$-torsion for some $q \geq 3$, then $\P_\Delta$ is not reflexive. 
\end{lemma}
\begin{proof}
    Let  $f_d(\Delta)=s$. Assume by contradiction that $\P_\Delta$ is reflexive, and let $\bv_\Delta \in \Q^s \setminus \Z^s$  such that $\partial_d \bv_\Delta+\im_{\Z}\partial_d$ generates one component of the $q$-torsion of $H_{d-1}(\Delta; \Z)$. 
    Then by~\cref{t:torsionreflexivitycriterion} \ref{i:torsionreflexive2}, we have $\bv_\Delta^\top\bb \in \Z$ for all $\bb\in \{-1,1\}^s$, In particular, $\bv_\Delta^\top \cdot \mathbf{1}=\sum_{\ell=1}^s(\bv_\Delta)_\ell\in \Z$. Since, by construction, $q \bv_\Delta \in \Z^s$ and $\bv_\Delta\notin\Z^s$, it is easy to see that there exist $1\leq i\leq s$ and $a\in \Z$ coprime to $q$ such that $(\bv_\Delta)_i=\frac{a}{q}$: Indeed, if such $a\in\Z$ did not exist, then we would have $q' \bv_\Delta \in \Z^s$ for some $q' < q$ which implies that $\partial_d\bv_\Delta+\im_\Z\partial_d$ would not generate the $q$-torsion. 
    Note that, by assumption, $\bv_\Delta^\top \bf{1} \in \Z$. 
    Setting $\bb=\mathbf{1}-2\be_i\in \{-1,1\}^s$, we get
    \[
    \bv_\Delta^\top \bb=\bv_\Delta^\top \mathbf{1}-2\cdot\frac{a}{q}\not\in \Z,
    \]
     where we use that $2\cdot\frac{a}{q}\not\in \Z$ as  $a$ and $q$ are coprime and $q \geq 3$.
\end{proof}

In view of~\cref{p:torsionNOTreflexive}, it might seem that given a $d$-dimensional simplicial complex $\Delta$ with torsion in homological degree $d-1$ and trivial top homology, the polytope $\P_\Delta$ is rarely reflexive. The next example shows however that $2$-torsion is not necessarily an obstruction to reflexivity. In particular, the condition in \Cref{t:torsionreflexivitycriterion} \ref{i:torsionreflexive2} might be fulfilled in this situation.

\begin{example}[\emph{Reflexivity and $2$-torsion: not everything fails to be reflexive}]\label{ex:2torsion}
Consider the triangulation $\Delta_{\R\PP^2}$ from \cref{ex:bjorner}. The $2$-torsion generator of $H_1(\Delta_{\R\PP^2};\Z)$ corresponds to the vector $\bv_\Delta = \frac{1}{2}\cdot\mathbf{1}^\top\in \R^{10}$. 
By \cref{t:torsionreflexivitycriterion} \ref{i:torsionreflexive2} the reflexivity of $\P_{\Delta_{\R\PP^2}}$ is equivalent to the statement that $\frac{1}{2}\sum_{i=1}^{10}b_i\in \Z$ for all $\bb\in \{-1,1\}^{10}$. Indeed, since we add an even number of $\pm \frac{1}{2}$, the result will always be an integer. So, despite $H_1(\Delta_{\R\PP^2};\Z)$ having  $2$-torsion, $\P_{\Delta_{\R\PP^2}}$ is reflexive.
\end{example}

Similar to \Cref{ex:torsion3}, the previous example is a special case of the following more general phenomenon which also shows that $2$-torsion might behave completely different than higher torsion.
    
\begin{lemma}\label{p:TorsionEvenReflexive}
    Let $\Delta$ be a $2k$-dimensional simplicial complex such that $H_{2k}(\Delta; \Z) = 0$. If $H_{2k-1}(\Delta; \Z)$ has only $2$-torsion and is free otherwise, then $\P_\Delta$ is reflexive.
\end{lemma}

\begin{proof}
    Let $f_d(\Delta)=s$. 
     Let $\bv_{\Delta,1},\ldots,\bv_{\Delta,r} \in \Q^s \setminus \Z^s$  such that $\partial_d \bv_{\Delta,1}+\im_{\Z}\partial_d,\ldots,\partial_d \bv_{\Delta,1}+\im_{\Z}\partial_d$ generate the $2$-torsion of $H_{d-1}(\Delta; \Z)$. Using \eqref{eq:wAndv}, we can assume that $\bv_{\Delta,j}\in \frac{1}{2}\Z^s$ for $1\leq j\leq r$. Moreover, since $\partial_d\bv\in \im_\Z A$ for any $\bv\in \Z^s$, we can even choose $\bv_{\Delta,j}\in \{0,\frac{1}{2}\}^s$ (see also \Cref{rem:OtherVAs}). 
    
    By~\cref{t:torsionreflexivitycriterion} \ref{i:torsionreflexive2}, $\P_\Delta$ is reflexive if and only if $\bv_{\Delta,j}^\top \bb \in \Z$ for every $\bb \in \{-1, 1\}$ and $1\leq j\leq r$. Since  $\bv_{\Delta,j}\in \{0,\frac{1}{2}\}^s$, this is equivalent to showing that the sets $S_j \coloneqq \{1\leq i\leq s \st (\bv_{\Delta,j})_i=\frac{1}{2}\}$ have even cardinality each. Fix $1\leq j\leq r$ and let $\sigma_1,\ldots,\sigma_t$ be the facets of $\Delta$ corresponding to the elements in $S_j$. 

    By definition of $\partial_d$, for a ridge $\tau\in \Delta$, the entry of $\partial_d (2\bv_{\Delta,j})$ corresponding to $\tau$ equals 
    
    $\partial_d (2\bv_{\Delta,j})_\tau=\#\{1\leq i\leq t~:~ \tau\subseteq \sigma_j\}$. Since, by assumption on $\bv_{\Delta,j}$, all entries of $\partial_d (2\bv_{\Delta,j})$ are even,  it follows that, in particular, so is $\partial_d (2\bv_{\Delta,j})_\tau$. 
    In particular, the sum of the entries $ \sum_{\substack{\tau \in \Delta, \\ |\tau| = 2k }} \partial_d (2\bv_{\Delta,j})_\tau$ is even. As every facet of $\Delta$ contains $2k + 1$ ridges, a double counting argument shows that this sum equals $t(2k+1)$.  We hence conclude that  $t$ is even.
\end{proof}

\begin{remark}\label{rem:odddimension}
    Note that the proof of \cref{p:TorsionEvenReflexive} shows that for a $d$-dimensional simplicial complex $\Delta$ with $H_{d}(\Delta;\Z)=0$ and $H_{d-1}(\Delta;\Z)$ having no $q$-torsion for $q\geq 3$, the polytope $\P_\Delta$ is reflexive if and only if the $2$-torsion can be generated by vectors with an even number of nonzero entries that all have absolute value $\frac{1}{2}$. 
The next example from Lutz's Manifold page~\cite{manifoldpage} shows that  for odd dimensional simplicial complexes, this is not always true. 
\end{remark}

    \begin{example}[\emph{A manifold with a non-reflexive symmetric homology polytope}] \label{ex:NonReflexiveManifold}
        Let $\Delta$ be the following minimal triangulation of the \emph{twisted sphere product} $S^2 \btimes S^1 $  (see \cite[Table 5]{Lutz}):
        \begin{align*}     
          \Delta = \tuple{ &1234,1235,1246,
          1257,1268,1278,1345,
          1456,1567,1679,1689,
          1789,2349,2359,2456,\\
          &2459,2567,2678,3458,
          3478,3479,3589,3678,
          3679,3689,4589,4789}.
        \end{align*}

        It can found in the list of $3$-dimensional manifolds on $9$ vertices from~\cite{manifoldpage} under the name \texttt{manifold\_3\_9\_989}. $\Delta$ is a $3$-dimensional pseudomanifold (even a manifold) with $ H_3(\Delta; \Z) = 0$,  $H_2(\Delta; \Z) = \Z_2$ and $f$-vector $(1, 9, 36, 54, 27)$. Setting $\bv_\Delta=\frac{1}{2}\cdot {\bf{1}} \in \Z^{27}$, one can show that  $\partial_2 \bv_\Delta+\im_\Z\partial_3$ generates $H_2(\Delta; \Z)$ (see also \cref{rem:odddimension}).
        In particular, since $
            \bv_\Delta^\top {\bf{1}} = \frac{27}{2} \not \in \Z$, ~\cref{t:torsionreflexivitycriterion} \ref{i:torsionreflexive2} implies that $\P_\Delta$ is not reflexive.
        \end{example}

We end this section with a sufficient criterion for a symmetric homology polytope to be reflexive, if the underlying simplicial complex is allowed to have non-vanishing top homology.

\begin{corollary}\label{c:gen}
    Let $\Delta$ be a $d$-dimensional simplicial complex. If $\P_\Gamma$ is reflexive for every simplicial spanning forest $\Gamma$ of $\Delta$, then $\P_\Delta$ is reflexive.
    In particular, if every simplicial spanning forest $\Gamma$ satisfies one of the conditions in \Cref{t:reflexivitySymHomPol}, then $\P_\Delta$ is reflexive.
\end{corollary}

\begin{proof}
    The first statement is an immediate consequence of \Cref{t:facets} and \Cref{p:smallerproblem}. For the second statement it suffices to observe that since every simplicial forest $\Gamma$ has vanishing top homology by definition, the polytope $\P_\Gamma$ is a crosspolytope and reflexivity of these is characterized in \Cref{t:reflexivitySymHomPol}.
\end{proof}

\section{Symmetric cohomology polytopes versus symmetric edge polytopes}\label{s:cohomology}


In this section, we turn our attention to symmetric cohomology polytopes. We are interested in the following question: When can a symmetric cohomology polytope be recognized as (generalized) symmetric edge polytope or a symmetric homology polytope? We will provide a complete answer in dimension $ 1$ and a partial answer in dimension $2$. 

The easiest case is that the underlying simplicial complex $\Delta$ is a graph. In this situation, the homology polytope $\P_\Delta$ is the well-known symmetric edge polytope of $\Delta$.
Relying on the fact that the boundary matrix $\partial_1$ is totally unimodular for graphs, using the results from~\cite{DJK2024}, we are able to characterize cohomology polytopes in this case, up to unimodular equivalence.

\begin{proposition}[\emph{The symmetric cohomology polytope of a graph}]\label{p:cohomologygraph}
    Let $n\geq 3$, and let $G$ be a connected graph on $n$ vertices. Then $\P^G$ is unimodularly equivalent to the symmetric edge polytope of a cycle on $n$ vertices.
\end{proposition}

\begin{proof}
         For a graph $G$, the boundary map $\partial_1$ is given by  its directed incidence matrix and it is well-known that the latter is a totally unimodular matrix. In particular, so is $\partial_1^\top$.
         By \Cref{p:regulargraphical}, $\partial_1^\top$, thus represents a regular matroid $M[\partial_1^\top]$. 
         Since  the kernel of $\partial_1^\top$ is generated by the vector $\bf{1}$, we conclude that $M[\partial_1^\top]$ has to be the uniform matroid on $n$ elements of rank $n - 1$. In particular, $M[\partial_1^\top]$ is isomorphic to the regular matroid realized by the boundary map of the cycle graph $C_n$ on $n$ vertices, see \Cref{ex:cycleMatroid}. The result then follows directly from~\cref{t:matroid_equivalence}.
\end{proof}

We now turn our attention to $2$-dimensional simplicial complexes that are obtained by coning over a graph. We need some preliminary notions from~\cite{DHK2011}.


A pure $d$-dimensional simplicial complex $\Delta$ is called a \emph{cycle complex} if there exists an ordering  $\sigma_0, \dots, \sigma_{k-1}$ of its facets   such that $\dim(\sigma_i \cap \sigma_{i+1})=d-1$ for all $0\leq i\leq k-1$  and $\sigma_i\cap \sigma_j=\emptyset$ if $j\neq i+1$ (where in both cases indices are taken mod $k$).

Note that any cycle complex is a pseudomanifold (see \Cref{s:prelim}). 
Any orientable cycle complex will be called a \emph{cylinder complex}, and any nonorientable cycle complex will be called  a \emph{\mobius complex}.

It was shown in \cite[Theorem 5.13]{DHK2011} that $2$-dimensional \mobius complexes are the unique obstruction for the second boundary map $\partial_2$ to be a totally unimodular matrix.

\begin{theorem}{\cite[Theorem 5.13]{DHK2011}}\label{t:TUdim2}
    Let $\Delta$ be a $2$-dimensional simplicial complex. Then, $\partial_2$ is totally unimodular if and only if $\Delta$ contains no $2$-dimensional \mobius complex.
\end{theorem}

A special case, in which \Cref{t:TUdim2} holds, occurs if $\Delta$ is the cone over a graph $G$. Here, we call $\Delta$ the \emph{cone over  $G$ with apex $v$}, if the facets of $\Delta$ are all sets of the form $e\cup \{v\}$ for $e\in E(G)$ and $v\notin V(G)$  a new vertex.

\begin{lemma}\label{l:coneTU}
    Let $G$ be a graph and  let $\Delta$ be a cone over $G$. Then the top boundary map $\partial_2$ of $\Delta$ is totally unimodular.
\end{lemma}

\begin{proof}
Since $\Delta$ is a $2$-dimensional cone, for any two facets $\sigma_1,\sigma_2$ of $\Delta$, $v\in \sigma_1\cap\sigma_2$.
In particular, $\Delta$ does not contain  a \mobius complex and the claim follows from  ~\cref{t:TUdim2}.
\end{proof}

\cref{l:coneTU} gives us a way of constructing symmetric cohomology polytopes that are \emph{reflexive} but not unimodularly equivalent to any symmetric edge polytope. We record this observation in the following result.

\begin{proposition}\label{p:cohomologycones}
    Let $G$ be a graph and  let $\Delta$ be a cone over $G$.
    Let $H$ be the graph consisting of the vertices and edges of $\Delta$. 
    Then, $\P^\Delta$ is unimodularly equivalent to a symmetric edge polytope of a graph if and only if $H$ is planar. In this case, $\P^\Delta \cong_u \P_{H^*}$, where $H^*$ is the planar dual of $H$. 
\end{proposition}
For the definition of the planar dual of a planar graph, we refer to \cite{diestel2005graphentheorie}.  

\begin{proof}
    By~\cref{l:coneTU}, the matrix $[\partial_2^\top]$ is totally unimodular and hence  $\P^\Delta$ can be interpreted as the generalized symmetric edge polytope $\P_{M[\partial_2^\top]}$ of the matroid $M[\partial_2^\top]$. By \Cref{l:matroid_duality}, the latter is isomorphic to the matroid $M[\partial_1]^*$. Since, by construction, $\partial_1$ is the incidence matrix of $H$, it follows that $M[\partial_1]$ is the graphic matroid of $H$. 

    Now assume that $H$ is planar. Then, by \cite[Theorem 5.2.2]{Oxley}, the matroid $M[\partial_1]^*$ is graphic as well. In particular, $M[\partial_1]^*$ is isomorphic to the graphic matroid $M[H^*]$ of $H^*$.  Since, by ~\cref{t:matroid_equivalence}, the polytopes  $\P_{M[\partial_2^\top]}$ and $\P_{M[H^*]}$ are unimodularly equivalent, the claim follows.

    Conversely, assume that $H$ is not planar. In this case, again by \cite[Theorem 5.2.2]{Oxley}, the matroid $M[\partial_1]^*$ is not graphic. Since, by \cite[Theorem 4.5 and Remark 3.2]{DJK2024}, the generalized symmetric edge polytope (with respect to a totally unimodular matrix) determines the underlying matroid (up to isomorphism), it follows that $\P_{M[\partial_2^\top]}$ (and hence $\P^\Delta$) is not unimodularly equivalent to any symmetric edge polytope of a graphic matroid. The claim follows.
\end{proof}

    One might ask whether, in the setting of ~\cref{p:cohomologycones}, the symmetric cohomology polytope of a graph is at least combinatorially equivalent to a symmetric edge polytope. 
    The next example shows that the answer to this question is no in general.

\begin{example}[\emph{Combinatorial equivalence also does not need to hold}]
    Let $\Delta$ be the cone over the complete bipartite graph $K_{3,3}$. Since $K_{3,3}$ is not planar, it follows from \cref{p:cohomologycones} that $\P^\Delta$ is not unimodularly equivalent to any symmetric edge polytope of a graph.

    Since $\Delta$ has $15$ edges and every facet contains only one free ridge, \Cref{l:latticepoints} implies that $\P^\Delta$ has  $30$ vertices. Moreover, as $H_2(\Delta; \Z) = 0$, we infer from \Cref{p:basic} that $\P^\Delta$ is $9$-dimensional. If $\P^\Delta$ was combinatorially equivalent to a symmetric edge polytope of a connected graph $G$, then $G$ would need to have $15$ edges and $10$ vertices. One can verify computationally using Sage, that  the $f$-vector of $\P^\Delta$ is not equal to the $f$-vector of any symmetric edge polytope of such a graph. Moreover, as the symmetric edge polytope of a disconnected graph has the same $f$-vector as the one of the $1$-sum of its connected components, it follows that $\P^\Delta$ is not combinatorially equivalent to any symmetric edge polytope.
\end{example}

\section{Symmetric (co)homology polytopes of (pseudo)manifolds and spheres}\label{s:examples}

The focus of this section lies on the  study of symmetric (co)homology polytopes of simplicial pseudomanifolds.
For the symmetric homology polytope of an orientable pseudomanifold with a fixed number of facets, it turns out that~--~up to unimodular equivalence~--~there are only two  such polytopes, depending on whether the pseudomanifold has boundary or not. For the symmetric cohomology polytope of an orientable pseudomanifold, the situation is more complicated: Still we will show, that every such polytope is  unimodularly equivalent to (a projection of) the symmetric edge polytope (of a whiskering) of the facet-ridge graph. 

We begin by stating a topological consequence of~\cref{p:TorsionEvenReflexive}.

\begin{corollary}\label{c:pseudomanifoldReflexive}
       Let $\Delta$ be a non-orientable $2k$-dimensional pseudomanifold without boundary. Then $\P_\Delta$ is reflexive.
 
    \end{corollary}
    \begin{proof}
        By \cite{homologyPseudomanifolds}, we know that  $H_{2k}(\Delta;\Z)=0$ and 
        $\tilde H_{2k-1}(\Delta;\Z)= \Z^\ell \oplus \Z_2$ for some $\ell\in \N$. The result now directly follows from \cref{p:TorsionEvenReflexive}.
    \end{proof}    

The restriction to even dimensional pseudomanifold in \Cref{c:pseudomanifoldReflexive} is indeed necessary, as \Cref{ex:NonReflexiveManifold} shows. 
    In this example, the obstruction to reflexivity is that $\Delta$ has $2$-torsion and that the torsion generator has an an odd number (27) of nonzero coefficients. A natural question that arises from our results is whether the reflexivity of $\P_\Delta$ depends only on the topology of $\Delta$. The next example together with the discussion above shows that this is not the case in general.

    \begin{example}[\emph{Reflexivity is not a topological property}]
        Let $\Delta'$  be the  simplicial complex obtained from $\Delta$ in  \cref{ex:NonReflexiveManifold} by a stellar subdivision of the facet $1234$, i.e., by replacing the facet $1234$ by $123x$, $124x$, $134x$ and $234x$, where $x$ is a new vertex. Note that $\Delta'$ has $30$ facets. Since $\Delta$ is a subdivision of $\Delta'$, both complexes are homeomorphic as topological spaces.
         One can show that the $2$-torsion of $\Delta'$ is generated by $\bv_{\Delta'} = \frac{1}{2} \cdot{\bf{1}} \in \Z^{30}$. Since  $\bv_\Delta^\top \cdot \bb\in \Z$ for all $\bb\in \{-1,1\}^{30}$, it follows from \Cref{l:cross_polytope} that  $\P_{\Delta'}$ is reflexive. 
        This example shows that although we are able to derive topological obstructions for reflexivity, this property does not only  depend on the topology, but also on the number of facets of $\Delta$. 
    \end{example}


    We already know from \Cref{t:tuboundary} that the top boundary map of any orientable pseudomanifold without boundary is totally unimodular. 
    As a consequence, it follows from \Cref{t:TUreflexive} that both, symmetric homology and cohomology polytopes, of orientable pseudomanifolds are reflexive. The next statement shows that in  this situation, these polytopes are even unimodularly equivalent to the symmetric edge polytope of a graph. 

\begin{theorem}\label{t:orientablepseudomanifolds}
    Let $\Delta$ be a $d$-dimensional orientable pseudomanifold without boundary and let $G$ be the facet-ridge graph of $\Delta$. Then 
    \begin{enumerate}[label = (\arabic*)]
        \item\label{i:last1} $\P_\Delta \cong_u \P_{C_s}$, where $f_d(\Delta)=s$, and 
        \item\label{i:last2} $\P^\Delta \cong_u \P_G$.
    \end{enumerate}
\end{theorem}
\begin{proof}
Since, by~\cref{t:tuboundary}, the boundary map $\partial_d$  is totally unimodular, $\P_\Delta$ and $\P^\Delta$ can be seen as  generalized symmetric edge polytopes. We prove each statement separately.

 (1) Since $\Delta$ is an orientable pseudomanifold without boundary, we know that $H_d(\Delta; \Z) = \Z$. This implies that there is a (up to scaling) unique linear dependency among the  columns of $\partial_d$. Moreover, since the facet-ridge graph $G$ of $\Delta$ is connected and every ridge lies in exactly two facets, it follows that this linear dependency uses every column. In particular, the matroid represented by $\partial$ is isomorphic to the graphic matroid of $C_s$ and the claim follows from \cref{t:matroid_equivalence}.

 (2)    Let $\sigma_1,\ldots,\sigma_s$ be the facets of $\Delta$. It follows from the proof of \ref{i:last1} that there is a linear dependency involving all columns of $\partial$ using only $+1$ and $-1$ as coefficients. In particular, there exist $\epsilon_i\in \{-1,+1\}$ for $1\leq i\leq s$ such that
        \[
        \varepsilon_1 \partial_d(\sigma_1)+ \cdots+\varepsilon_s \partial_d(\sigma_s) =0.
        \]
         Let $D(\varepsilon)$ be the $(s\times s)$-diagonal matrix with diagonal entries $\varepsilon_1,\ldots,\varepsilon_s$. 
         Since $\Delta$ is a pseudomanifold, every row of $\partial_d$ has exactly two non-zero entries and it follows that $\partial_d \cdot D(\varepsilon)$ equals the transpose of the incidence matrix of $G$, i.e., $D(\varepsilon)\cdot \partial_d^\top$ equals the incidence matrix of $G$. As $D(\varepsilon)$ is weakly unimodular,  $D(\varepsilon)\cdot \partial_d^\top$ and $\partial_d^\top$ define the same matroid. The claim follows from \cref{t:matroid_equivalence}.
\end{proof}

Before stating the analogue of~\cref{t:orientablepseudomanifolds} for orientable pseudomanifolds with boundary, we need to set some notation. Given a pseudomanifold with boundary $\Delta$, its facet-ridge graph $G = G(\Delta)$ and  a subset $A$ of  the vertices of $G$, the \emph{whiskering} of $G$ at $A$, denoted by $w(G, A)$, is the graph obtained by adding a leaf adjacent to each vertex in $A$. If $A$ is the complete set of vertices of $G$, this construction has appeared many times in the literature in commutative algebra~\cite{V1990}, in the study of piecewise-linear spheres~\cite{JTZ2021}, in graph theory~\cite{BC2018} and more.
Moreover, we will need the following result on pseudomanifolds with boundary, which we assume to be well-known but could not find in the literature. We therefore include a full proof.
\begin{lemma}\label{lem:homologyPseudo}
    Let $\Delta$ be a $d$-dimensional pseudomanifold with boundary. Then, $H_d(\Delta; \Z) = 0$.
\end{lemma}
\begin{proof}
    Assume by contradiction that $H_d(\Delta; \Z) \neq0$, i.e., $\Delta$ contains a $d$-homology cycle $\tau$. Let $\Gamma$ be the pure subcomplex of $\Delta$ whose facets are the facets of $\Delta$ appearing in $\tau$. As $\tau$ is a homology cycle, every ridge of $\Gamma$ has to be contained in at least two facets of $\Gamma$. Moreover, since $\Delta$ is a pseudomanifold, it follows that every ridge of $\Gamma$ has to be contained in exactly two facets of $\Gamma$, and there is no other facet of $\Delta$ containing it. In particular, this implies that $\Gamma$ is strongly connected and hence, that $\Gamma=\Delta$. This is a contradiction, since $\Delta$ has boundary.
\end{proof}


The analogous statement of ~\cref{t:orientablepseudomanifolds} for orientable pseudomanifolds with boundary is the following: 
\begin{theorem}\label{t:orientablepseudomanifoldswithboundary}
    Let $\Delta$ be a $d$-dimensional orientable pseudomanifold with boundary  and let $G=(V,E)$ be the facet-ridge graph of $\Delta$. Let $A$ be the set of vertices of $G$ corresponding to facets of $\Delta$ that contain a free  ridge. Then

    \begin{enumerate}[label = (\arabic*)]
        \item\label{i:orientablepseudomanifold1} $\P_\Delta$ is unimodularly equivalent to an $s$-dimensional crosspolytope, where $f_d(\Delta) = s $.
        \item $\P^\Delta \cong_u \pi_V (\P_{w(G, A)})$, where $\pi_V$ is the projection onto the coordinates corresponding to $V$.
    \end{enumerate}
\end{theorem}

\begin{proof}
    Since, by~\cref{t:tuboundary}, the boundary map $\partial_d$ is totally unimodular, we can interpret both $\P_\Delta$ and $\P^\Delta$ as  generalized symmetric edge polytopes.  First note  that,  by~\cref{l:latticepoints}, if a facet of $\Delta$ contains more than one free ridge, only a single free ridge becomes a vertex of $\P^\Delta$.  We prove each statement separately:
    
        (1) Follows directly by combining ~\cref{c:crosspolytope} with \Cref{lem:homologyPseudo}. 
        
        (2) Since $\Delta$ is an orientable pseudomanifold with boundary, we know $H_d(\Delta, \partial \Delta; \Z) = \Z$. In other words, the matrix $\partial_d^\ast$ that represents the top relative boundary map, has a single linear dependency (up to scaling), and the coefficients in this linear dependency are $+1, -1$. The same arguments as in the proof of \Cref{t:orientablepseudomanifolds} \ref{i:last2} applied to $\partial_d^\ast$ instead of $\partial_d$ show that $D(\epsilon)\cdot(\partial_d^\ast)^\top$ equals the incidence matrix of $G$, where $D(\epsilon)$ is a diagonal matrix with $+1$ and $-1$ on the diagonal. 
        %
        Let $\partial'_d$ be the matrix obtained from $D(\varepsilon) \cdot \partial_d^\top$ by adding one new row per boundary ridge, such that the only nonzero entry of each new row is $\pm 1$ in the column associated to the corresponding free ridge, where the sign is the opposite sign of the other nonzero entry in this column.

        By construction, the matrix $\partial'_d$ is the incidence matrix of the graph $w(G, A)$, and in particular, the symmetric cohomology polytope $\P^\Delta$ is obtained from the symmetric edge polytope $\P_{w(G, A)}$ by projecting it to the coordinates not  corresponding to the whiskers.
    \end{proof}

We note that in view of~\cref{t:orientablepseudomanifolds}, ~--~up to unimodular equivalence~--~the symmetric homology polytope of a sphere (or more generally, orientable pseudomanifold without boundary) only depends on the number of facets of the sphere. 
This can be seen as an analogue of~\cite[Theorem 4.6]{DKM2011}, where the authors show that the top critical group of a sphere depends only on the number of facets of the simplicial complex.  Moreover, we want to remark that it can be seen from the proof of \cref{t:orientablepseudomanifolds}~\ref{i:orientablepseudomanifold1}
that the statement is also true for non-orientable pseudomanifolds with boundary.

In what follows we will use the notion of  shellable simplicial complexes without defining it, since we only need it for the next corollary. 
The definition can be found for example in~\cite[Chapter 10]{B1995}.
In 2024, Yang~\cite[Theorem 1.2]{Y2024} showed that if $\Delta_1$ and  $\Delta_2$ are shellable simplicial complexes whose geometric realizations are  homeomorphic to spheres (or \emph{simplical spheres} for short), then $\Delta_1$ and $\Delta_2$ are combinatorial equivalent  if and only if their facet-ridge graphs $G(\Delta_1)$ and $G(\Delta_2)$ are isomorphic. In view of Yang's result and the results in~\cite{DJK2024}, we have the following corollary of~\cref{t:orientablepseudomanifolds}.

\begin{corollary}\label{c:distinguishspheres}
    Let $\Delta_1$ and  $\Delta_2$ be shellable simplicial spheres. Then, $\P^{\Delta_1} \cong_u \P^{\Delta_2}$ if and only if $\Delta_1$ and $\Delta_2$ are combinatorially equivalent if and only if $G(\Delta_1)$ and $G(\Delta_2)$ are isomorphic.
\end{corollary}

\begin{proof} 
Due to \cite[Theorem 1.2]{Y2024} we only need to show that $\P^{\Delta_1} \cong_u \P^{\Delta_2}$ if and only if $\Delta_1$ and $\Delta_2$ are combinatorially equivalent.

    Let $d_i=\dim \Delta_i$ for $i=1,2$.

If $\Delta_1$ and $\Delta_2$ are combinatorially equivalent, then their facet-ridge graphs  $G(\Delta_1) $ and  $G(\Delta_2)$ are isomorphic. \cref{t:orientablepseudomanifolds,t:matroid_equivalence} hence imply
    $$
        \P^{\Delta_1} \cong_u \P_{G(\Delta_1)} \cong_u \P_{G(\Delta_2)} \cong_u \P^{\Delta_2}.
    $$
    
    Conversely, if $\P^{\Delta_1} \cong_u \P^{\Delta_2}$, then 
    \cref{t:orientablepseudomanifolds}  directly yields that $\P_{G(\Delta_1)} \cong_u \P_{G(\Delta_2)}$. We distinguish the following two cases:
    
    {\sf Case 1:  $d_i\geq 2$ for $i=1,2$.}
    Since the facet-ridge graph of every triangulation of a $d$-sphere is $(d+1)$-connected~\cite[5.Theorem]{Klee},   $G(\Delta_1)$ and $G(\Delta_2)$ are 
    $3$-connected and the result follows from~\cite[Corollary 4.7]{DJK2024} and \cite[Theorem 1.2]{Y2024}.
     
        {\sf Case 2: $d_1 = 1$.} In this case, $\Delta_1$ and, consequently, $G(\Delta_1)$ are cycles of the same length $s$.
        Hence, the graphic matroid of $G(\Delta_1)$ is the uniform matroid of rank $s-1$ on $s$ elements. \cite[Theorem 4.6]{DJK2024} implies that $G(\Delta_1)$ and $G(\Delta_2)$ give rise to isomorphic matroids.
        Since the cycle graph on $s$ edges is the unique graph whose graphic matroid is the uniform matroid of rank $s-1$ on $s$ elements, we conclude that also $G(\Delta_2)$ is an $s$-cycle, and so is $\Delta_2$. The claim follows. In particular, we also have $d_2=1$ and there is no other case to consider.
\end{proof} 

In view of ~\cref{c:distinguishspheres} one might ask whether coarser invariants of the cohomology polytope of a shellable sphere (such as its $f$-vector or $h^\ast$-vector) already determine its unimodular type. This would also mean that the combinatorial type of the sphere would be determined by these invariants. The next example shows that, not much surprisingly, this is not the case.

\begin{example}\label{ex:hstarvectorbad}
    Consider the two simplicial $2$-spheres $\Delta_1$ and $\Delta_2$ defined as follows:
    \begin{align*}
        \Delta_1 =& \tuple{123,125,136,159,168,189, 234,245,347,367,459,478,489,678} \qand \\
        \Delta_2 =& \tuple{123,124,135,145,237,248,279,289,356,367,456,467,478,789}.
    \end{align*}
   On the one hand, a Sage~\cite{sagemath} computation shows that the $h^\ast$-vector of both cohomology polytopes equals
    $$
        (1, 29, 384, 3016, 14955, 45327, 80048, 80048, 45327, 14955, 3016, 384, 29, 1).
    $$
    On the other hand, it is easy to see that $\Delta_1$ and $\Delta_2$ have non-isomorphic facet-ridge graphs (e.g., $G(\Delta_2)$ contains a triangle, whereas $G(\Delta_1)$ does not), see \cref{fig:graphsHstarvectorbad}) and, in particular, $\Delta_1$ and $\Delta_2$ are not combinatorially equivalent. \Cref{c:distinguishspheres} implies that $\P^{\Delta_1}$ and $\P^{\Delta_2}$ are not unimodularly equivalent. This example shows that  the $h^\ast$-vector of the cohomology polytope of a shellable sphere does not determine the polytope up to unimodular equivalence.
    
    \begin{figure}[!h]
        \centering
        \begin{subfigure}{0.4\textwidth}
        \scalebox{0.5}{\begin{tikzpicture}[every node/.style={circle, draw, fill=red!20, minimum size=20pt, inner sep=0pt}]

\node (123) at (0,-0.7) {123};
\node (125) at (-2.3,-1) {125};
\node (245) at (-1.5,-2.5) {245};
\node (234) at (1.5,-2) {234};
\node (459) at (-1.5,0.7) {459};
\node (159) at (-3,1.2) {159};
\node (189) at (-1.5,3) {189};
\node (489) at (-0.3,2.2) {489};
\node (168) at (1,3.5) {168};
\node (678) at (2.5,3.7) {678};
\node (478) at (2,2) {478};
\node (136) at (0.9,1.2) {136};
\node (367) at (3.2,1.4) {367};
\node (347) at (2.4,-0.5) {347};

\draw (123) -- (125);
\draw (123) -- (136);
\draw (123) -- (234);
\draw (125) -- (245);
\draw (245) -- (234);
\draw (459) -- (159);
\draw (459) -- (489);
\draw (159) -- (189);
\draw (189) -- (489);
\draw (189) -- (168);
\draw (489) -- (478);
\draw (168) -- (678);
\draw (168) -- (136);
\draw (678) -- (478);
\draw (678) -- (367);
\draw (478) --(347);
\draw (367) -- (347);
\draw (347) -- (234);
\draw (367) -- (136);
\draw (159) -- (125);
\draw (459) -- (245);

\end{tikzpicture}}    
        \end{subfigure}
        \begin{subfigure}{0.4\textwidth}
        \scalebox{0.5}{\begin{tikzpicture}[every node/.style={circle, draw, fill=red!20, minimum size=20pt, inner sep=0pt}, scale=1]
\node (478) at (3.88,5.19) {478};
\node (789) at (1.00,6.50) {789};
\node (279) at (1.51,4.75) {279};
\node (145) at (9.49,4.56) {145};
\node (356) at (7.75,0.00) {356};
\node (367) at (4.75,0.59) {367};
\node (467) at (4.89,2.90) {467};
\node (123) at (6.59,4.35) {123};
\node (248) at (5.27,7.44) {248};
\node (237) at (3.09,3.24) {237};
\node (289) at (2.27,7.65) {289};
\node (135) at (9.19,2.58) {135};
\node (124) at (7.39,6.91) {124};
\node (456) at (7.94,1.92) {456};
\draw (478) -- (789);
\draw (478) -- (467);
\draw (478) -- (248);
\draw (789) -- (279);
\draw (789) -- (289);
\draw (279) -- (237);
\draw (279) -- (289);
\draw (145) -- (135);
\draw (145) -- (124);
\draw (145) -- (456);
\draw (356) -- (367);
\draw (356) -- (135);
\draw (356) -- (456);
\draw (367) -- (467);
\draw (367) -- (237);
\draw (467) -- (456);
\draw (123) -- (237);
\draw (123) -- (135);
\draw (123) -- (124);
\draw (248) -- (289);
\draw (248) -- (124);
\end{tikzpicture}}
        \end{subfigure}
        \caption{The facet ridge graphs of $\Delta_1$ and $\Delta_2$  from~\cref{ex:hstarvectorbad}}
        \label{fig:graphsHstarvectorbad}
    \end{figure}
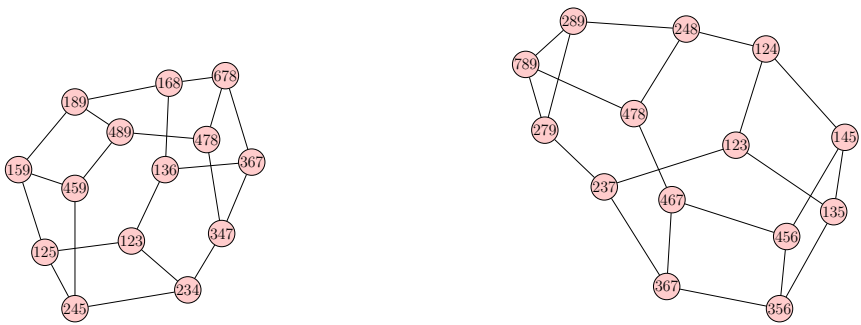
\end{example}

\section{Further examples and directions}
\label{s:questions}

Throughout this paper, we introduced and studied polytopes arising from the boundary and coboundary maps of a simplicial complex $\Delta$. In view of most of our results  (with the exception of~\cref{p:cohomologygraph}) one might get the impression  that  symmetric cohomology polytopes  are more complicated than their homological counterpart. However, the next example shows that both types of polytopes might have nicer properties than their counterpart in some examplea.

\begin{example}{}
    \begin{itemize}
        \item[(1)] Consider the simplicial complexes $\Delta_{\R\PP^2}$ and $\B$ from example \Cref{ex:bjorner}. 
    One can easily verify that neither $\P_{\R\PP^2}$ nor $\P^{\R\PP^2}$ are IDP: For $\P_{\R\PP^2}$, this follows from  \Cref{ex:projectiveplanenotidp}); for $\P^{\R\PP^2}$, one can verify that the $h^\ast$-polynomial differs from the numerator polynomial of the Hilbert series of the toric ideal.
     
    Moreover, whereas the symmetric homology polytope $\P_{\R\PP^2}$ is reflexive by \cref{p:TorsionEvenReflexive}, one can verify that its cohomological counterpart $\P^{\R\PP^2}$ is not reflexive (e.g., by checking that the $h^\ast$-vector is not symmetric). 
So, this is an example, for which the homology polytope is better behaved.
    
    \item[(2)] For the simplicial complex $\B$,  in \cref{ex:bjornernotidp}, we saw that $\P_{\B}$ is reflexive but not IDP.
    In contrast, $\P^{\B}$ is both reflexive and IDP, and further its $h^\ast$-vector equals the Hilbert series of the toric ideal of $\P^{\R\PP^2}$. This provides an example that the cohomology polytope has nicer properties.  
    \end{itemize}
   
\end{example}

If $\Delta$ has vanishing top homology, the only obstruction to the reflexivity of symmetric homology polytopes that we have seen is higher torsion. The next example shows that, for simplicial complexes with non-trivial top homology, even simplicial complexes with very good topological (and combinatorial) properties can fail to be reflexive.

\begin{example}[\emph{Large complexes failing reflexivity}]\label{ex:largecomplex}
    Let $\Delta$ be the $3$-skeleton of the $6$-dimensional simplex, i.e., 
    $$
        \Delta = \tuple{A \st A \subseteq [7];\; |A|=4}.
    $$
    One can easily verify that $\Delta$ does not have torsion. Hence there is no obvious obstruction for  $\P_\Delta$ to be non-reflexive. However, computing the $h^\ast$-polynomial we get
    \begin{multline*}
        1 + 50t + 1240t^2 + 19930t^3 + 226115t^4 + 1863268t^5 + 11080490t^6 + {\color{red} 46179088t^7} + \\ {\color{red} 130616924t^8 + 245464888t^9} + 303300596t^{10} + {\color{red} 245633728t^{11} + 130718060t^{12} + 46194712t^{13}} + \\11080490t^{14} +  1863268t^{15} + 226115t^{16} + 19930t^{17} + 1240t^{18} + 50t^{19} + t^{20},
    \end{multline*}
    which is not symmetric. This shows that $\P_\Delta$ is not reflexive.
\end{example}

As a natural extension of \Cref{p:smallerproblem} one can ask the following:
\begin{question}\label{q:reflexiveIFF}
    Let $\Delta$ be a simplicial complex such that $\P_\Delta$ is reflexive. Is it true that $\P_{\Gamma}$ is reflexive for any subcomplex $\Gamma$ with $\dim\Delta=\dim\Gamma$?
\end{question}
We note that \cref{p:smallerproblem} already tells us that the converse is true.
Though we could imagine that arbitrary subcomplexes might have torsion and, hence, the corresponding symmetric homology polytopes might not be reflexive, we could not find such an example. A positive answer to \Cref{q:reflexiveIFF} would allow to check for reflexivity without knowing the facets of $\P_\Delta$.

It is a wide open conjecture \cite[Conjecture 5.11]{gammaPositivity} that  $h^\ast$ polynomials of classical symmetric edge polytopes are $\gamma$-positive. It was shown in \cite[Section 4]{DHO2024} that this is false if one goes to the more general setting of generalized symmetric edge polytopes. The counterexample provided therein comes from the dual matroid of the graphic matroid of a non-planar graph. In particular, the underlying matroid is not graphic. As its circuits are of size $3$, one can conclude that it is not possible to realize this  matroid  by the  boundary map of some simplicial complex. Extensive computations yield the following question:

\begin{question}
    Is the $h^\ast$-polynomial of any reflexive symmetric (co)homology polytope $\gamma$-positive? 
\end{question}
In order to attack this question, the first step is to better understand the $h^\ast$-vector of these polytopes. Possible approaches might go via  a Gr\"obner basis (see \Cref{t:groebner_homology}) or the attempt to generalize the description of the $h^\ast$-vectors for symmetric edge polytopes of graphs from \cite{KT2023} to the more general setting of simplicial complexes.


\
Another line of further research naturally arises from the following connections: We have seen in \Cref{c:facettrees}, that facets of $\P_\Delta$ correspond to subcomplexes of $\Delta$ that contain a simplicial spanning forest of $\Delta$. In view of the similarities between the tools we use (such as the SNF of the boundary  map) and the tools used in the study of critical groups of simplicial complexes (see e.g., \cite{DKM2009}) one might wonder whether there exists a relationship between critical groups of simplicial complexes and symmetric (co)homology polytopes. In~\cite{DKM2009} the authors show that the order of a given critical group of a simplicial complex $\Delta$ is equal to a  torsion-weighted enumeration of simplicial spanning trees of $\Delta$. 
Since in the case that $\Delta$ admits a simplicial spanning tree, facets of $\P_\Delta$ correspond to subcomplexes containing simplicial  spanning trees, and  higher torsion is an obstruction to reflexivity (see \Cref{p:torsionNOTreflexive}), it is conceivable that there exists a criterion in terms of the size of the (last) critical group of $\Delta$ and the number of facets of $\P_\Delta$ showing that the symmetric homology polytope $\P_\Delta$ is not reflexive.

\begin{Acknowledgments}
We would like to thank Matthias Beck for several discussions concerning the contents of this paper. We also are grateful to Bernd Sturmfels for clearing up some confustion which we had when working on the section on Gr\"obner bases.
This work was supported by the German Research Foundation (DFG) via SPP 2458 \textit{Combinatorial Synergies}, project number JU 3097/6-1.
\end{Acknowledgments}

    \bibliographystyle{plain}
    \bibliography{bibliography.bib}
\end{document}